\documentclass[phd,tocprelim,final]{cornell}

\renewcommand{\caption}[1]{\singlespacing\hangcaption{#1}\normalspacing}

\usepackage{mathtools}
\usepackage{amssymb}

\DeclareTextFontCommand{\emph}{\bfseries}

\usepackage{amsthm}
\makeatletter
\def\thm@space@setup{%
  \thm@preskip=\topsep \thm@postskip=-0.5\thm@preskip
}
\makeatother

\makeatletter
\renewenvironment{proof}[1][\proofname]{\par
  \vspace{-0.5\topsep}
  \pushQED{\qed}%
  \normalfont
  \topsep0pt \partopsep0pt 
  \trivlist
  \item[\hskip\labelsep
        \itshape
    #1\@addpunct{.}]\ignorespaces
}{%
  \popQED\endtrivlist\@endpefalse
}
\makeatother

\usepackage{bm}

\newcommand{\bR}{\mathbb{R}}

\newcommand{\bN}{\mathbb{N}}
\newcommand{\bZ}{\mathbb{Z}}

\newcommand{\bx}{\mathbf{x}}

\newcommand{\bb}{\mathbf{b}}
\newcommand{\ba}{\mathbf{a}}
\newcommand{\bO}{\mathbf{0}}

\newcommand{\inv}[1]{#1^{-1}}

\DeclareMathOperator{\rank}{rank}

\newcommand{\sC}{\mathcal{C}}

\newcommand{\abs}[1]{\lvert#1\rvert}

\newcommand{\covers}{\gtrdot}
\newcommand{\defeq}{\coloneqq}
\renewcommand{\subset}{\subseteq}

\newcommand{\vast}{\bBigg@{4}}
\newcommand{\Vast}{\bBigg@{5}}

\usepackage[mode=image]{standalone}
\usepackage{newpxtext, newpxmath}

\usepackage{subcaption}

\usepackage{xcolor}

\usepackage{enumitem} 
\setlist{topsep=0pt}
\usepackage{tabstackengine}
\setstacktabbedgap{1.5ex}
\setstackgap{L}{19.2pt}
\fixTABwidth{T}
\newcommand{\matrixspacingrule}{\rule[-35pt]{1pt}{76.5pt}}

\newcommand{\sxrightarrow}[2][]{%
  \mathrel{\text{$\xrightarrow[#1]{#2}$}}%
}

\newcommand{\toc}[2]{\sxrightarrow{c[#1, #2]}} 
\newcommand{\tor}[2]{\sxrightarrow{r[#1, #2]}} 

\usepackage{float}
\usepackage{hyperref}
\usepackage[capitalize, noabbrev, nameinlink]{cleveref}

\theoremstyle{plain}
\newtheorem{lemma}{Lemma}[section]
\newtheorem{theorem}[lemma]{Theorem}
\newtheorem{proposition}[lemma]{Proposition}

\theoremstyle{definition}
\newtheorem{definition}[lemma]{Definition}
\newtheorem{example}[lemma]{Example}
\AtBeginEnvironment{example}{%
  \pushQED{\qed}%
}
\AtEndEnvironment{example}{\popQED\endexample}

\theoremstyle{remark}
\newtheorem{remark}[lemma]{Remark}

\usepackage{savesym}
\savesymbol{autoref}

\newcommand{\autoref}[1]{\cref{#1}}
\usepackage{caption}

\usepackage{algorithm}
\usepackage[beginComment=//~]{algpseudocodex}
\makeatletter
\patchcmd{\ALG@step}{\addtocounter{ALG@line}{1}}{\refstepcounter{ALG@line}}{}{}
\newcommand{\ALG@lineautorefname}{Line}
\makeatother
\makeatletter
\newcounter{HALG@line}
\renewcommand{\theHALG@line}{\thealgorithm.\arabic{ALG@line}}
\makeatother

\algnewcommand\algorithmicinput{\textbf{Input:}}
\algnewcommand\Input{\item[\algorithmicinput]}

\algrenewcommand\Output{\item[\textbf{Output:}]}

\DeclareMathOperator{\SSYT}{SSYT}
\DeclareMathOperator{\SYT}{SYT}
\DeclareMathOperator{\WCT}{WCT}
\DeclareMathOperator{\CT}{CT}
\DeclareMathOperator{\Des}{Des}
\DeclareMathOperator{\flatt}{flat}
\DeclareMathOperator{\PD}{PD}

\DeclareMathOperator{\Dem}{Dem}
\DeclareMathOperator{\len}{len}
\DeclareMathOperator{\del}{del}
\DeclareMathOperator{\link}{link}
\DeclareMathOperator{\comp}{comp}
\DeclareMathOperator{\RW}{RW}

\DeclareMathOperator{\init}{in}
\DeclareMathOperator{\mult}{mult}
\DeclareMathOperator{\Ess}{Ess}

\DeclareMathOperator{\wt}{wt}
\DeclareMathOperator{\QPD}{QPD}

\DeclareMathOperator{\excess}{excess} 

\newcommand{\dotsw}{\dotsm} 
\newcommand{\bk}{\mathbb{k}}
\newcommand{\forest}{\mathfrak{P}}
\newcommand{\sG}{\mathcal{G}}
\newcommand{\fp}{\mathfrak{p}}
\newcommand{\sS}{\mathcal{S}}

\newcommand{\wl}[3]{L_{#1}^{#2, #3}}

\newcommand{\markdown}[1]{\overset{AAAAA\downarrow}{#1}}
\DeclareRobustCommand{\markdown}[1]{\overset{\scriptscriptstyle\downarrow}{#1}} \newcommand{\markup}[1]{NOT USED\hat{#1}}
\DeclareRobustCommand{\markup}[1]{\overset{\scriptscriptstyle\uparrow}{#1}}

\newcommand{\downmark}{``$\downarrow$''}
\newcommand{\upmark}{``$\uparrow$''} 

\DeclareRobustCommand{\downmark}{``$\downarrow$''}
\DeclareRobustCommand{\upmark}{``$\uparrow$''}

\newcommand{\monkdown}[1]{\markdown{#1}}
\newcommand{\monkup}[1]{\markup{#1}}

\newcommand{\bone}{\mathbf{1}}
\newcommand{\sH}{\mathcal{H}}
\newcommand{\sK}{\mathcal{K}}
\newcommand{\sT}{\mathcal{T}}
\newcommand{\slide}{\mathcal{F}}
\newcommand{\glide}{\mathcal{G}}

\newcommand{\overbar}[1]{\mkern 1.5mu\overline{\mkern-1.5mu#1\mkern-1.5mu}\mkern 1.5mu} 
\newcommand{\st}{:}
\newcommand{\fc}{\colon} 
\newcommand{\dash}{{-}}
\newcommand{\dd}{\dash}
\renewcommand{\subset}{\subseteq}
\renewcommand{\supset}{\supseteq} 
\newcommand{\Schub}{\mathfrak{S}}
\newcommand{\schur}{s}
\newcommand{\bSchub}{\overleftarrow{\Schub}}
\newcommand{\Groth}{\mathfrak{G}}
\newcommand{\Leh}{L}
\newcommand{\ssc}[3]{c_{#1,#2}^{#3}}
\newcommand{\bsc}[3]{\overleftarrow{c}_{#1,#2}^{#3}}
\newcommand{\size}{\abs}
\newcommand{\setsize}[1]{\##1}
\newcommand{\leftword}[1]{\dot{#1}}

\newcommand{\compred}[1]{\mathbf{#1}}

\newcommand{\backR}{\overleftarrow{R}}
\newcommand{\mSchub}[1]{\overbar{X}_{#1}}
\newcommand{\comptoset}[1]{S_{#1}}

\newcommand{\quasi}{F}

\newcommand{\SRring}[1]{\mathit{SR}_{#1}}
\DeclareMathOperator{\Spec}{Spec}

\newcommand{\interval}[1]{[#1]}
\newcommand{\setbuilder}[2]{\{\,#1\st#2\,\}}
\newcommand{\groupbuilder}[2]{\langle\,#1 \mid #2\,\rangle}
\newcommand{\idealbuilder}[2]{\langle\,#1 \st #2 \,\rangle}

\newcommand{\tabcomplex}[1]{\Delta(#1)}
\newcommand{\tabcomplexE}[2]{\Delta_{#2}(#1)}

\newcommand{\RPI}{R_1}
\newcommand{\RPII}{R_2}
\newcommand{\RPIII}{R_3}

\newcommand{\forestequiv}{\equiv}
\newcommand{\forestsim}{\sim}

\newcommand{\bs}{backwards saturated}

\newcommand{\sqbl}{[}
\newcommand{\sqbr}{]}

\def\qY/{quasi-Yamanouchi}
\def\QY/{Quasi-Yamanouchi}

\renewcommand{\caption}[1]{\singlespacing\hangcaption{#1}\normalspacing}

\title{Some results related to the slide decomposition of Schubert polynomials}
\author{Karl Thomas Baath Sjoblom}
\conferraldate {August}{2024}
\degreefield {Ph.D.}
\copyrightholder{Karl Thomas Baath Sjoblom}
\copyrightyear{2024}

\begin{document}

\maketitle
\makecopyright

\begin{abstract}
The expansion of a Schubert polynomial into slide polynomials corresponds to a sum over sub-balls in the subword complex.
There has been recent interest in other, coarser, expansions of Schubert polynomials. We extend the methods used in \cite{Subword-complexes-paper} to prove that the subword complex is a ball or a sphere to a more general method, and use it to prove that the expansion of a Schubert polynomial into forest polynomials also corresponds to a sum over sub-balls in the subword complex.

When expanding the product $\Schub_\pi\Schub_\rho$ of two Schubert polynomials into Schubert polynomials $\Schub_\sigma$, there is a bijection between shuffles of reduced words for $\pi$ and $\rho$ and reduced words for $\sigma$ (counted with multiplicity). We give such a bijection for Monk's rule and Sottile's Pieri rule.

We give tableau-based definitions of slide polynomials, glide polynomials, and fundamental quasisymmetric polynomials and show that the expansion of a Schur polynomial into fundamental quasisymmetric polynomials corresponds to a sum over sub-balls in the tableau complex.

The Schubert polynomials are the cohomology classes of matrix Schubert varieties, but there is no geometric explanation of the slide polynomials. We show that the slide polynomials are not (antidiagonal) Gr\"obner degenerations of matrix Schubert varieties, answering in the negative a question of \cite[Section 1.4]{Smirnov-Tutubalina}. 
\end{abstract}

\begin{biosketch}
Thomas B\aa{}\aa{}th Sj\"oblom 
completed  a Bachelor's degree in applied mathematics and a Master's degree in computer science at Chalmers University of Technology in 2013. He also took a number of classes in pure mathematics at the University of Gothenburg. After that, he started the Ph.D. program in mathematics at Cornell University. 
He was writing his dissertation when the COVID-19 pandemic broke out, which slowed things down.
In 2021 he moved to Germany, which further slowed the progress down. While in Germany he worked part time as a data scientist while finishing up his dissertation, and in 2024, he finally finished it.
Outside of mathematics, he enjoys napping and going to the gym.
\end{biosketch}

\begin{dedication}
Till farmor och morfar. Saknar er.
\end{dedication}

\begin{acknowledgements}
First, I would like to thank my advisor, Allen Knutson, 
for being approachable and helpful, for teaching me a lot of fun math 
and  
for being very good at figuring out when (not) to tell me about geometry.

Second, I would like to thank my committee, Mike Stillmann and Ed Swartz (and Lou Billera, who retired before I managed to finish),
for their helpful suggestions and their patience.
 I \textit{acknowledge} that I should 
have interacted with them more.

Third, I would like to thank the math department. 
In particular, I would like to thank all my academic siblings for making my time at Cornell way more fun than it otherwise would have been. 
I would also like to thank Melissa for putting up with years of silly ideas, and Elly for helping me with all the administrative stuff needed to come back and graduate.

Fourth, I would like to thank Teagle Down.
It was the first place started going to regularly outside the math department and it turned out to be a very friendly and welcoming place and the locker rooms were a great place to sit and think about math, or about your life, while waiting for the sweat to dry. Thank you, Ben, Tom, and everyone else for making it so.

Fifth, I would like to thank Ithaca and the people there.
Making friends with people outside the math department and the university was really helpful, especially in the later years, when people were starting to graduate and move away.
Among others, thank you, Meghan, Augusto, and Brian.

Sixth, I would like to thank the Thanks to Scandinavia Foundation for supporting me during my first year at Cornell.

Finally, I would like to thank Joy for all her love and support, and Ser for being a cat!
\end{acknowledgements}

\contentspage
\tablelistpage
\figurelistpage

\normalspacing \setcounter{page}{1} \pagenumbering{arabic}
\pagestyle{cornell} \addtolength{\parskip}{0.5\baselineskip}

\chapter{Background}
In this Chapter, we give a short background to the problems we solve, then 
we recall the necessary background and set out our notation for things. Most of the content is fairly standard and we cite the less standard content. In \autoref{sec:tableaux}, we provide three definitions of our own, which we use in Sections \ref{sec:fundamental-quasisymmetric} and \ref{sec:slide}.

\textit{Schubert polynomials} were introduced by Lascoux and Sch\"utzenberger in order to compute in the cohomology ring of flag varieties \cite{Lascoux-Schutzenberger}.
Fulton proved that they represent the cohomology classes of \textit{matrix Schubert varieties} \cite{Fulton-92}. They form a basis for the polynomial ring and for geometric reasons, it is known that the product of two Schubert polynomials expands positively in this basis. Finding a combinatorial rule for computing this product is a major open problem in Schubert calculus.
 
\textit{Subword complexes} were introduced by Knutson and Miller \cite{Grobner-geometry-paper} and they proved that they are balls or spheres \cite{Subword-complexes-paper}. They include \textit{Gr\"obner degenerations} of matrix Schubert varieties and therefore give interpretations of formulas for Schubert and \textit{Grothendieck polynomials} in terms of simplicial complexes.

Assaf and Searles introduced the \textit{slide polynomials} \cite{Slide-paper}. They indexed them by weak compositions, found a combinatorial multiplication rule, and showed how Schubert polynomials expand in terms of them. When one indexes them by reduced words, the multiplication rule they found no longer works and in this setting, a combinatorial multiplication rule would give a rule for multiplying Schubert polynomials. 
Pechenik and Searles introduced \textit{glide polynomials}, which extend the slide polynomials to $K$-theory \cite{Glide-paper}.

Smirnov and Tutubalina defined \textit{slide complexes}, corresponding to slide polynomials, and showed that they decompose the subword complex into balls \cite{Smirnov-Tutubalina}. Therefore, the expansion of a Schubert polynomial into slide polynomials corresponds to a sum over these balls. They asked whether there is a partial Gr\"obner degeneration of the corresponding matrix Schubert variety that gives this decomposition. In \autoref{sec:slides-not-degen-of-schuberts}, we show that that is not the case.

In \autoref{sec:backwards-saturated-balls}, we extend the proof by Knutson and Miller, which showed that the subword complex is a ball, to show that a large class of subcomplexes of the subword complex are balls. This class inclues the sets of reduced words corresponding to the \textit{forest polynomials} of Nadeau and Tewari \cite{Forest-polynomial-paper}.

\textit{Tableau complexes} were introduced by Knutson, Miller, and Yong who also showed that they are balls or spheres \cite{Tableau-complexes-paper}. They give an interpretation of a tableau-based formula for \textit{vexillary double Grothendieck polynomials} (which specialize to \textit{Schur polynomials}) as the $K$-polynomials of the \textit{Stanley--Reisner rings} of simplicial complexes. 
In \autoref{sec:decomposing-tableau-complexes}, we show that, in analogy with the slide expansion of a Schubert polynomial, the expansion of a Schur polynomial into \textit{fundamental quasisymmetric polynomials} corresponds to a decomposition into balls of the tableau complex of \textit{semistandard Young tableaux}.

As mentioned above, the multiplication rule of Assaf and Searles only works when the slide polynomials are indexed by weak compositions. 
Nenashev showed that there should exist a ``shuffle rectification'' rule for reduced words \cite{Nenashev}, and we found such a rule for \textit{Monk's rule} (\autoref{sec:monk's-rule}) and for \textit{Sottile's Pieri rule} (\autoref{sec:pieri's-rule}), however, these rules do not agree with the multiplication rule of Assaf and Searles. Since our rules work on the level of words, they do give a positive rule for multiplying Schubert polynomials in these two cases. Other positive rules for computing the products of Schubert polynomials in these two cases as well as some others exist, for example in terms of \textit{pipe dreams} \cite{RC-graphs-and-schubert-polynomials}, \cite{Kogan-Kumar}.

\section{Permutations and related concepts}
Here we recall some basic facts about permutations and set up our notation. In particular, we specify which of the many different conventions we use for wiring diagrams.
\subsection{Basic definitions and conventions}
The symmetric group $S_n$ is the group of all permutations $\pi$ of the set $\interval{n} \defeq \{1, \dotsc, n\}$.
We will usually write permutations in \emph{one-line notation}: $\pi = [\pi(1)\dotsm\pi(n)]$, so $[312]$ is the permutation that takes $1$ to $3$, $2$ to $1$ and $3$ to $2$. The \emph{cycle} $(i_1,\dotsc,i_k)$ is the permutation that takes $i_j$ to $i_{j+1}$ for $1 \le j < k$ and $i_k$ to $i_1$ and leaves all other numbers fixed. Of special interest among the cycles are the \emph{transpositions} $t_{i, j} \defeq (i, j)$, and of special interest among the transpositions are the \emph{adjacent transpositions} $s_i \defeq t_{i, i+1}$. The adjacent transpositions for $i \in \interval{n-1}$ generate $S_n$, and they satisfy the following relations.
\begin{itemize}
\item $s_is_j = s_js_i$ whenever $\abs{i-j} \ge 2$,
\item $s_is_{i+1}s_i = s_{i+1}s_is_{i+1}$ (the braid relation),
\item $s_i^2 = 1$.
\end{itemize}
All other relations between the $s_i$ follow from these.

If $\pi = s_{i_1} \dotsm s_{i_k}$, then the sequence $i_1 \dotsw i_k$ is a \emph{word} of \emph{length} $k$ for $\pi$ and the $i_j$ are called \emph{letters}. Given a word $w = i_1 \dotsw i_k$, we let $\prod w \defeq s_{i_1} \dotsm s_{i_k}$.
If $w$ is a word of length $k$ for $\pi$ and there are no words for $\pi$ of length less than $k$, then $w$ is a \emph{reduced word} for $\pi$, and the \emph{length} $\len(\pi)$ of $\pi$ is $k$. We will denote the set of reduced words for $\pi$ by $\RW(\pi)$. 
We will abuse notation and say that $w$ is a word in $S_n$ when $w$ is the word for a permutation $\pi \in S_n$. When $n$ does not matter, we will simply say that $w$ is a word. 
If $i < j$ and $\pi(i) > \pi(j)$, then $(i, j)$ is an \emph{inversion} of $\pi$. The length $\len(\pi)$ is equal to the number of inversion of $\pi$. 
\begin{example} \label{ex:triangular-word}
  The permutation $\pi = [n\,(n-1)\,\dotsm\,1] \in S_n$ is the longest permutation in $S_n$. It has length $n(n-1)/2$ since every pair $i < j$ is an inversion. 
One word for $\pi$ of particular interest to us is the \emph{triangular word} on $n-1$ letters:
\[
(n-1)(n-2) \dotsm 21 \; (n-1)(n-2) \dotsm 2\; \dotsm\; (n-1)(n-2)\;(n-1).\qedhere
\]
\end{example}
The above discussion (except \autoref{ex:triangular-word}) also applies to the group $S_\infty$ consisting of permutations of the positive integers that fix all but finitely many integers, and to the group $S_\bZ$ consisting of permutations of the integers that fix all but finitely many integers. These groups are generated by the adjacent transpositions $s_i$ for $i \ge 1$ and $i \in \bZ$, respectively.

Given a word $w = i_1\dotsm i_m$ in $S_n$ or $S_\infty$, a \emph{compatible sequence} for $w$ is a sequence of positive integers $j_1 \dotsm j_m$ such that:
\begin{itemize}
\item $j_k \le i_k$ for $k \in \interval{m}$,
\item $j_k \le j_{k+1}$ for $k \in \interval{m-1}$,
\item if $i_k < i_{k+1}$, then $j_k < j_{k+1}$.
\end{itemize} 
For words in $S_\bZ$, we allow non-positive integers in the compatible sequences. 
\begin{example}\label{ex:compatible-sequences}
  The compatible sequences for the word $w = 21434$ in $S_n$ (or $S_\infty$) are $11223$, $11224$, $11234$ and $11334$.  
\end{example}

Given a (not necessarily reduced) word $Q$ of length $k$, an ordered subsequence of $P$ of $Q$ is called a \emph{subword} of $Q$. Note that each subword comes with an embedding into $Q$. That is, a single word can have multiple embeddings into $Q$ and these are \textit{different as subwords}.  
To denote a subword (with its embedding), we will write a dash for each letter in $Q$ that is not in the subword. We will sometimes abuse the notation and treat a subword $P$ as if it is a word (that is, by ignoring any dashes). 
A subword $P$ \emph{represents} a permutation $\pi$ if $P$ (seen as a word) is a reduced word for $\pi$, and $P$ \emph{contains} $\pi$ if some subword of $P$ (seen as a word) represents $\pi$.
\begin{example}
  The word $Q = 121$ has subwords $121$, $\dash21$, $1\dash1$, $12\dash$, $\dash\dash1$, $\dash2\dash$, $1\dash\dash$ and $\dash\dash\dash$.
\end{example}

A \emph{Coxeter group} on $n$ generators is a group 
\[
G = \groupbuilder{r_1, \dotsc, r_n}{\text{$(r_ir_j)^{m_{ij}} = 1$ if $m_{ij} \neq \infty$}},
\] 
where $m_{ii} = 1$ for all $i$ and for $i \neq j$, $m_{ij} = m_{ji}$ and $m_{ij} \in \{2, 3, \dotsc, \infty\}$.
That is, $r_i^2 = 1$ for all $i$. When $m_{ij} = 2$, $r_i$ and $r_j$ commute. When $m_{ij} = 3$, $r_i$ and $r_j$ satisfy the braid relation. 
Elements of the form $g r \inv g$, where $r$ is a generator and $g$ is any element of $G$, are called \emph{reflections}.
The definitions of \emph{word} and \emph{subword} extend to this setting.

The groups $S_n$, $S_\infty$ and $S_\bZ$ are Coxeter groups:
$r_i = s_i$, 
$m_{ij} = 3$ if $\abs{i-j} = 1$ and $m_{ij} = 2$ if $\abs{i-j} \ge 2$.  The reflections are the transpositions $t_{ab}$.
The universal Coxeter group $W_n$ on $n$ generators shows up briefly in \autoref{sec:subword-complexes}.
In this group, $m_{ij} = \infty$ for all $i \neq j$. That is, the only relations are $r_i^2 = 1$ for $i \in \interval{n}$.
A word in $W_n$ is reduced if and only if it has no repeated adjacent letters, and there is exactly one reduced word for each element, so we will identify the reduced words with the elements.

\begin{definition}
  Let $G$ be a Coxeter group. The \emph{(strong) Bruhat order} on $G$ is the partial order on $G$ generated by $\pi\cdot t > \pi$, for all reflections $t$ with $\len(\pi\cdot t) > \len(\pi)$.
If $\len(\pi \cdot t) = \len(\pi) + 1$, we say that $\pi \cdot t$ \emph{covers} $\pi$ and write $\pi\cdot t \covers \pi$. 
\end{definition}
The \emph{Demazure product} (cf. \cite[Definition 3.1]{Subword-complexes-paper}) of a word $w$, $\Dem(w)$,  is defined inductively by
\begin{itemize}
\item the Demazure product of the empty word is the identity permutation, and
\item $\Dem(wi) = 
  \begin{cases}
    \Dem(w) \cdot s_i & \text{if $\Dem(w) \cdot s_i > \Dem(w)$,}\\
    \Dem(w) & \text{otherwise.}
  \end{cases}
$
\end{itemize}
If $w$ is reduced, the Demazure product is equal to the product of the word.

\subsection{Wiring diagrams} 
Given a word $w = w_1 \dotsm w_n$ in $S_n$, a \emph{wiring diagram} for $w$ (see \autoref{fig:wiring-diagram-example} for an example) consists of $n+1$ columns (numbered from $0$ to $n$, starting on the left). Each column contains a permutation of the numbers $1$, \ldots, $k$ as follows. 
\begin{itemize}
\item In the rightmost column, the numbers are in order, with $1$ at the bottom.
\item The $(i-1)$st column is equal to the $i$th column, except the $w_i$th and the $(w_i+1)$st numbers (from the bottom) have been swapped.
\end{itemize}
Thus, the numbers in the $i$th column form the permutation $\prod w_n \dotsm w_{i+1}$, so in the $0$th column, they form the permutation $\inv \pi$.
For each $j$, there are straight lines connecting the $j$ in column $i$ to the $j$ in column $i+1$ and column $i-1$. 
The lines connecting all the $j$s are a \emph{wire}, the \emph{$j$-wire}, and the \emph{label} on it is $j$. 
In order to match the terminology for words, we call the space \textit{between} column $i - 1$ and column $i$ \emph{position $i$}. At position $i$, the $w_i$th and the $(w_i + 1)$st wires from the bottom cross, we will say that there is a \emph{cross} at \emph{height $w_i$} at position $i$. That is, the positions are to the left of the columns and the heights are above the numbers.
Given a word $w$, let $\wl{w}{i}{j}$ be the $j$th label from the bottom in the $i$th column in the wiring diagram for $w$. This notation will be useful in \autoref{sec:multiplying-slides}, when we want to analyse what the wiring diagram looks like when we change the word letters in $w$.
The cross in column $i$ is an \emph{$(a,b)$-cross} if $\wl{w}{i}{w_i} = a$ and $\wl{w}{i}{w_i+1} = b$. That is, if the cross moves the $a$-wire up and the $b$-wire down when following them from right to left.  
Reading the labels on the crosses from right to left give a factorization of the word. If the word is reduced, the factors in the product correspond to the inversions of the permutation, see \autoref{fig:wiring-diagram-example}. 
A word is reduced if and only if no pair of wires cross twice, or equivalently, if and only if all crosses are labeled $(a,b)$ with $a < b$. 

\begin{figure}
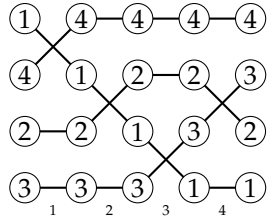
\centering
\includestandalone{wiring-diagram-ex}
\caption[The wiring diagram for the word $3212$ for $\sqbl4213\sqbr$.]{The wiring diagram for the word $3212$ for the permutation $[4213]$. Reading the labels on the crosses from right to left gives the factorization $[4213] = t_{23}t_{13}t_{12}t_{14}$.}
\label{fig:wiring-diagram-example}
\end{figure}

If we add an empty column and insert an $(a,b)$-cross in it, we obtain a wiring diagram word for $\pi t_{ab}$. This fact is the main reason we place the identity permutation on the right side of the wiring diagram instead of on the left.

\begin{lemma}[\protect{\cite[Lemma 2.3]{Billey-Holroyd-Young-Little-bump-paper}}]\label{lemma:defect-lemma}
  If $w = w_1\dotsb w_n$ is not reduced, but the word $w_1 \dotsb \widehat{w}_t \dotsb w_n$ is reduced, then there exists exactly one $t' \neq t$ such that $w_1 \dotsb \widehat{w}_{t'} \dotsb w_n$ is reduced. We say that $t$ and $t'$ are the \emph{defects} of $w$.
\end{lemma}
In terms of wiring diagrams, this says that there are exactly two wires that cross twice.

\section{Combinatorial objects}
In this section, we introduce the combinatorial objects we are considering. That is, objects that have some kind of structure that make them interesting in their own right. Note in particular that we use $\lambda$ to denote partitions as well as compositions and weak compositions.
\subsection{Partitions, compositions and weak compositions}\label{sec:compositions}
We are interested in three kinds of sequences of integers that add up to $n$. In order of decreasing specificity:
\begin{itemize}
\item A \emph{partition} $\lambda = (\lambda_1, \dotsc, \lambda_k)$ of $n$ is a \textit{weakly decreasing} sequence of \textit{positive integers} that add up to $n$.
\item A \emph{composition} $\lambda = (\lambda_1, \dotsc, \lambda_k)$ of $n$ is a sequence of \textit{positive integers} that add up to $n$. 
\item A \emph{weak composition} $\lambda = (\lambda_1, \dotsc, \lambda_k)$ of $n$ is a sequence of \textit{non-negative integers} that add up to $n$. 
\end{itemize}
We write $\size{\lambda} \defeq n$ and $x^\lambda \defeq x_1^{\lambda_1}\dotsm x_k^{\lambda_k}$. 
We will generally display these using Young diagrams (using the English convention).
Given $\lambda = (\lambda_1, \dotsc, \lambda_k)$, the \emph{Young diagram} for $\lambda$ is a diagram with $\lambda_1$ left justified boxes in the top row, $\lambda_2$ in the next and so on. The boxes are indexed using matrix coordinates. We identify $\lambda$ and its Young diagram. A Young diagram for a partition, a composition and a weak composition are shown in \autoref{fig:young-diagram-partition-compositions}. 

\begin{figure}
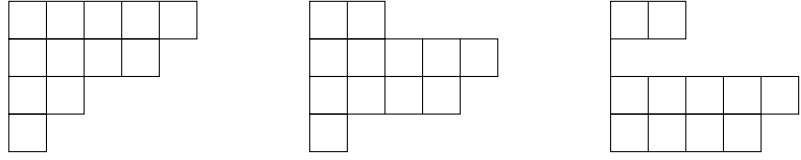

  \centering
  \includestandalone{young-diagrams-ex}
  \caption[Three Young diagrams.]{The Young diagrams for the partition $(5,4,2,1)$, the composition $(2,5,4,1)$, and the weak composition $(2,0,5,4)$.} 
  \label{fig:young-diagram-partition-compositions}
\end{figure}

The \emph{Lehmer code} of a permutation $\pi \in S_n$ is the weak composition $\Leh(\pi) = (\Leh(\pi)_1, \dotsc, \Leh(\pi)_n)$, where $\Leh(\pi)_i = \#\{\,j > i  \st \pi(j) < \pi(i)\,\}$. 
The map taking $\pi$ to $\Leh(\pi)$ is a bijection between permutations of length $k$ in $S_\infty$ and weak compositions of $k$.

There is a bijection between compositions $\lambda = (\lambda_1, \dotsc, \lambda_k)$ of $n$ and subsets $S = \{i_1 < \dotsb < i_{k-1}\}$ of $\interval{n-1}$ (see \cite[Section 7.19]{Stanley-2}) given by
\begin{align*}
  (\lambda_1, \dotsc, \lambda_k) &\mapsto \{\lambda_1, \lambda_1 + \lambda_2, \dotsc, \lambda_1 + \dotsb + \lambda_{k-1} \},
\\
\{i_1 < \dotsb < i_{k-1}\} &\mapsto (i_1, i_2 - i_1, i_3 - i_2, \dotsc, i_{k-1}-i_{k-2}, n - i_{k-1}).
\end{align*}
Let $\comptoset{\lambda}$ denote the set corresponding to the composition $\lambda$ and $\comp(S)$ the composition corresponding to the set $S$.

We end this section with some less standard definitions.
The following three are from \cite{Slide-paper}.
Given a weak composition $\lambda$, the \emph{flattening} $\flatt(\lambda)$ of $\lambda$ is the composition obtained by removing all $0$s from $\lambda$. 
For (weak) compositions $\lambda$ and $\mu$, $\lambda$ \emph{dominates} $\mu$, $\lambda \ge \mu$ if $\lambda_1 + \dotsb + \lambda_i \ge \mu_1 + \dotsb + \mu_i$ for all $i$. 
Let $\lambda$ and $\mu$ be compositions, $\lambda$ \emph{refines} $\mu$ if there are $0 = i_0 < i_1 < \dotsb < i_l$ such that $\lambda_{i_{j-1} + 1} + \dotsb + \lambda_{i_{j}} = \mu_{j}$ for $j = 1$, \ldots, $l$ (or equivalently, if $\comptoset{\lambda} \subseteq \comptoset{\mu} \subseteq \interval{n-1}$).

The following three are from \cite{Glide-paper}. 
A \emph{weak komposition} is a weak composition where some non-zero entries are bold. 
The \emph{excess} is the number of bold entries.
\begin{definition}[\protect{\cite[Definition 2.2]{Glide-paper}}]\label{def:glide-of-composition}
  Let $\lambda$ be a weak composition with $l$ non-zero entries at $n_1 < \dotsb < n_l$. The weak komposition $\kappa$ is a \emph{glide} of $\lambda$ if there are $0 = i_0 < i_1 < \dotsb < i_l$, such that for each $j \in \interval{l}$,
  \begin{itemize}
  \item $\kappa_{i_{j-1} + 1} + \dotsb + \kappa_{i_j} = \lambda_{n_j} + \excess(\kappa_{i_{j-1} + 1}, \dotsc, \kappa_{i_j})$,
  \item $i_j \le n_j$,
  \item the first non-zero number among $\kappa_{i_{j-1} + 1}$, \ldots, $\kappa_{i_j}$ is not bold.
  \end{itemize}
\end{definition}
\begin{example}\label{ex:glide-of-composition}
  Let $\lambda = (0, 1, 0, 0, 0, 3)$, so that $n_1 = 2$, $n_2 = 6$ and $\lambda_{n_1} = 1$, $\lambda_{n_2} = 3$. 
Taking $i_1 = 1$, $i_2 = 6$ shows that $(1, 0, 1, 0, \compred{1}, \compred{3})$ is a glide of $\lambda$, and $i_1 = 2$, $i_2 = 6$ shows that $(1, \compred{1}, 0, 2, 0, \compred{2})$ is another one. However, $(0, 1, \compred{1}, 1, 2, 0)$ is not one, since the third condition forces $i_1 \ge 3$, which contradicts the second condition.
See also \autoref{ex:tableau-kontent}.
\end{example}
\subsection{Tableaux}
\label{sec:tableaux}
We are following \cite{Tableau-complexes-paper} when defining tableaux, so that one of their main theorems, \autoref{thm:tableau-complex-ball-or-sphere}, applies.
First, we associate a poset with each Young diagram as follows. 
If $\lambda$ is a partition, we order the boxes by $(i_1, j_1) \le (i_2, j_2)$ if $i_1 \le i_2$ and $j_1 \le j_2$. If $\lambda$ is a (weak) composition, we order them lexicographically: $(i_1, j_1) \le (i_2, j_2)$ if $i_1 < i_2$ or if $i_1 = i_2$ and $j_1 < j_2$.
A \emph{tableau} $T$ is an order preserving function $T \fc \lambda \to \interval{n}$.
We draw tableaux by filling each box $b$ with $T(b)$.
\begin{figure}
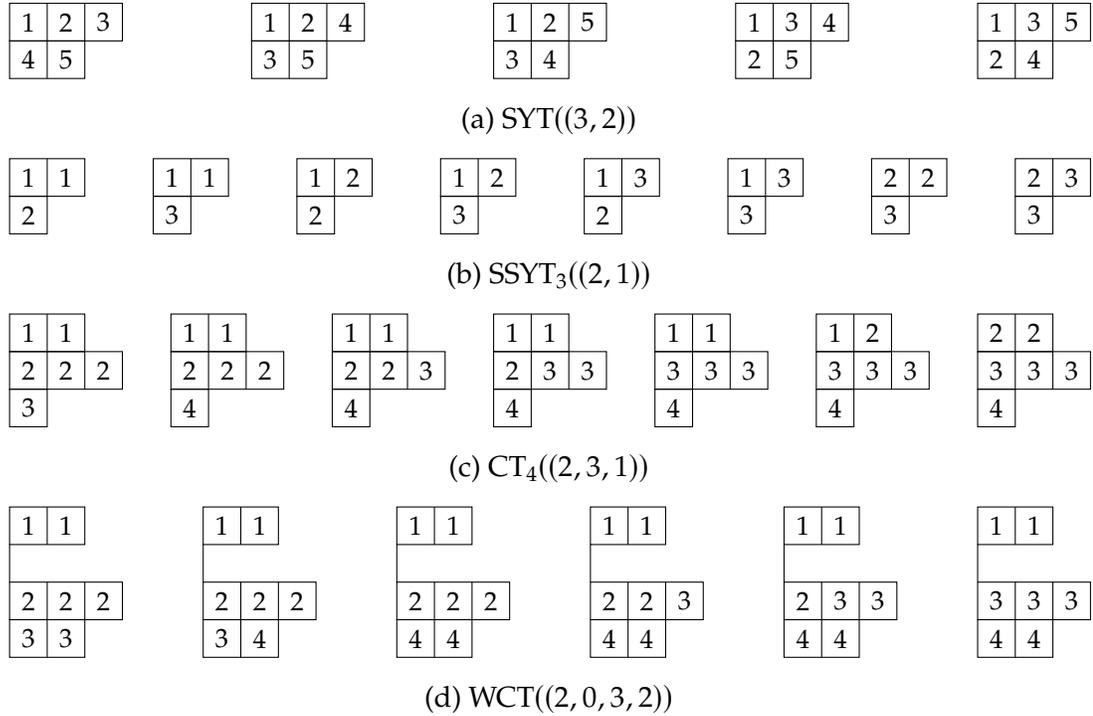
\centering
\begin{subfigure}{.99\textwidth}\centering
  \includestandalone{standard-young-tableaux-ex}
  \caption{$\SYT((3,2))$}
  \label{fig:standard-young-tableaux-ex}
\end{subfigure}
\par\smallskip
\begin{subfigure}{.99\textwidth}\centering
  \includestandalone{semistandard-young-tableaux-ex}
  \caption{$\SSYT_3((2,1))$}
  \label{fig:semistandard-young-tableaux-ex}
\end{subfigure}
\par\smallskip
\begin{subfigure}{.99\textwidth}\centering
  \includestandalone{composition-tableaux-ex}
  \caption{$\CT_4((2,3,1))$}
  \label{fig:composition-tableaux-ex}
\end{subfigure}
\par\smallskip
\begin{subfigure}{.99\textwidth}\centering
\includestandalone{weak-composition-tableaux-ex}
  \caption{$\WCT((2,0,3,2))$}
  \label{fig:weak-composition-tableaux-ex}
\end{subfigure}
\caption{Examples of tableaux.}\label{fig:all-tableaux-ex}
\end{figure}
\begin{itemize}
\item
A \emph{standard Young tableau} of shape $\lambda$, where $\lambda$ is a partition, is a tableau of shape $\lambda$ where each of the numbers $1$, \ldots, $\size{\lambda}$ appears exactly once.
\item
A \emph{semistandard Young tableau} of shape $\lambda$, where $\lambda$ is a partition, is a tableau of shape $\lambda$ with strictly increasing columns.
\item
A \emph{composition tableau} of shape $\lambda$, where $\lambda$ is a composition, is a tableau of shape $\lambda$ where for every $i < j$, each entry in row $i$ is strictly smaller than each entry in row $j$.
\item
A \emph{weak composition tableau} of shape $\lambda$, where $\lambda$ is a weak composition, is a tableau of shape $\lambda$ where for every $i < j$, each entry in row $i$ is strictly smaller than each entry in row $j$, and $i$ does not appear above the $i$th row.
\end{itemize}
We will denote the sets of these tableaux (with largest entry $n$) by $\SYT_n(\lambda)$, $\SSYT_n(\lambda)$, $\CT_n(\lambda)$ and $\WCT_n(\lambda)$, respectively. See \autoref{fig:all-tableaux-ex} for examples.
\begin{remark}
Note that our composition tableaux are different from what for example \cite{Other-composition-tableaux-paper} call composition tableau, 
and our weak composition tableaux are a special case of semi-skyline fillings (compare \cite[Proposition 2.9]{Skyline-slides-paper} with our \autoref{prop:assaf-searles-slide-definition}). We believe our definitions are useful as a mnemonic device (see \autoref{def:fundamental-quasisymmetric} and \autoref{def:slide-new}) and to point out connections to tableau complexes (see \autoref{ex:different-simplicial-complexes} and \autoref{thm:schur-splits-into-balls}).
\end{remark}
Let $\lambda$ be a partition or a (weak) composition and let $\sT$ be a set of tableaux of shape $\lambda$. 
Let $T \subset \lambda \times \bN$ be a set of pairs $(b,i)$, such that every $b \in \lambda$ is in at least one pair. We say that $T$ is a \emph{set-valued tableau}.
If there is a function $f \subseteq T$ with $f \fc \lambda \to \interval{n}$ that lies in $\sT$, then $T$ is a \emph{limit set-valued $\sT$ tableau}. If every such function lies in $\sT$, then $T$ is a \emph{set-valued $\sT$ tableau}.
We draw set-valued tableaux by by filling each box $b$ with the (nonempty) set of numbers $i$ such that $(b,i) \in T$.
See \autoref{fig:set-valued-tableaux-ex} for some examples.

\begin{figure}
  \centering
\includestandalone{set-valued-tableaux-ex}
  \caption[Three set-valued tableaux.]{Three set-valued tableaux. The first is a set-valued semistandard Young tableau. The second and third are limit set-valued semistandard Young talbeau.}
  \label{fig:set-valued-tableaux-ex}
\end{figure}

The \emph{content} of $T \subset \lambda \times \bN$ is the weak composition $\mu = (\mu_1, \dotsc, \mu_n)$, $\mu_i = \setsize{\setbuilder{b}{(b, i) \in T}}$, further, $x^T \defeq x^\mu$ and $\size{T} = \size{\mu}$.
The \emph{standardization} of a semistandard Young tableau $T$ with content $(c_1, \dotsc, c_n)$ is the standard Young tableau obtained by replacing the $1$s by $1$, \ldots, $c_1$, from left to right, the $2$s by $c_1+1$, \ldots, $c_1 + c_2$, also from left to right, and so on. 
The \emph{descent set} of a standard Young tableau $T$ is the set of $i$ that appear strictly above $i+1$ (not necessarily in the same column) in $T$.

There can be many standard and semistandard tableaux with the same shape and content. For example, the middle two tableaux in \autoref{fig:semistandard-young-tableaux-ex} both have content $(1,1,1)$ and shape $(2,1)$. On the other hand, since the entries have to weakly increase when reading from left to right, top to bottom, composition and weak composition tableaux are uniquely determined by their content. Hence we get the following.
\begin{lemma}\label{lemma:composition-tableaux-content}
Given a composition $\lambda$ and a weak composition $\mu$, $\mu$ is the content of a composition tableau of shape $\lambda$ if and only if $\flatt(\mu)$ refines $\lambda$.
\end{lemma}
\begin{lemma}\label{lemma:weak-composition-tableaux-content}
Given weak compositions $\lambda$ and $\mu$, $\mu$ is the content of a weak composition tableau of shape $\lambda$ if and only if $\flatt(\mu)$ refines $\flatt(\lambda)$ and $\mu$ dominates $\lambda$.
\end{lemma}
Set-valued weak composition tableau are not uniquely determined by their content, see \autoref{ex:tableau-kontent}, but not much extra information is needed.
We define the \emph{kontent} $\kappa$ of a set-valued weak composition tableau $T$ of shape $\lambda$ to be the content of $T$, with $\kappa_j$ bolded if $j$ occurs in a box together with some $i < j$.  
We note that the conditions for a $\kappa$ being a glide of $\lambda$ (\autoref{def:glide-of-composition}) are equivalent to $\kappa$ being the kontent of a set-valued weak composition tableau of shape $\lambda$. So we get the following.

\begin{figure}
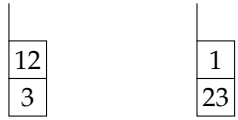

  \centering
  \includestandalone{set-valued-weak-composition-tableaux-content}
  \caption[Two set-valued weak composition tableaux.]{Two set-valued weak composition tableaux, with shape $(0,1,1)$ and content $(1,1,1)$. The left one has kontent $(1,\compred{1}, 1)$ and the right one has kontent $(1,1,\compred{1})$.}
  \label{fig:set-valued-weak-composition-tableaux-content}
\end{figure}

\begin{lemma}\label{lemma:set-valued-weak-composition-tableaux-kontent}
  Given a weak composition $\lambda$ and a weak komposition $\kappa$, $\kappa$ is the kontent of a set-valued weak composition tableau of shape $\lambda$ if and only if $\kappa$ is a glide of $\lambda$.
\end{lemma}

\begin{example}
\label{ex:tableau-kontent}
  Let $\lambda = (0,1,0,0,0,2)$, \autoref{fig:set-valued-weak-composition-tableaux-glide-ex} 
shows three set-valued tableaux of shape $\lambda$. The first two are set-valued weak composition tableaux of shape $\lambda$, but the crossed out one is not, because the $3$ is in row $2$.
So the kontents of the first two tableaux, $(1, 0, 1, 0, \compred{1}, \compred{3})$ and $(1, \compred{1}, 0, 2, 0, \compred{2})$ are glides of $\lambda$. If we tried to make a set-valued weak composition tableau of shape $\lambda$ and kontent $\kappa = (0, 1, \compred{1}, 1, 2, 0)$, we would end up with the third tableau, so $\kappa$ is not a glide of $\lambda$. Compare with \autoref{ex:glide-of-composition}.
\end{example}

\begin{figure}
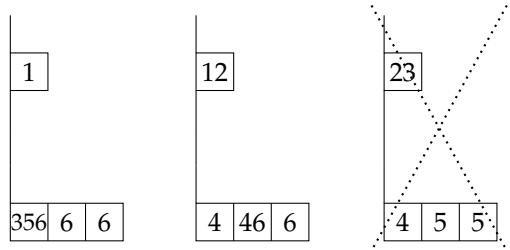

  \centering
  \includestandalone{set-valued-weak-composition-tableaux-glide-ex}
  \caption{Three set-valued tableaux with shape $(0,1,0,0,0,3)$.}
  \label{fig:set-valued-weak-composition-tableaux-glide-ex}
\end{figure}

\subsection{Pipe dreams}\label{sec:pipe-dreams}
Pipe dreams were introduced by Bergeron and Billey \cite{RC-graphs-and-schubert-polynomials} (they called them \emph{rc-graphs}). 
A \emph{pipe dream} $P$ is a tiling of an $n \times n$ matrix using \emph{cross} and \emph{elbow} tiles (see \autoref{cross-and-elbow-tile}), with crosses only appearing above the main antidiagonal. When we draw pipe dreams, we leave out the pipes connecting to the right and bottom edges. 
Above the upper edge, we write the numbers $1$ through $n$. On the left edge, we write the permutation $\pi \in S_n$ that is the result of following the pipes from the upper edge, but treating crosses where the pipes have already crossed (that is, where the label on the horizontal pipe is smaller than the label on the vertical pipe) as elbows. We say that $P$ is a pipe dream for $\pi$.
If no pair of pipes cross twice, the pipe dream is \emph{reduced}. The set of all reduced pipe dreams for $\pi$ is denoted by $\PD(\pi)$.
The \emph{excess} of a non-reduced pipe dream is the number of crosses where the pipes have already crossed.

The \emph{reading word} of $P$ is the word recording on which antidiagonal each cross lies, going from right to left, top to bottom.
The permutation on the left side of the pipe dream is the Demazure product of the reading word. If $P \in \PD(\pi)$, then the reading word for $P$ is a reduced word for $\pi$. If we record the row in which the crosses appear when finding the reading word, we get a compatible sequence for it, and this gives a bijection between pipe dreams and pairs consisting of a word $w$ and a compatible sequence for $w$. 
Given a pipe dream $P$, with $c_i$ crosses in the $i$th row, the \emph{weight} $\wt(P)$ of $P$ is the weak composition $(c_1, \dotsc, c_k)$, and $x^P \defeq x^{\wt(P)}$.

\begin{figure}
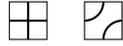

\centering
\includestandalone{elbow-cross-no-half-elbow}
\caption{A cross tile and  an elbow tile.\label{cross-and-elbow-tile}}
\end{figure}

All the information contained in a pipe dream is contained in its set of crosses, so we generally identify a pipe dream with its set of crosses and sometimes draw just the crosses. For example, we say that $P$ contains $P'$, $P \supset P'$, if the crosses of $P'$ are a subset of the crosses of $P$.
Every pipe dream for a permutation $\pi \in S_n$ is contained in the pipe dream for the triangular word in $S_n$, and this gives a bijection between pipe dreams for $\pi$ and subwords of the triangular word. 

\begin{figure}
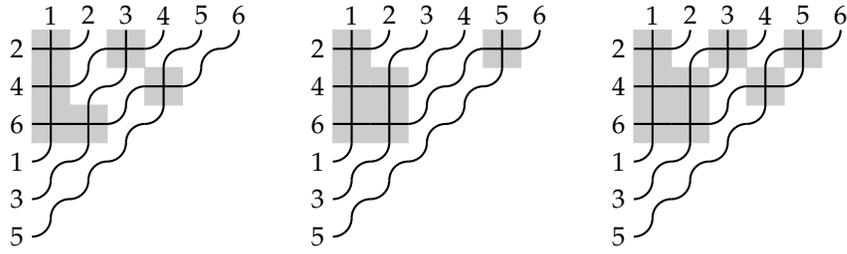

\centering
\includestandalone{pipe-dream-ex}
\caption{Two reduced and one non-reduced pipe dream for $[246135]$.}
\label{fig:pipe-dream-picture}
\end{figure}

\begin{example}
  Two reduced and one non-reduced pipe dream for $[246135]$ are shown in \autoref{fig:pipe-dream-picture}. They have reading words $315243$, $513243$ and $53153243$, respectively. The two reduced pipe dreams have the compatible sequence $112233$ and the non-reduced one has compatible sequence $11122233$. They correspond to the subwords $\dash\dash3\dash15\dash\dash2\dash43\dash\dash\dash$, $5\dash\dash\dash1\dash\dash32\dash43\dash\dash\dash$ and $5\dash3\dash15\dash32\dash43$ of the triangular word $543215432543545$, respectively.
\end{example}

Let $P$ be a pipe dream that has a cross at one of $(i + 1, j)$ and $(i, j + k + 1)$ and an elbow at the other, elbows at $(i, j)$ and $(i + 1, j + k + 1)$ and crosses at $(i, b)$ and $(i+1, b)$ for all $b$ with $j < b < j + k + 1$. 
A \emph{chute move} swaps the tiles at $(i + 1, j)$ and $(i, j + k + 1)$.
The transpose of a chute move is a \emph{ladder move}, and when $k = 0$, the two coincide. See \autoref{fig:chute-ladder-move}. 
\begin{figure}
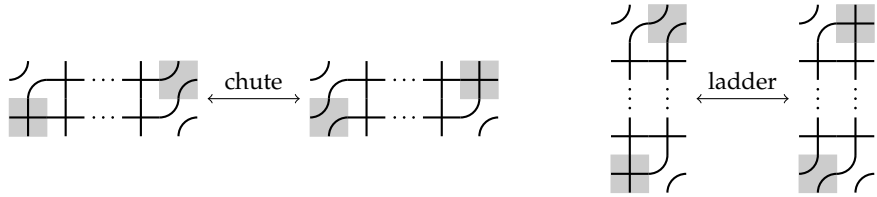

\centering
\includestandalone{chute-ex}
\caption{A chute move and a ladder move. \label{fig:chute-ladder-move}}
\end{figure}
The chute and ladder moves do not change the permutation. Further, given a reduced pipe dream $P$ for a permutation $\pi$, we can find all other reduced pipe dreams for $\pi$ by doing chute and ladder moves to $P$ \cite[Corollary 3.8]{RC-graphs-and-schubert-polynomials}.
In order to find all the pipe dreams for $\pi$, it is thus enough to find one pipe dream for $\pi$ and perform chute and ladder moves to it. One starting point for this process is the pipe dream with crosses at $(i, j)$ for $j \le l_i$, where $(l_1, \dots, l_k)$ is the Lehmer code of $\pi$. This pipe dream is called the \emph{bottom pipe dream} for $\pi$.

However, the chute and ladder moves can change the reading word, and in \autoref{sec:slides-for-words} we will be interested in the sets of pipe dreams with a given reading word. Each of these sets has an analogue of the bottom pipe dream, introduced by Assaf and Searles \cite[Definition 3.10]{Slide-paper} and extended to non-reduced pipe dreams by Pechenik and Searles \cite[Definition 2.11]{Glide-paper}.
A possibly non-reduced pipe dream is \emph{\qY/} if the leftmost cross in each row lies either in the first column or weakly to the left of some cross in the row below.
The \qY/ reduced pipe dreams for $[135624]$ are shown in \autoref{fig:qy-pipe-dreams}.
\begin{figure}
\centering
\includestandalone{qy-example}
\caption{All the \qY/ reduced pipe dreams for $[135624]$.} \label{fig:qy-pipe-dreams}
\end{figure}
\begin{lemma}[\protect{\cite[Lemma 3.12]{Slide-paper}}]\label{lemma:quasi-yamanouchi-to-word}
  For each word $w$ for $\pi$ that is a subword of the triangular word, there is exactly one \qY/ reduced pipe dream with reading word $w$.
\end{lemma}

\section{Simplicial complexes}
Here we introduce the relevant facts concerning simplicial complexes and the two classes of simplicial complexes we are interested in.
\subsection{Basic definitions}
A \emph{simplicial complex} $\Delta$ on a finite set $X$ is a downward closed collection of sets $F \subseteq X$, called \emph{faces}. 
That is, if $F \in \Delta$ and $F' \subset F$, then $F' \in \Delta$. A maximal face is called a \emph{facet}. The dimension of a face $F$ is $\dim F = \size{F} - 1$, and if $\dim F = d$, $F$ is called a \emph{$d$-face}. The $0$-faces are called \emph{vertices}, and we generally write $v$ for the set $\{v\}$. If all facets of $\Delta$ have the same dimension $n$, $\Delta$ is \emph{pure} and we say that $\dim\Delta = n$, otherwise, $\dim\Delta$ is the maximum dimension of its facets. The simplicial complex consisting of all subsets of a $d$-face is called a \emph{$d$-simplex}. An element $x \in X$ that does not lie in any face of $\Delta$ is a \emph{phantom vertex}.

Let $S$ be a polynomial ring with variables $x_v$ for each $v \in X$. The \emph{Stanley--Reisner ideal} of $\Delta$ is the ideal $I_\Delta$ generated by the products $\prod_{v \in F} x_v$ for each $F \subset X$ that is not a face of $\Delta$. The \emph{Stanley--Reisner ring} of $\Delta$ is the ring $\SRring{\Delta} \defeq S/I_\Delta$.
The \emph{geometric realization} $\abs{\Delta}$ of $\Delta$ is the intersection of $\Spec\SRring{\Delta}$ with the \emph{standard simplex} $\setbuilder{\sum a_v x_v}{\text{$\sum a_v = 1$ and $0 \le a_v \in \bR$}}$. 
A set $F = \{x_{v_1}, \dotsc, x_{v_k}\}$ is a face of $\Delta$ if and only if the convex hull of $F$ is contained in the geometric realization.
We say that a simplicial complex is a \emph{ball} or a \emph{sphere}, respectively, if its geometric realization is a ball or a sphere, respectively. 
If the geometric realization of $\Delta$ is a topological manifold with boundary, the \emph{boundary} of $\Delta$ consists of the faces corresponding to the boundary of $\abs{\Delta}$.

The \emph{deletion} $\del_F(\Delta)$ of a face $F$ from a simplicial complex $\Delta$ is the simplicial complex consisting of all the faces in $\Delta$ that do not intersect $F$:
\[
\del_F(\Delta) \defeq \{\,G \in \Delta \st G \cap F = \emptyset\,\},
\]
and the \emph{link} $\link_F(\Delta)$ is the simplicial complex consisting of all the faces in $\Delta$ that do not intersect $F$, but which together with $F$ form a face of $\Delta$:
\[
\link_F(\Delta) \defeq \{\,G \in \Delta \st G \cap F = \emptyset \text{ and } G \cup F \in \Delta\,\}.
\]
See \autoref{fig:simplicial-deletion-link} for an example.

\begin{figure}
  \centering
  \begin{subfigure}{0.32\textwidth}\centering
    \includestandalone[width=0.95\textwidth]{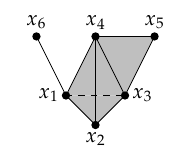}
    \caption{$\Delta$}
  \end{subfigure}
  \hfill
  \begin{subfigure}{0.32\textwidth}\centering
    \includestandalone[width=0.95\textwidth]{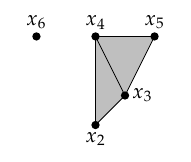}
    \caption{$\del_{x_1}\Delta$}
  \end{subfigure}
  \hfill
  \begin{subfigure}{0.32\textwidth}\centering
    \includestandalone[width=0.95\textwidth]{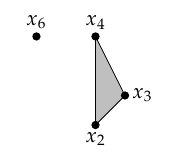}
    \caption{$\link_{x_1}\Delta$}
  \end{subfigure}
  \caption[A simplicial complex and the deletion and link of a vertex.]{The simplicial complex $\Delta$ with facets $\{x_1, x_2, x_3, x_4\}$, $\{x_1, x_6\}$ and $\{x_3, x_4, x_5\}$, and the deletion and link of $x_1$. \label{fig:simplicial-deletion-link}}
\end{figure}

Let $\Delta$ be a simplicial complex on $X$ with facets $F_1$, \ldots, $F_n$, and let $v \notin X$. The \emph{cone} from $v$ on $\Delta$ is the simplicial complex on $X \cup \{v\}$ with facets $F_1 \cup \{v\}$, \ldots, $F_n \cup \{v\}$.
If a vertex $v$ lies in every facet of $\Delta$, we say that $v$ is a \emph{cone vertex} of $\Delta$. In that case, $\del_v(\Delta) = \link_v(\Delta)$, and $\Delta$ is the cone from $v$ on $\del_v(\Delta)$.

\begin{example}
  Let $\Delta$ be the simplicial complex in \autoref{fig:non-vertex-decomposable}, with facets $\{x_1, x_2, x_3\}$ and $\{x_3, x_4, x_5\}$. Then, $x_3$ is a cone vertex of $\Delta$ and $\Delta$ is the cone from $x_3$ on the simplicial complex with facets $\{x_1, x_2\}$ and $\{x_4, x_5\}$.
\end{example}
Vertex-decomposability was introduced by Billera and Provan \cite{Vertex-decomposability-paper}. There are many equivalent formulations of it. The following is from \cite{Subword-complexes-paper}.
\begin{definition}\label{def:vertex-decomposable}
A simplicial complex $\Delta$ is \emph{vertex-decomposable}  if it is pure and either 
$\Delta = \{\emptyset\}$, or there is a vertex $v$ such that $\del_v(\Delta)$ and $\link_v(\Delta)$ are both vertex-decomposable.
\end{definition}

\begin{example}
  The $d$-simplex $\Delta$ on $\{x_1, \dotsc, x_d\}$ is vertex-decomposable: both $\del_{x_d}(\Delta)$ and $\link_{x_d}(\Delta)$ equal the $(d-1)$-simplex on $\{x_1, \dotsc, x_{d-1}\}$, which is vertex-decomposable by induction.
\end{example}
\begin{example}
  The simplicial complex $\Delta$ on $\{x_1, \dotsc, x_5\}$ with facets $\{x_1,x_2,x_3\}$ and $\{x_3, x_4, x_5\}$ shown in \autoref{fig:non-vertex-decomposable} is pure but is not vertex-decomposable: 
$\del_{x_i}(\Delta)$ is not pure for $i \neq 3$, and $\del_{x_3}(\Delta)$ is not vertex-decomposable, since none of the $\del_{x_i}(\del_{x_3}(\Delta))$ are pure.
\end{example}
\begin{figure}
\centering
\includestandalone[width=.3\textwidth]{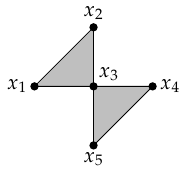}
\caption[A pure simplicial complex that is not vertex-decomposable.]{A pure simplicial complex that is not vertex-decomposable. The vertex $x_3$ is a cone vertex.}
\label{fig:non-vertex-decomposable}
\end{figure}

A $d$-dimensional simplicial complex $\Delta$ is \emph{shellable} if it is pure and there is an ordering $F_1$, \ldots, $F_n$ of its facets such that the intersections $F_j \cap \left(\bigcup_{i = 1}^{j-1} F_i\right)$ are nonempty, pure and $(d-1)$-dimensional.
Every vertex-decomposable simplicial complex is shellable \cite{Vertex-decomposability-paper}. Basically, a shelling for $\link_v(\Delta)$, coned from $v$, plus a shelling for $\del_v(\Delta)$, gives us a shelling of $\Delta$. 
We will use the following theorem to prove that simplicial complexes are balls or spheres. 

\begin{theorem}[\protect{\cite[Proposition 4.7.22]{Oriented-matroids}}]\label{thm:ball-sphere}
 A shellable simplicial complex where every codimension $1$ face appears in at most two facets is a ball or a sphere. If every codimension $1$ face appears in exactly two facets it is a sphere, otherwise it is a ball and 
the codimension $1$ faces that appear in only one facet are in its boundary. 
\end{theorem}

\subsection{Subword complexes}
\label{sec:subword-complexes}
Subword complexes were introduced by Knutson and Miller \cite{Grobner-geometry-paper} to explain the combinatorics of determinantal ideals and Schubert polynomials, and were further studied in \cite{Subword-complexes-paper}, where they proved that subword complexes are balls or spheres. 

Let $\pi$ be an element of a Coxeter group and $Q$ a word of length $n$.
We identify subwords of $Q$ with subsets of the positions in $Q$. 
Given two subwords of $Q$ that contain $\pi$, their union also contains $\pi$, while their intersection might not. Similarly, if $P$ contains $\pi$ and $P' \supset P$, then $P'$ also contains $\pi$. So if we take complements inside $Q$, we get a simplicial complex, the \emph{subword complex} $\Delta(Q, \pi)$. More precisely, the faces of $\Delta(Q, \pi)$ are the subwords whose complements contain $\pi$.
The facets of $\Delta(Q, \pi)$ are the complements of subwords representing $\pi$, and the vertices are the individual letters of $Q$ (a vertex is a phantom vertex if it appears in every word for $\pi$).

\begin{example}
  Let $Q = 321323$ and $\pi = [1432] \in S_n$. There are two reduced words for $\pi$: $232$ and $323$, so the facets of $\Delta(321323, [1432])$ in $S_n$ correspond to the complements of the embeddings $\dash2\dash32\dash$, $32\dash3\dash\dash$, $32\dash\dash\dash3$, $3\dash\dash\dash23$ and $\dash\dash\dash323$, that is: $3\dd1\dd\dd3$, $\dd\dd1\dd23$, $\dd\dd132\dd$, $\dd213\dd\dd$ and $321\dd\dd\dd$.
  On the other hand, the facets of $\Delta(321323, 323)$ in $W_n$ correspond to only the complements of the embeddings of $323$ into $Q$, that is, all but the first one above. 
  The two complexes are shown in \autoref{fig:subword-complex-ex}.
\end{example}

\begin{figure}
  \centering
  \begin{subfigure}{0.49\textwidth}\centering
    \includestandalone[width=\textwidth]{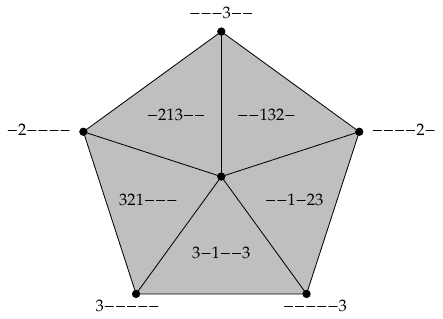}
    \caption{$\Delta(321323, [1432])$}
  \end{subfigure}
  \hfill
  \begin{subfigure}{0.49\textwidth}\centering
    \includestandalone[width=\textwidth]{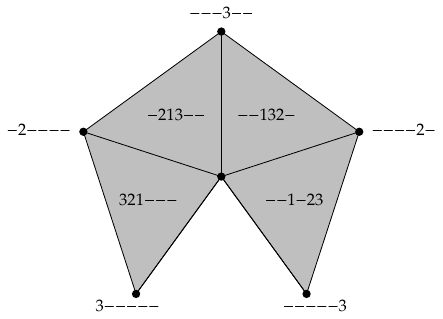}
    \caption{$\Delta(321323, 323)$}
  \end{subfigure}
  \caption[$\Delta(321323, \sqbl1432\sqbr)$ in $S_n$ and $\Delta(321323, 323)$ in $W_n$.]{The subword complexes $\Delta(321323, [1432])$ in $S_n$ and $\Delta(321323, 323)$ in $W_n$. We have labeled the facets and the vertices (except the center vertex $\dash\dash1\dash\dash\dash$).}
  \label{fig:subword-complex-ex}
\end{figure}

The following steps show that subword complexes are balls or spheres. We will modify \autoref{thm:subword-complex-vertex-decomposable} in \autoref{sec:which-sets-give-balls}, so we are including the proof here. 

\begin{theorem}[\protect{\cite[Theorem E]{Grobner-geometry-paper}}]\label{thm:subword-complex-vertex-decomposable}
  The subword complex $\Delta(Q,\pi)$ is vertex-decom\-posable. In particular, it is shellable. 
\end{theorem}
\begin{proof}
  Let $Q = (\sigma, \sigma_1, \dotsc, \sigma_n)$,  and $Q' = (\sigma_1, \dotsc, \sigma_n)$. The link of $\sigma$ consists of the faces $P$ of $\Delta(Q, \pi)$ that do not contain $\sigma$, and $P \cup \sigma$ is a face of $\Delta(Q, \pi)$. That is, it consists of those faces $P$ of $\Delta(Q, \pi)$ whose complements $(Q \setminus P)$ are of the form $\sigma R$, where $R$ contains a word for $\pi$. So the link is isomorphic to $\Delta(Q', \pi)$.

The deletion of $\sigma$ consists of faces $P$ of $\Delta(Q, \pi)$ that do not contain $\sigma$. So the complement $(Q \setminus P)$ is of the form $\sigma R$, where $\sigma R$ contains a word for $\pi$. Now there are two cases, depending on whether $\len(\sigma \pi)$ is equal to $\len(\pi) + 1$ or $\len(\pi) - 1$. 

If $\len(\sigma \pi) = \len(\pi) + 1$, then no word for $\pi$ begins with $\sigma$, and therefore, if $\sigma R$ contains a word for $\pi$, the word has to be fully contained in $R$. So the deletion of $\sigma$ is equal to the link, and is isomorphic to $\Delta(Q', \pi)$.

If $\len(\sigma \pi) = \len(\pi) - 1$, then, $\sigma R$ contains a word for $\pi$ if and only if $R$ contains a word for $\sigma\pi$. So the deletion is isomorphic to $\Delta(Q', \sigma\pi)$.

Thus, the link and deletion of $\sigma$ in $\Delta(Q, \pi)$ are both isomorphic to subword complexes for the shorter word $Q'$, and are therefore vertex-decomposable, and so $\Delta(Q, \pi)$ is vertex-decomposable.
\end{proof}

The exchange condition for Coxeter groups implies the following, which together with Theorems \ref{thm:ball-sphere} and \ref{thm:subword-complex-vertex-decomposable} gives \autoref{thm:subword-complexes-are-balls-or-spheres}.

\begin{lemma}[\protect{\cite[Lemma 3.5]{Subword-complexes-paper}}]\label{lemma:subword-complex-thin}
In the subword complex $\Delta(Q, \pi)$, every codimension $1$ face lies in at most two facets.
\end{lemma}
\begin{theorem}[\protect{\cite[Theorem 3.7, Corollary 3.8]{Subword-complexes-paper}}]\label{thm:subword-complexes-are-balls-or-spheres}
  The subword complex $\Delta(Q,\pi)$ is either a ball or a sphere. If $\Dem(Q) = \pi$ it is a sphere. Otherwise, it is a ball and the faces $P$ with $\Dem(Q \setminus P) > \pi$ form the boundary sphere.
\end{theorem}
Since pipe dreams correspond to subwords of triangular words in $S_n$, when $Q$ is a triangular word, we will refer to the subword complex $\Delta(Q, \pi)$ as a \emph{pipe dream complex} and when we draw it, we will label the facets $P$ with the pipe dream corresponding to $(Q \setminus P)$.

\begin{definition}[\protect{\cite[Definition 20]{Smirnov-Tutubalina}}]
  Let $Q$ be a word, $\pi$ a permutation, and $w$ a word for $\pi$. The \emph{slide complex} $\Delta(Q, w)$ is the subcomplex of $\Delta(Q, \pi)$ whose facets have complements representing $w$.
\end{definition}
The following was proven by Smirnov and Tutubalina with the same proof \cite[Theorem 6]{Smirnov-Tutubalina}. It also follows as a corollary of \autoref{thm:delta-w-ball-or-sphere}.
\begin{lemma}\label{lemma:slide-complex-ball-or-sphere}
  The slide complex $\Delta(Q, w)$ is a subword complex, in particular, it's a ball or a sphere.
\end{lemma}
\begin{proof}
We can identify $w$ with an element in the universal Coxeter group $W_n$, so $\Delta(Q, w)$ is a subword complex and by \autoref{thm:subword-complexes-are-balls-or-spheres}, it is a ball or sphere.
\end{proof}

\subsection{Tableau complexes}
Tableau complexes were introduced by Knutson, Miller, and Yong \cite{Tableau-complexes-paper}. One of their main results is \autoref{thm:tableau-complex-ball-or-sphere} (stated there in greater generality).

The union of two limit set-valued tableaux of shape $\lambda$ is again a limit set-valued tableau of shape $\lambda$, while the intersection might not be, since there might be empty boxes. Like with subword complexes, we therefore get a simplicial complex by taking complements inside some large limit set-valued tableau. More precisely, let $\sT$ be a finite set of tableaux and let $X$ be a set-valued tableau containing all the tableaux in $\sT$.
The tableau complex $\tabcomplexE{\sT}{X}$ is the simplicial complex on $X$ whose faces $F \subset X$ have $T \subset (X \setminus F)$ for some $T \in \sT$.
Thus, the facets of $\tabcomplexE{\sT}{X}$ are the complements of the tableaux in $\sT$, the vertices are individual entries and in general, faces are complements of set-valued tableaux. 
We write $\tabcomplex{\sT}$ for $\tabcomplexE{\sT}{\bigcup_{T \in \sT}T}$. Statements involving $\tabcomplexE{\sT}{X}$ are do not depend on $X \supset \bigcup_{T \in \sT}T$. 
When drawing tableau complexes, we label the face $(X \setminus F)$ with the tableau $F$. See \autoref{fig:tableau-complex-example} for an example.

\begin{figure}
  \centering
  \includestandalone{tableau-complex-SSYT}
  \caption{The tableau complex $\tabcomplex{\SSYT_4((1,1,1))}$.}
  \label{fig:tableau-complex-example}
\end{figure}

\begin{theorem}[\protect{\cite[Theorem 2.8]{Tableau-complexes-paper}}]\label{thm:tableau-complex-ball-or-sphere}
  Let $B$ be a finite partially ordered set. For each $b$ in $B$, let $I_b$ be a subset of $\interval{n}$. Let $\Psi$ be a set of pairs $(b_1, b_2)$ of elements in $B$ such that $b_1 < b_2$. Let $\sT$ be the set of all tableaux $T$ that satisfy:
  \begin{itemize}
  \item for each $b$, $T(b) \in I_b$,
  \item if $(b_1, b_2) \in \Psi$, then $T(b_1) < T(b_2)$.
  \end{itemize}
Then, the tableau complex $\tabcomplexE{\sT}{X}$ is vertex-decomposable, and it is a ball or a sphere.
\end{theorem}
In \cite{Tableau-complexes-paper}, they used this to prove that the tableau complex $\tabcomplexE{\SSYT_n(\lambda)}{X}$ is a vertex-decomposable ball or sphere. We include that proof and also show that $\tabcomplexE{\CT_n(\lambda}{X}$ and $\tabcomplexE{\WCT_n(\lambda)}{X}$ are balls or spheres. On the other hand, $\tabcomplexE{\SYT_n(\lambda)}{X}$ is generally not a ball or a sphere (see \autoref{ex:syt-not-ball-or-sphere}).

\begin{proposition}\label{thm:semistandard-tableau-complex}\label{thm:composition-tableau-complex}\label{thm:weak-composition-tableau-complex}
  Let $\lambda$ be a partition, composition or weak composition as appropriate. The tableau complexes $\tabcomplexE{\SSYT_n(\lambda)}{X}$, $\tabcomplexE{\CT_n(\lambda)}{X}$ and $\tabcomplexE{\WCT_n(\lambda)}{X}$ are vertex-decomposable balls or spheres.
\end{proposition}
\begin{proof}
  In all three cases, $B$ is the Young diagram of shape $\lambda$, with the different orders defined in \autoref{sec:tableaux}. For $\SSYT_n(\lambda)$ and $\CT_n(\lambda)$, $I_b = \interval{n}$ for every $b$, and for $\WCT_n(\lambda)$, $I_{(i,j)} = \setbuilder{k}{k \le i}$. For $\SSYT_n(\lambda)$, $\Psi $ is the set of all pairs $((i,j), (i+1, j))$, so that the entry in a box is smaller than the entry in the box below it. For $\CT_n(\lambda)$ and $\WCT_n(\lambda)$, $\Psi$ is the set of all pairs $((i, j), (i', j'))$ with $i<i'$, so that every entry in row $i$ is smaller than every entry in row $i'$.
Thus, by \autoref{thm:tableau-complex-ball-or-sphere}, they are vertex-decomposable balls or spheres.
\end{proof}

\begin{example}\label{ex:syt-not-ball-or-sphere}
  The tableau complex of standard Young tableaux of shape $\lambda$ is in general not a ball or a sphere. For example, the tableau complex of standard Young tableaux of shape $(2,1)$ is shown in \autoref{fig:standard-young-tableau-not-complex}.
\end{example}

\begin{figure}
  \centering
  \includestandalone{standard-young-tableau-not-complex-correct}
  \caption{The tableau complex $\tabcomplex{\SYT_3((2,1))}$.}
  \label{fig:standard-young-tableau-not-complex}
\end{figure}

Just like for subword complexes, ``ball or sphere'' means ``ball'' except in very special cases. 
For example, the tableau complex $\tabcomplexE{\SSYT_n(\lambda)}{X}$ is a sphere if and only if $X$ equals $\bigcup_{T \in \SSYT_n(\lambda)}T$ and contains exactly one box with more than one entry (necessarily the rightmost box in the top row).
The analogue of the condition on $\Dem(Q)$ in \autoref{thm:subword-complexes-are-balls-or-spheres} is the following.
\begin{proposition}[\protect{\cite[Proposition 2.2]{Tableau-complexes-paper}}]\label{thm:tableau-complex-sphere-condition} 
  Assume that $\tabcomplexE{\sT}{X}$ is homeomorphic to a ball or sphere. A face $(X \setminus F)$ is interior to $\tabcomplexE{\sT}{X}$ if and only if $F$ is a set-valued $\sT$ tableau (as opposed to a limit set-valued $\sT$ tableau).
\end{proposition}

\section{Polynomials}
Here we talk about various families of polynomials that are related to the objects we are studying.

\subsection{Symmetric and quasisymmetric polynomials}
\label{sec:fundamental-quasisymmetric}
A polynomial in $n$ variables $x_1$, \ldots, $x_n$ is \emph{quasisymmetric} if it is invariant under swapping $x_i$ and $x_{i+1}$ in all monomials that do not contain both $x_i$ and $x_{i+1}$, or equivalently, if the coefficient of $x_{i_1}^{a_1} \dotsm x_{i_k}^{a_k}$ equals the coefficient of $x_{j_1}^{a_1} \dotsm x_{j_k}^{a_k}$ for all increasing sequences $i_1 < \dotsb < i_k$ and $j_1 < \dotsb < j_k$ and compositions $(a_1, \dotsc, a_k)$. 
It is \emph{symmetric} if it is invariant under swapping $x_i$ and $x_{i+1}$ in all monomials, or equivalently, if the coefficient of $x_{i_1}^{a_1} \dotsm x_{i_k}^{a_k}$ equals the coefficient of $x_{j_1}^{a_1} \dotsm x_{j_k}^{a_k}$ for all sequences $i_1$, \ldots, $i_k$ and $j_1$, \ldots, $j_k$ of distinct numbers and all compositions $(a_1, \dotsc, a_k)$. 

\begin{example}
  A quasisymmetric polynomial in $x_1$, $x_2$ and $x_3$ that contains the term $x_1^2x_2$ must also contain $x_1^2x_3$ and $x_2^2x_3$. A symmetric polynomial containing the term $x_1^2 x_2$ must also contain those two terms, as well as $x_1x_2^2$, $x_1x_3^2$ and $x_2x_3^2$.
\end{example}

The Schur polynomials are symmetric and form a basis for the ring of symmetric polynomials \cite[Proposition 1 in Section 6.1]{Fulton-young-tableaux-book}. Among other applications, they represent the cohomology classes of Schubert varieties in the Grassmannian \cite[Section~9.4]{Fulton-young-tableaux-book}. 
\begin{definition}\label{def:schur-polynomial}
Given a partition $\lambda$, the \emph{Schur polynomial} $\schur_\lambda$ is given by 
\[
\schur_\lambda(x_1, \dotsc, x_n) \defeq \sum_{T \in \SSYT_n(\lambda)}x^T.
\]
\end{definition}

We give our own definition of fundamental quasisymmetric polynomials in terms of composition tableaux. We find it easier to remember than the common definition in \autoref{prop:standard-fundamental-quasisymmetric-definition}, especially because of how similar it is to the definition above.
\begin{definition}\label{def:fundamental-quasisymmetric}
Given a composition $\lambda$, the \emph{fundamental quasisymmetric polynomial} $\quasi_\lambda$ is defined by 
\[
\quasi_\lambda(x_1, \dotsc, x_n) \defeq \sum_{T \in \CT_n(\lambda)} x^T.
\]
\end{definition}

The following proposition shows that \autoref{def:fundamental-quasisymmetric} is equivalent to a standard definition of fundamental quasisymmetric functions \cite[Equation~7.89]{Stanley-2}.

\begin{proposition}\label{prop:standard-fundamental-quasisymmetric-definition}
  Given a composition $\lambda$,
\[
\quasi_\lambda(x_1, \dotsc, x_n) = \sum_{\substack{1\le i_i \le \dotsb \le i_k \le n\\ \text{$i_j < i_{j+1}$ if $j \in S_\lambda$}}}x_{i_1}\dotsm x_{i_k}.
\]
\end{proposition}
\begin{proof}
Given a sequence $i_1 \le \dotsb \le i_k$, let $\kappa$ be the weak composition whose $j$th part is the number of $j$s in the sequence, so that $x_{i_1} \dotsm x_{i_k} = x^\kappa$.
The condition that $i_j < i_{j+1}$ if $j \in \comptoset{\lambda}$ becomes the condition that $\flatt(\kappa)$ refines $\lambda$. By \autoref{lemma:composition-tableaux-content}, this happens if and only if there is a composition tableau of shape $\lambda$ with content $\kappa$.
\end{proof}
\begin{example}
The composition tableaux of shape $(1,2,1)$ with largest entry $4$ are shown in \autoref{fig:composition-tableaux-fundamental-quasisymmetric}, so 
\[
\quasi_{(1,2,1)}(x_1, x_2, x_3, x_4) = x_1x_2^2x_3 + x_1x_2^2x_4 + x_1x_2x_3x_4 + x_1x_3^2x_4 + x_2x_3^2x_4. \qedhere
\]
\end{example}
\begin{figure}
  \centering
  \includestandalone{composition-tableaux-fundamental-quasisymmetric}
  \caption{All composition tableaux of shape $(1,2,1)$ with largest entry $4$.}
  \label{fig:composition-tableaux-fundamental-quasisymmetric}
\end{figure}

Schur polynomials can be written as a sum of fundamental quasisymmetric polynomials \cite[Theorem 7.19.7]{Stanley-2}: 
\[ 
\schur_\lambda = \sum_{T \in \SYT(\lambda)}\quasi_{\comp(\Des(T))}.
\]
Note that the sum is over \textit{standard} Young tableaux, not \textit{semistandard} Young tableaux.

\subsection{Schubert polynomials}
\label{sec:schubert-polynomials}
There are many equivalent definitions of Schubert polynomials. We will mention two related ones that we will use. 
The main one is in terms of pipe dreams.
\begin{definition}[\protect{\cite[Corollary 3.3]{RC-graphs-and-schubert-polynomials}}]
For a permutation $\pi$, the \emph{Schubert polynomial} $\Schub_\pi$ is defined as
\[
\Schub_\pi = \sum_{P \in \PD(\pi)} x^P.
\] 
\end{definition}
Since reduced pipe dreams for $\pi$ correspond to reduced words for $\pi$ together with compatible sequences, we can write the definition in terms of reduced words. 
\begin{proposition}[\protect{\cite[Theorem 2.2]{Fomin-Stanley}, \cite[Theorem 1.1]{Billey-Jockusch-Stanley}}]
\label{prop:schubert-reduced-words} 
For a permutation $\pi$, 
\[
\Schub_\pi = \sum_{\substack{w \in \RW(\pi)\\s \le w}}x^s,
\] 
where $s \le w$ means that $s$ is a compatible sequence for $w$.
\end{proposition}
The Schubert polynomials form a basis for the polynomial ring $\bZ[x_1,\dotsc]$ \cite[Corollary 3.9]{RC-graphs-and-schubert-polynomials}, so there are numbers $\ssc{\pi}{\rho}{\sigma}$ called \emph{Schubert structure constants} such that
\[
\Schub_\pi \Schub_\rho = \sum_{\sigma}\ssc\pi\rho\sigma \Schub_\sigma.
\]
It is known from geometric considerations that these numbers are non-negative, and finding a positive formula for these numbers for arbitrary $\pi$, $\rho$ and $\sigma$ is a major open problem in Schubert calculus \cite[Problem 1]{Nenashev}. 

\begin{definition}
  Given a permutation $\pi$, the \emph{Grothendieck polynomial} $\Groth_\pi$ is given by
\[
\Groth_\pi = \sum_{P}(-1)^{\excess P}x^P,
\]
where the sum is over all pipe dreams for $\pi$, not just the reduced ones.
\end{definition}
Grothendieck polynomials are the $K$-polynomials of matrix Schubert varieties \cite[Theorem A]{Grobner-geometry-paper}.

\subsection{Back-stable Schubert polynomials}\label{sec:back-stable-schuberts}
The back-stable Schubert polynomials were introduced by Knutson around 2003 and studied by Lam, Lee and Shimozono \cite{Back-stable-paper} as a generalization of Schubert polynomials to the ring $\bZ[\dotsc, x_{-1}, x_0, x_1, \dotsc]$. 
In \autoref{sec:multiplying-slides}, we give shuffle rules for the products $\Schub_\pi\Schub_{s_i}$ and $\Schub_\pi\Schub_{(i, i+1, \dotsc, i+k)}$.
In fact, the rules are for back-stable Schubert polynomials, because they are more uniform. For example, everything that ``should'' have a compatible sequence has one. 
By \autoref{thm:back-stable-constants-are-regular-constants}, any rule for multiplying back-stable Schubert polynomials gives us a rule for multiplying Schubert polynomials.
The below construction of the back-stable Schubert polynomial follows \cite[Section 3]{Nenashev}.

For $\pi$ in $S_\bZ$, define $\tau\pi$ by $(\tau\pi)(i+1) \defeq \pi(i) + 1$. If $\pi(i) = i$, then $(\tau\pi)(i+1) = i+1$, so since $\pi$ fixes all but finitely many elements of $\bZ$, $\tau\pi$ also fixes all but finitely many elements. If $(\tau\pi)(i) = (\tau\pi)(j)$, then $\pi(i-1) = \pi(j-1)$, so $i = j$. So $\tau\pi$ is a permutation of the integers that fixes all but finitely many elements, and $\tau$ defines a map $S_\bZ \to S_\bZ$, with $\inv \tau \pi(i-1) = \pi(i) - 1$. 
\begin{example}\label{ex:shifting-of-schubert}
If $\pi = [321]$, then $\tau\pi = [1432]$. The pipe dream for $\pi$ and the pipe dreams for $\tau\pi$ are shown in \autoref{fig:tau-pi-pipe-dreams}. The pipe dreams for $\tau\pi$ include the pipe dream for $\pi$, shifted down one row, as well as four new ones (all of which have at least one cross in the top row). So $\Schub_\pi = x_1^2x_2$ and $\Schub_{\tau\pi} = x_2^2x_3 + x_1x_2x_3 + x_1^2x_3 + x_1^2x_2 + x_1x_2^2$. Note that if we plug $x_0$, $x_1$ and $x_2$ into $\Schub_{\tau\pi}$, we get 
\[
\Schub_{\tau\pi}(x_0, x_1, x_2) = \Schub_{\pi}(x_1, x_2) + x_0(x_1x_2 + x_0x_2 + x_0x_1 + x_1^2),
\]
so $\Schub_\pi(x_1, x_2) = \Schub_{\tau\pi}(0, x_1, x_2)$.

The reduced words for $\pi$ are $212$ and $121$, while the reduced words for $\tau\pi$ are $323$ and $232$.
\end{example}
\begin{figure}
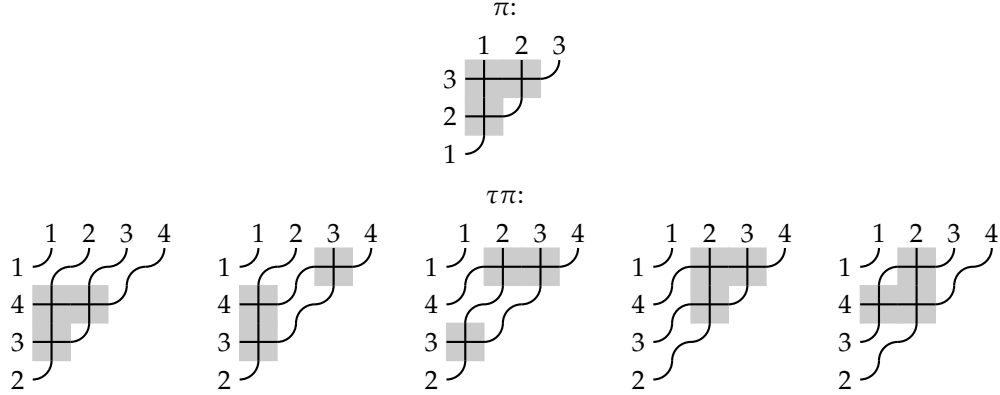
\centering
\includestandalone{pipe-dream-pi-tau-pi}
\caption[Pipe dreams for $\pi = \sqbl321\sqbr$ and $\tau\pi = \sqbl1432\sqbr$.]{The pipe dream for $\pi = [321]$ and the pipe dreams for $\tau\pi = [1432]$. The leftmost pipe dream for $\tau\pi$ is the pipe dream for $\pi$ with all the crosses shifted down one row.}
\label{fig:tau-pi-pipe-dreams}
\end{figure}

\autoref{ex:shifting-of-schubert} motivates the following definition.
\begin{definition}[\protect{\cite[Definition 4]{Nenashev}, \cite[cf. Theorem 3.2]{Back-stable-paper}}]  
\label{def:back-stable-definition} Given a permutation $\pi \in S_\bZ$, the \emph{back-stable Schubert polynomial} $\bSchub_\pi$ is the formal power series
\[
  \bSchub_\pi = \lim_{k \to \infty} \Schub_{\tau^k\pi}(x_{1-k}, x_{2-k}, x_{3-k}, \dotsc).
\]
\end{definition}
A more useful description is the following, which also shows that they are similar to Schubert polynomials (cf. \autoref{prop:schubert-reduced-words}).
\begin{proposition}[\protect{\cite[Theorem 3.2]{Back-stable-paper}}]\label{prop:back-stable-reduced-words}
  For $\pi$ in $S_\bZ$,
\[
\bSchub_\pi = \sum_{\substack{w\in \RW(\pi)\\s\le w}}x^s,
\]
where $s \le w$ if $s$ is a compatible sequence (possibly containing negative numbers) for $w$.
\end{proposition}

Let $\backR$ be the ring of formal power series in the variables $x_i$ for $i \in \bZ$ with coefficients in $\bZ$ that are:
\begin{itemize}
\item of bounded total degree: for each $f \in \backR$, there is an $M$ such that the degree of every term of $f$ is at most $M$,
\item bounded above: for each $f \in \backR$, there is an $N$ such that no term of $f$ contains $x_i$ for $i > N$,
\item back-symmetric: for each $f \in \backR$, there is a $b$ such that $s_i(f) = f$ for all $i < b$.
\end{itemize}

\begin{proposition}[\protect{\cite[Theorem 3.5 and Remark 2.8]{Back-stable-paper}}]
The ring $\backR$ equals $\bZ[\dotsc, x_{-1}, x_0, x_{1}, \dotsc, c_1, c_2, \dotsc]$, where $c_k = \sum_{i_1 < \dotsb < i_k \le 0} x_{i_1}\dotsm x_{i_k}$. The back-stable Schubert polynomials form a $\bZ$-basis for $\backR$.
\end{proposition}

Since the back-stable Schubert polynomials form a basis for $\backR$, there are \emph{back-stable Schubert structure constants} $\bsc\pi\rho\sigma$ such that 
\begin{equation}\label{eq:back-stable-multiplication}
\bSchub_\pi \bSchub_\rho = \sum_{\sigma}\bsc\pi\rho\sigma \bSchub_\sigma.
\end{equation}
The following propositions tells us that the problem of multiplying the back-stable Schubert polynomials and expanding them in the basis of back-stable Schubert polynomials is exactly the same as the corresponding problem for Schubert polynomials.
\begin{proposition}[\protect{\cite[Proposition 3.20]{Back-stable-paper}}]\label{thm:back-stable-constants-are-regular-constants}
For $\pi$, $\rho$ and $\sigma$ in $S_\bZ$ and $k \in \bZ$, $\bsc\pi\rho\sigma = \bsc{\tau^k\pi}{\tau^k\rho}{\tau^k\sigma}$.
If $\pi$, $\rho$ and $\sigma$ lie in $S_\infty$, then $\bsc\pi\rho\sigma = \ssc\pi\rho\sigma$. If $k$ is large enough, so that $\tau^k\pi$, $\tau^k\rho$ and $\tau^k\sigma$ all lie in $S_\infty$, then $\bsc\pi\rho\sigma = \ssc{\tau^k\pi}{\tau^k\rho}{\tau^k\sigma}$.
\end{proposition}
So a combinatorial rule for multiplying Schubert polynomials gives us a combinatorial rule for multiplying back-stable Schubert polynomials and vice versa.

The main reason we are considering back-stable Schubert polynomials instead of the Schubert polynomials is the following theorem.
\begin{theorem}[\protect{\cite[Proposition 3]{Nenashev}}]\label{thm:Nenashev-shuffling}
  Given permutations $\pi$ and $\rho$ in $S_\bZ$,
\[
\binom{\len(\pi) + \len(\rho)}{\len(\rho)}\size{\RW(\pi)}\size{\RW(\rho)} = \sum_{\sigma \in S_\bZ} \bsc\pi\rho\sigma\size{\RW(\sigma)}.
\]
\end{theorem}

The left hand side counts the number of shuffles of a word for $\pi$ and a word for $\rho$ (keeping track of which word each letter came from). 
The right hand side counts the number of words for the permutations $\sigma$ (with multiplicity) appearing on the right hand side of \autoref{eq:back-stable-multiplication}. 
The equality implies that there should be a rule for taking a shuffle of a word for $\pi$ and a word for $\rho$ and producing a word $S$ for $\sigma$ in the product. Further, the fiber over $S$ should have size $\bsc\pi\rho\sigma$, independent of $S$, similar to the jeu-de-taquin rule for computing Littlewood--Richardson coefficients.

\begin{example}\label{ex:shuffling}
  Let $\pi = [21]$ and $\rho = [321]$, with $\RW([21]) = \{1\}$ and $\RW([321]) = \{121, 212\}$. So there are $8$ shuffles of a word for $[21]$ and a word for $[321]$: $\leftword{1}121$, $1\leftword{1}21$, $12\leftword{1}1$, $121\leftword{1}$, $\leftword{1}212$, $2\leftword{1}12$, $21\leftword{1}2$ and $212\leftword{1}$, where $\leftword{1}$ is the $1$ from the reduced word for $[21]$.

The product of the back-stable Schubert polynomials for $[21]$ and $[321]$ is
\[
\bSchub_{[21]} \bSchub_{[321]} = \bSchub_{[4213]} + \bSchub_{\sigma_1} + \bSchub_{\sigma_2},
\]
where $\sigma_1 = \inv\tau[3412]
$
and $
\sigma_2 = \inv\tau[2431]$. 
Now $\RW([4213]) = \{1321, 3121, 3212\}$,  $\RW(\sigma_1) = \{1021, 1201\}$ and $\RW(\sigma_2) = \{0121, 0212, 2012\}$. So there are $8$ reduced words for the permutations that appear in the product. 
In \autoref{ex:monks-formula} we will transform the shuffles into the reduced words.
\end{example}

\subsection{Slide and glide polynomials}\label{sec:slide}
Assaf and Searles \cite{Slide-paper} introduced the fundamental slide polynomials and expanded Schubert polynomials in them, similar to the expansion of Schur polynomials into a sum of fundamental quasisymmetric polynomials. 
They also introduced monomial slide polynomials, which we will not consider, and we will therefore refer to the fundamental slide polynomials as simply slide polynomials.

Just like with the fundamental quasisymmetric polynomials, we give a definition in terms of tableaux.
In \autoref{prop:assaf-searles-slide-definition}, we show that it is equivalent to the original definition.
\begin{definition}\label{def:slide-new}
  Given a weak composition $\lambda$, the \emph{slide polynomial} $\slide_\lambda$ is defined by
\[
\slide_\lambda(x_1, \dotsc, x_n) \defeq \sum_{T \in \WCT(\lambda)} x^T.
\]
\end{definition}

\begin{example}
The weak composition tableaux of shape $(3,0,2,2)$ are shown in \autoref{fig:weak-composition-tableaux-slide}, so 
\[
\slide_{(3,0,2,2)}(x_1, x_2, x_3, x_4) = x_1^3x_2^2x_3^2 + x_1^3x_2^2x_3x_4 + x_1^3x_2^2x_4^2 + x_1^3x_2x_3x_4^2 + x_1^3x_3^2x_4^2. \qedhere
\]
\end{example}
\begin{figure}
  \centering
  \includestandalone{weak-composition-tableaux-slide}
  \caption{All weak composition tableaux of shape $(3,0,2,2)$.}
  \label{fig:weak-composition-tableaux-slide}
\end{figure}

\begin{proposition}[\protect{\cite[Definition 3.6]{Slide-paper}}]\label{prop:assaf-searles-slide-definition}
  Given a weak composition $c$, the slide polynomial $\slide_c$ is given by
\[
\slide_c(x_1, \dotsc, x_n) = \sum_{\substack{d \\ \text{$d$ dominates $c$, and}\\\text{$\flatt(d)$ refines $\flatt(c)$}}}x^d. 
\]
\end{proposition}
\begin{proof}
  This is \autoref{lemma:weak-composition-tableaux-content}.
\end{proof}
Searles \cite[Proposition 2.9]{Skyline-slides-paper} notes that they can be defined in terms of certain fundamental semi-skyline fillings, where the conditions defining semi-skyline fillings turn out to be automatic, and which turn out to be the same as our weak composition tableau (possibly rotated or reflected).

From our \autoref{def:slide-new}, we see easily that the fundamental quasisymmetric polynomial $\quasi_\lambda(x_1, \dotsc, x_n)$ for a composition $\lambda$ equals the slide polynomial $\slide_{(0^n,\lambda)}(x_1, \dotsc, x_n)$, where $(0^n,\lambda)$ is the weak composition which starts with $n$ zeros, followed by $\lambda$, since in $\WCT_n(0^n, \lambda)$, the restrictions on which entries can be in which row hold automatically. 

However, slide polynomials are in general not quasisymmetric. For example, $\slide_{(1,0,1)} = x_1x_2 + x_1x_3$ is not quasisymmetric since it does not include a $x_2x_3$-term.  

\begin{proposition}[\protect{\cite[Theorem 3.13]{Slide-paper}}]
\label{prop:schubert-sum-of-slides}
  For a permutation $\pi$ in $S_\infty$,
\[
\Schub_\pi = \sum_{P \in \QPD(\pi)}\slide_{\wt(P)},
\]
where $\QPD(\pi)$ is the set of \qY/ reduced pipe dreams for $\pi$.
\end{proposition}

The slide polynomials form a basis for the ring of polynomials \cite[Theorem 3.9]{Slide-paper} and there is a positive rule \cite[Theorem 5.11]{Slide-paper} for expanding the product of two slide polynomials in the basis of slides. 
However, there is no known positive rule for reassembling the resulting slide polynomials into Schubert polynomials. If there were such a rule, it would lead to a positive rule for expanding the product of two Schubert polynomials as a sum of Schubert polynomials.

By \autoref{lemma:quasi-yamanouchi-to-word}, \qY/ reduced pipe dreams correspond to reduced words that have compatible sequences. So for $w$ a reduced word for $\pi$, we define $\slide_w \defeq \slide_{\wt (P)}$ if there is a \qY/ pipe dream $P$ with reading word $w$ and $0$ otherwise. Then,
\begin{equation}\label{eq:schubert-slide-expansion}
\Schub_\pi = \sum_{w \in \RW(\pi)}\slide_w.
\end{equation}

\label{sec:slides-for-words}
Since there are more reduced words than weak compositions, this means that the slide polynomials for two different words can be the same polynomial.

\begin{example}
  The reduced words $21$ and $31$ both have the same slide polynomial: $\slide_{21} = \slide_{31} = \slide_{(2)} = x_1^2$. 
\end{example}
The decomposition of a Schubert polynomial as a sum of slide polynomials is not multiplicity free, so there are permutations $\pi$ and reduced words $w$ and $w'$ for $\pi$ such that $\slide_w = \slide_{w'}$.
\begin{example}
  Let $\pi = [246135]$ and $w = 315243$ and $w' = 513243$, then $\slide_w = \slide_{w'} = \slide_{(2,2,2)} = x_1^2x_2^2x_3^2$. The \qY/ pipe dreams for these two words are shown in \autoref{fig:pipe-dream-picture} in \autoref{sec:pipe-dreams}.
\end{example}
\begin{theorem}[Restatement of \autoref{lemma:slide-complex-ball-or-sphere}]\label{thm:schubert-slide-ball-decomposition}
The pipe dream complex $\Delta(Q, \pi)$ can be decomposed into balls or spheres, with the balls or spheres corresponding to the slide polynomials appearing in the expansion
\[
\Schub_\pi = \sum_{w \in \RW(\pi)} \slide_w.
\]
\end{theorem}

\begin{figure}
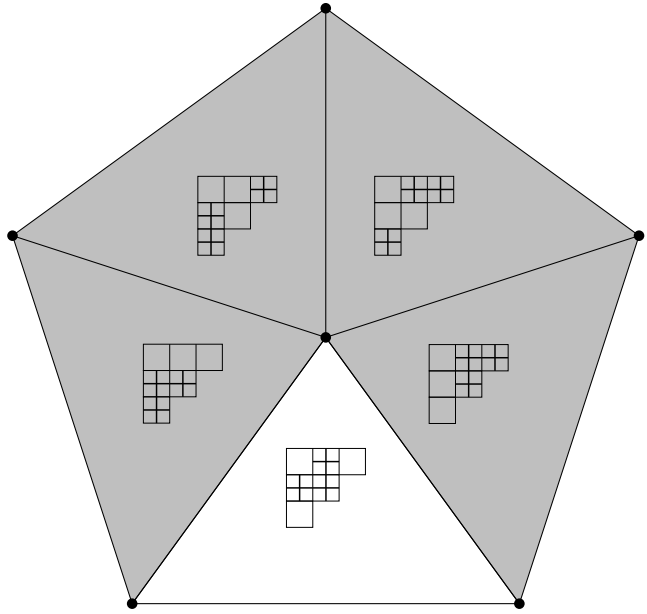
\centering
\includestandalone{subword-complex-1432-pd}
\caption[A pipe dream complex with slide complexes marked.]{The pipe dream complex for $\pi = [1432]$, with the slide complex for $323$ in gray and the slide complex for $232$ at the bottom in white.}
\label{fig:subword-complex-1432}
\end{figure}

\begin{example}\label{ex:different-simplicial-complexes}
The decomposition of the pipe dream complex for $[1432]$ and its decomposition into the slide complexes for $323$ and $232$ is shown in \autoref{fig:subword-complex-1432}.

Combining \autoref{thm:weak-composition-tableau-complex} and \autoref{def:slide-new}, we get that there is a tableau complex corresponding to each slide polynomial. However, in general, they do not fit into the pipe dream complex. The tableau complex corresponding to that slide, $\tabcomplex{\WCT_3((0,2,1)}$, shown in \autoref{fig:tableau-complex-does-not-fit}, does not fit into the pipe dream complex $[1432]$. 
This is due to the fact that the facets $\dd\dd1\dd23$ and $321\dd\dd\dd$ share a vertex in the slide complex $323$, while the corresponding facets, which correspond to the tableaux with content $(2,1,0)$ and $(0,2,1)$, do not.

This difference in adjacency between subword complexes and tableau complexes has also been noted in the case of Grassmannian permutations \cite[Remark 2.11]{Tableau-complexes-paper}.
\end{example}

\begin{figure}\centering
\includestandalone{tableau-complex-slides}
\caption{The tableau complex $\tabcomplex{\WCT_3((0,2,1))}$.}
\label{fig:tableau-complex-does-not-fit}
\end{figure}
Glide polynomials were defined by Pechenik and Searles \cite{Glide-paper} and are a generalization of slide polynomials similar to the generalization from Schubert polynomials to Grothendieck polynomials.

We again give our own definition in terms of tableaux, and prove that it is equivalent to the definition in \cite{Glide-paper}.
\begin{definition}\label{def:glide-polynomial}
Given a weak composition $\lambda$, the \emph{glide polynomial} $\glide_\lambda$ is defined by
\[
\glide_\lambda(x_1, \dotsc, x_n) \defeq \sum_T (-1)^{\size{T} - \size{\lambda}}x^T,
\]
where the sum is over all set-valued weak composition tableaux of shape $\lambda$.
\end{definition}

\begin{proposition}[\protect{\cite[Definition 2.5]{Glide-paper}}]
  Given a weak composition $\lambda$, the \emph{glide polynomial} $\glide_c$ is given by 
\[
\glide_c = \sum_{\text{$d$ is a glide of $c$}}(-1)^{\excess(d)}x^d.
\]
\end{proposition}
\begin{proof}
  This is \autoref{lemma:set-valued-weak-composition-tableaux-kontent}.
\end{proof}

\begin{theorem}[\protect{\cite[Theorem 2.14]{Glide-paper}}]
  Given a permutation $\pi$, the Grothendieck polynomial $\Groth_\pi$ expands into glides
\[
  \Groth_\pi = \sum_{P}(-1)^{\excess(P)} \glide_{\wt(P)},
\]
where the sum is over all \qY/ pipe dreams for $\pi$, not just the reduced ones.
\end{theorem}

Just like with slide polynomials, we can lift the definition to words by setting $\glide_w \defeq \glide_{\wt(P)}$ if there is a \qY/ pipe dream $P$ with reading word $w$ and $0$ otherwise. 

\section{Gr\"obner theory}
Here we define the Gr\"obner theory we need to in order to start tying things together.
\subsection{Basics}
Let $R = \bk[x_1, \dotsc, x_r]$ be a polynomial ring. A \emph{term order} (or \emph{monomial order}) is total order $\succ$ on the monomials in $R$ such that 
\begin{itemize}
\item if $m_1 \succ m_2$, then $m_1n \succ m_2n$ for every monomial $n$ in $R$,
\item $1 \nsucc m$ for every monomial $m$ in $R$.
\end{itemize}
The order extends to terms by defining $c_1 m_1 \succ c_2 m_2$ if $m_1 \succ m_2$ and $c_1$ and $c_2 $ are nonzero scalars. 
The \emph{initial term} $\init_{\succ} g$ of $g \in R$ is the maximal term of $g$ and the \emph{initial ideal} of an ideal $I$ is the ideal $\init_\succ I \defeq \idealbuilder{\init_\succ g}{g \in I}$.
Examples of term orders include the lexicographic order, where $x^a \succ x^b$ if the first non-zero coordinate of $a - b$ is positive, and reverse lexicographic order, where $x^a \succ x^b$ if $\size{a} > \size{b}$ or the last non-zero coordinate of $a-b$ is negative. 

Given a vector $w \in \bR^r_{\ge 0}$, called a \emph{weight vector}, the \emph{weight order} $\succ_w$ defined by $w$ is defined by $x^a \succ_w x^b$ if $w \cdot a > w \cdot b$.
The \emph{initial form} $\init_{\succ_w}g$ is the sum of all maximal terms of $g$ and the \emph{initial ideal} $\init_{\succ_w}I$ is the ideal generated by the initial forms of the elements of $I$.
Generally, weight orders are not total orders, however, for any term order $\succ$ and any ideal $I$, there is a weight $w \in \bR^r_{\ge0}$ such that $\init_\succ I = \init_{\succ_w} I$ \cite[Proposition 1.11]{Sturmfels-Grobner-bases-book}.
A weight order $\succ_w$ is \emph{compatible} with a (finer) order $\succ$ if $m_1 \succ_w m_2$ implies $m_1 \succ m_2$. In particular, this implies that $\init_\succ g = \init_\succ(\init_{\succ_w}g)$.

In the polynomial ring where the variables are indeterminates in an $n \times n$ matrix, $S = \bk[z_{11}, \dotsc, z_{nn}]$, a term order $\succ$ is \emph{antidiagonal} if, for every minor of the $n \times n$ matrix $M$ with $(M)_{ij} = z_{ij}$, the initial term is the antidiagonal term. 
A weight order is \emph{antidiagonal} if it is compatible with an antidiagonal term order.

\begin{example} \label{ex:antidiagonal-orders}
  Any revlex order where the variables in each row and column are decreasing, that is: $z_{i, j_1} \succ z_{i, j_2}$ when $j_1 < j_2$ and $z_{i_1, j} \succ z_{i_2, j}$ when $i_1 < i_2$ is antidiagonal. 
\end{example}
A set of generators $\{ g_1, \dotsc, g_m \}$ for an ideal $I$ is a \emph{Gr\"obner basis} of $I$ with respect to the order $\succ$ if  $\init_\succ I = \langle \init_\succ g_1, \dotsc, \init_\succ g_m\rangle$. 
One of the fundamentals of the theory is Buchberger's algorithm, which extends a finite generating set for $I$ to a Gr\"obner basis.

Given a weight vector $w$ and an ideal $I = \langle g_1, \dotsc, g_m \rangle$.  
For any $g \in R$, let $b$ be the maximum weight of a monomial in $g$ and let 
\[
\tilde{g} = t^bg(t^{-w_1}x_1, \dotsc, t^{-w_r}x_r) \in R[t],
\] 
and $\tilde{I} = \langle \tilde{g}_1, \dotsc, \tilde{g}_n\rangle$. Then, $R[t]/\tilde{I}$ defines a flat family whose general fiber is $R/I$ and whose special fiber is $R/ \init_{\succ w} I$ \cite[Theorem 15.17]{Eisenbud}.
We therefore call $\init_{\succ_w} I$ a \emph{Gr\"obner degeneration} of $I$.

\subsection{Polyhedra and the antidiagonal cone}
A reference for this section is Chapter 2 of \cite{Sturmfels-Grobner-bases-book}. 

A \emph{polyhedron} is the intersection of finitely many half spaces in $\bR^n$. A matrix $A$, and a vector $b$, define a polyhedron $P = \{\,v \st Av \ge b\,\}$.
If $b = 0$, the polyhedron is a \emph{cone}. If $P$ is a cone, there are vectors $R_i$ such that every $v$ in $P$ can be written of the form $v = \sum_i a_iR_i$, with all $a_i \ge 0$. 
The intersection of a cone $P = \{\,v \st A v \ge 0\,\}$ with its negative $-P = \{\,v \st -Av \ge 0\,\}$ is a vector space called the \emph{lineality space} of $P$. 
Letting $L_i$ be basis vectors for the lineality space of $P$, and only including those $R_i$ that are not in the lineality space of $P$, we can write each vector $v$ in $P$ as
\[
v = \sum_i a_iR_i + \sum_i b_iL_i.
\]
The $R_i$ are called \emph{rays}, and the $L_i$ are called \emph{lines}.

A subset $F$ of a polyhedron $P$ is called a \emph{face} if it maximizes some linear functional $u$: $F = \{\,v \in P \st \text{$u \cdot v \ge u \cdot w$ for all $w$ in $P$}\,\}$.
When $P$ is a cone, if a linear functional has a maximum on $P$, that maximum has to be equal to $0$.
So the faces are of the form $F = \{\,v \in P \st u \cdot v = 0\,\}$, where $u \cdot w \le 0$ for all $w \in P$. Adding a row $u$ and a row $-u$ to the matrix $A$ and two zeros to $b$, we see that the faces of a cone are again cones. The \emph{interior} of a polyhedron $P$ is the set of $w \in P$ that do not lie in any of its proper faces.

Given a weight vector $w$, an ideal $I$ and a Gr\"obner basis $\{ g_1, \dotsc, g_k \}$ for $I$ with respect to $\succ_w$. The set of weight vectors $w'$ such that $\init_{\succ_w} I = \init_{\succ_{w'}}$ form an open cone $C[w]$: it is defined by inequalities $w' \cdot a = w' \cdot b$ when $x^a$ and $x^b$ are terms in $\init_w g_i$ and $w' \cdot a > w' \cdot c$ when $x^c$ is not.  Its closure $\overbar{C[w]}$ contains all weight vectors $w'$ that are compatible with $w$, and there is a conservation of degeneracy: the more degenerate $w'$ is, the less it degenerates $I$, and vice versa. If $w' \in \overbar{C[w]}$ and $F$ is the maximal face of $\overbar{C[w]}$ containing $w'$, then $F = C[w']$, so all the weight vectors in $F$ define the same Gr\"obner degeneration of $I$. Conversely, if $w' \notin C[w]$, by definition of $C[w]$, $w$ and $w'$ define different Gr\"obner degenerations of $I$. That is, the Gr\"obner degenerations of $I$ using to weight orders compatible with $\succ_w$ are in one-to-one correspondence with the faces of $C[w]$.
When $I$ is a determinental ideal in $\bk[z_{11}, \dotsc, z_{nn}]$ and $\succ_w$ an antidiagonal term order, $\overbar{C[w]}$ is the \emph{antidiagonal cone} of $I$.

From a slightly different perspective, the antidiagonal cone is the part of the Gr\"obner fan that consists of the normal cones of a neighborhood of the antidiagonal initial ideal in the state polytope of $I$.

\section{Geometry}
Here we tie the previously defined objects and polynomials to the geometry of matrix Schubert varieties.
\subsection{Matrix Schubert varieties}
Let $\pi$ be a permutation in $S_n$. The \emph{permutation matrix} $M_\pi$ is the matrix defined by $(M_\pi)_{ij} = 1$ if $\pi(i) = j$ and $0$ otherwise. 
Given a matrix $A$, denote by $A_{[p,q]}$ the $p\times q$ matrix in the north-west corner of $A$. The \emph{matrix Schubert variety} $\mSchub\pi$ is the variety of all $n \times n$ matrices $A$ satisfying $\rank A_{[p,q]} \le \rank (M_\pi)_{[p,q]}$ for all $p$ and $q$. 
The ideal $I_\pi$ generated by these equations is prime \cite{Fulton-92}.

The \emph{Rothe diagram} of $\pi$ consists of an $n\times n$ square of boxes, with dots in the boxes $(i, \pi(i))$ and the boxes to the right and below the dots crossed out.
The boxes that contain neither a line nor a dot form Young diagrams (for partitions), and the south-eastern-most boxes in these Young diagrams make up the \emph{Fulton essential set} $\Ess(\pi)$. See \autoref{fig:rothe-diagram-example} for an example. 
The rank conditions $\rank A_{[p,q]} \le \rank (M_\pi)_{[p,q]}$ for $(p,q)$ in the Fulton essential set imply all the other rank conditions that define $\mSchub\pi$ \cite[Lemma 3.10]{Fulton-92}.
Thus, the ideal $I_\pi$ of $\mSchub\pi$ in $\bk[z_{11}, \dotsc, z_{nn}]$ is generated by the minors of $A_{[p, q]}$ of size $\rank{(M_\pi)_{[p,q]}} + 1$  for $(p, q)$ in the Rothe diagram of $\pi$.
Since the dots in the Rothe diagram are where the ones are in $M_\pi$, these rank conditions are of the form ``$\rank A_{[p, q]}$ is less than or equal to the number of dots to the north-west of $(p,q)$ in the Rothe diagram of $\pi$.'' 
\begin{example}
\begin{figure}
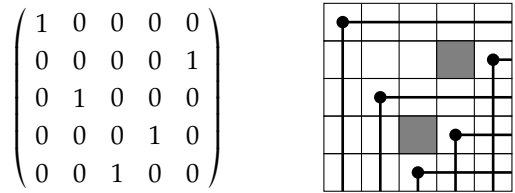
\centering
\includestandalone{rothe-diagram}
\caption[The permutation matrix and Rothe diagram for $\sqbl15243\sqbr$.]{The permutation matrix and Rothe diagram for $[15243]$. The Fulton essential set is marked with gray.}
\label{fig:rothe-diagram-example}
\end{figure}
The Rothe diagram for $[15243]$ is given in \autoref{fig:rothe-diagram-example}, so $\mSchub{[15243]}$ is given by $\rank A_{2,4} \le 1$ and  $\rank A_{4,3} \le 2$. The ideal of $\mSchub{[15243]}$ is 
\begin{align*}
 I_{[15243]} = \langle
&
  \vertMatrixstack{z_{11} & z_{12} \\ z_{21}  & z_{22}
  },
  \vertMatrixstack{z_{11} & z_{13} \\ z_{21}  & z_{23}
  },
  \vertMatrixstack{z_{11} & z_{14} \\ z_{21}  & z_{24}
  },
  \vertMatrixstack{z_{12} & z_{13} \\ z_{22}  & z_{23}
  },
  \vertMatrixstack{z_{12} & z_{14} \\ z_{22}  & z_{24}
  },
  \vertMatrixstack{z_{13} & z_{14} \\ z_{23}  & z_{24}
  },\\ 
  &
  \vertMatrixstack{z_{11} & z_{12} & z_{13} \\ z_{21}  & z_{22} & z_{23} \\ z_{31} & z_{32} & z_{33}
  }, 
  \vertMatrixstack{z_{11} & z_{12} & z_{13} \\ z_{21}  & z_{22} & z_{23} \\ z_{41} & z_{42} & z_{43}
  },                                                                              
  \vertMatrixstack{z_{11} & z_{12} & z_{13} \\ z_{31}  & z_{32} & z_{33} \\ z_{41} & z_{42} & z_{43}
  },  
  \vertMatrixstack{z_{21} & z_{22} & z_{23} \\ z_{31}  & z_{32} & z_{33} \\ z_{41} & z_{42} & z_{43}
  }  
 \rangle. \qedhere
\end{align*}
\end{example}

\subsection{Homological invariants}\label{sec:geometry}
A reference for this section, where things are done in more generality, is \cite{Miller-Sturmfels}. See also \cite{Woo-Yong-Survey-paper} for a more specific account, as well as a more recent survey of related results.

A polynomial ring $R = \bk[z_1, \dotsc, z_r]$ is \emph{multigraded by $\bZ^n$} if it can be written as a direct sum $R = \bigoplus_{\ba\in \bZ^n} R_\ba$, with $R_\ba R_\bb \subseteq R_{\ba+\bb}$. 
If $R_\bO = \bk$, the grading is \emph{positive}. 
An $R$-module $M$ is \emph{multigraded} if it can be written as a direct sum $M = \bigoplus_{\ba \in \bZ^n} M_\ba$, with $R_\ba M_\bb \subseteq M_{\ba + \bb}$.

If $M$ is finitely generated and $R$ is positively graded, then the vector space dimension of each $M_\ba$ over $\bk$ is finite \cite[Theorem 8.6]{Miller-Sturmfels}. The \emph{Hilbert function} of $M$ is the function $\ba \mapsto \dim_\bk(M_\ba)$. 
and the \emph{Hilbert series} of $M$ is the Laurent series $\sH(M;\bx) = \sum_{\ba \in A}\dim_\bk(M_\ba)\bx^\ba$. 
\begin{theorem}[\protect{\cite[Theorem 8.20]{Miller-Sturmfels}}]
If $M$ is a finitely generated graded module and $R$ is positively graded, there is a Laurent polynomial $\sK(M;\bx)$, the \emph{$K$-polynomial} of $M$, such that 
\[
\sH(M;\bx) = \frac{\sK(M;\bx)}{\prod_{\ba \in \bZ^n}(1 - \bx^\ba)}.
\]
\end{theorem}
The \emph{multidegree} of $M$, $\sC(M;\bx)$, is the sum of all terms of $\sK\lBrack M; \bone - \bx\rBrack$ of degree $r - \dim(M)$. 

To compute $\sK(R/I, \bx)$ and $\sC(R/I, \bx)$, we can use the following three propositions.
\begin{proposition}\cite[Theorem 8.36]{Miller-Sturmfels}\label{thm:k-polynomial-degenerative}
The $K$-polynomial is invariant under degeneration: given a weight vector $w$, 
\[
\sK(R/I, \bx) = \sK(R/\init_w(I), \bx).
\]
\end{proposition}
\begin{proposition}\cite[Theorem 8.53]{Miller-Sturmfels}\label{thm:c-polynomial-additive}
  The multidegree $\sC(R/I;\bx)$ is additive: 
\[
\sC(R/I; \bx) = \sum_{k = 1}^n \mult_{\fp_k}(R/I) \cdot \sC(R/\fp_k),
\]
where $\fp_1$, \ldots, $\fp_n$ are the maximal dimensional associated primes of $I$ and $\mult_{\fp_k}$ is the multiplicity of $\fp_k$ in $R/I$ (in the case we are interested in, the multiplicities are all $1$). 
\end{proposition}
\begin{proposition}\cite[Proposition 8.49]{Miller-Sturmfels}\label{thm:c-polynomial-coordinate-hyperplane}
If $I = \langle z_{i_1}, \dotsc, z_{i_k} \rangle$, then
\[
\sC(R/I; \bx) = \langle \ba_{i_1},\bx\rangle \dotsm \langle \ba_{i_k}, \bx \rangle,
\]
where $\ba_{i_j} = \deg(z_{i_j})$ and $\langle \ba, \bx \rangle = a_1 x_1 + \dotsb + a_r x_r$.
\end{proposition}
To tie everything togehter, we apply this to the ring $S = \bk[z_{11}, \dotsc, z_{nn}]$, $\bZ^{n}$-graded by matrix rows. 
\begin{proposition}[\protect{\cite[Theorem B]{Grobner-geometry-paper}}]\label{thm:to-pipe-dreams}
If $I_\pi$ is the ideal of a matrix Schubert variety $\mSchub{\pi}$, and $\succ$ is an antidiagonal term order, then 
\[
\init_\succ I_\pi = \bigcap_{P \in \PD(\pi)} \idealbuilder{z_{ij}}{\text{$(i,j)$ is a cross in $P$}}.
\] 
\end{proposition}
\label{sec:schubert-polynomial-multidegree}
So, combining Propositions \ref{thm:k-polynomial-degenerative}, \ref{thm:c-polynomial-additive}, \ref{thm:c-polynomial-coordinate-hyperplane}, and \ref{thm:to-pipe-dreams}, the multidegrees of $S/I_\pi$ is the Schubert polynomial of $\pi$: 
\[
\sC(S/I_\pi; \bx) = \sC(S/\init I_\pi; \bx) =  \sum_{P \in \PD(\pi)} x^P = \Schub_\pi(\bx).
\]
Further, $\sK(S/I_\pi; \bone - \bx)$ is the Grothendieck polynomial $\Groth_\pi$ \cite[Theorem A]{Grobner-geometry-paper}.

\chapter{Results}
Here we present our main results. They are all at least loosely inspired by the expansion of Schubert polynomials into slide polynomials.
\section{Bijections}
\label{sec:multiplying-slides}
By \autoref{thm:Nenashev-shuffling}, there exists a bijection between shuffles of words for $\pi$ and $\rho$ and words for the $\sigma$s in the product $\Schub_{\pi} \Schub_{\rho} = \sum c_{\pi\rho}^\sigma\Schub_\sigma$ (with multiplicity, and dealing appropriately with non-positive letters).
We give such bijections in the cases of Monk's rule and Sottile's Pieri rule (which are both multiplicity-free). The bijections are similar to Little's bumping bijection that was used to give a bijective proof of the Schur positivity of Stanley symmetric functions \cite{Little-bump-paper} and Macdonald's reduced word identity \cite{Billey-Holroyd-Young-Little-bump-paper}.

Bijective proofs of Monk's rule and Sottile's Pieri rule, as well as some other cases already exist. Sottile \cite[Theorem 1]{Sottile-Pieri-rule-paper} gave explicit formulas for the cases we are considering, essentially working on the level of permutations. On the other hand, Bergeron and Billey \cite[Section 5]{RC-graphs-and-schubert-polynomials} found a bijection for Monk's rule at the level of pipe dreams, and Kogan and Kumar \cite[Section 3]{Kogan-Kumar} found a bijection for Sottile's Pieri rule, also at the level of pipe dreams.
Our bijections at the level of reduced words are in between these two levels. There is more freedom than at the permutation level but also more structure than at the level of pipe dreams.
\subsection{Monk's rule}\label{sec:monk's-rule}
Monk's rule expands the product $\Schub_{\pi}\Schub_{s_i}$ as a sum of Schubert polynomials. Monk \cite[Theorem 3]{Monk} gave the geometric version of the rule.
\begin{theorem}
For $\pi \in S_\bZ$ and $i \in \bZ$,
\[
\Schub_\pi\Schub_{s_i} = \sum_{\substack{a \le i < b\\\len(\pi t_{ab}) = \len(\pi) + 1}}\Schub_{\pi t_{ab}}.
\]
\end{theorem}

Fix $i$ and let $w = a_1 \dots a_n$ be a reduced word for $\pi$ and let $1 \le j \le n + 1$. We will present an algorithm (\autoref{alg:monk-shuffling}) for inserting $i$ into $w$ at position $j$ and rectifying the word to get a reduced word $w'$ for some $\pi \cdot t_{ab}$ with $a \le i < b$. We also give the inverse algorithm (\autoref{alg:monk-unshuffling}) for taking a reduced word $w'$ for some $\pi \cdot t_{ab}$ with $a \le i < b$ to a pair $w$ and $j$. So this algorithm gives a bijection between shuffles of $i$ into words for $\pi$ and words for $\pi \cdot t_{ab}$ with $a \le i < b$. 

\subsubsection{The algorithm and its inverse}
\begin{algorithm}
  \linespread{2}\selectfont
  \begin{algorithmic}[1]
    \item[]
    \Input $i$, a position $1 \le j \le n+1$, a word $w = a_1 \dotsm a_{n+1}$ such that $a_j = \infty$ and $a_1 \dotsm \widehat{a}_j \dotsm a_{n+1}$ is a reduced word for $\pi$.
    \Output a reduced word for $\pi t_{ab}$ where $a \le i < b$
    \While{$w$ is not reduced} \label{alg-line:monk-while}
      \State $j \gets \text{leftmost defect in $w$}$ \label{alg-line:monk-while-start}
      \State $w[j] \gets \max \setbuilder{k}{\text{$k < w[j]$ and $\wl{w}{j}{k} \le i < \wl{w}{j}{k+1}$}}$ \label{alg-line:monk-move-down} \label{alg-line:monk-while-end}
    \EndWhile
    \State \Return $w$
  \end{algorithmic}
  \caption{Shuffling $i$ into $w$ at position $j$.}
  \label{alg:monk-shuffling}
\end{algorithm}

\begin{algorithm}
  \linespread{2}\selectfont
  \begin{algorithmic}[1]
    \item[]
    \Input $i$, a reduced word $w = a_1 \dotsm a_{n+1}$ for $\pi t_{ab}$ where $a \le i < b$.
    \Output a word $b_1 \dotsm b_{n+1}$, such that there is a unique $j$ with $b_j = \infty$ and $b_1 \dotsm \widehat{b}_j \dotsm b_{n+1}$ is a reduced word for $\pi$.
    \State $j \gets k$ where $a_k$ is the cross $(a,b)$
    \While{$w[j] \neq \infty$}
      \State $w[j] \gets \min \setbuilder{k}{\text{$k > w[j]$ and $\wl{w}{j}{k+1} \le i < \wl{w}{j}{k}$}}$ \Comment{\makebox[5em][l]{$\min \emptyset = \infty$}} 
      \If{$w[j] \neq \infty$}
        \State $j \gets \text{rightmost defect of $w$}$
      \EndIf
    \EndWhile
    \State \Return $w$
  \end{algorithmic}
  \caption{Unshuffling $i$ from a word for $\pi t_{ab}$.}
  \label{alg:monk-unshuffling}
\end{algorithm}

In terms of wiring diagrams, the algorithm is easy to understand. If $c \le i < d$, a swap that moves the $c$-wire up and the $d$-wire down is \emph{allowed}.
We start with the given word $w$ and insert an empty column between the $j$th and $j+1$st letter, with a cross infinitely far up. We move this cross down until we reach an allowed swap. Now, either the word is reduced and we are done, or the two wires that it crossed cross again (at some position to the left). If they do cross again, we move this new cross down until we reach an allowed swap and keep going until the word is reduced. 

\begin{example}
  We will go through the insertion of $i = 3$ into the word $323432$ at position $5$ using \autoref{alg:monk-shuffling} in detail. The wiring diagrams for the steps are shown in \autoref{fig:monk-example-1}. We mark the letter we are about to decrease with \downmark{} and the letter we just decreased with \upmark{}.

We begin with the word $3234\monkdown{\infty}32$, and decrease the $5$th letter. The largest allowed letter is $4$ which swaps $2$ and $5$, (for $k \ge 5$, $k$ swaps $k$ and $k+1$, which is not allowed since $k > i$). When we replace the $\infty$ with $4$, we get a non-reduced word, the $2$-wire and the $5$-wire cross again at position $4$. So after the first run through the while loop on \origautoref{alg-line:monk-while}, the word is $323\monkdown{4}\monkup{4}32$.

Now we decrease the $4$ at position $4$. The largest allowed letter below it is $2$, which swaps $3$ and $4$ ($3$ swaps $4$ and $5$, which is not allowed since $4 > i$). When we replace the $4$ at position $4$ with $2$, we still do not have a reduced word, the $3$-wire and $4$-wire cross again at position $1$. So after this run through the while loop, the word is $\monkdown{3}23\monkup{2}432$.

Finally, we decrease the $3$ at position $1$. The largest allowed letter below it is $1$, which swaps $1$ and $5$ ($2$ swaps $5$ and $4$, which is not allowed, because $5 > i$). We replace the $3$ at position $1$ with $1$ and now the word is reduced. So our final word is $\monkup{1}232432$.
\end{example}

\begin{figure}
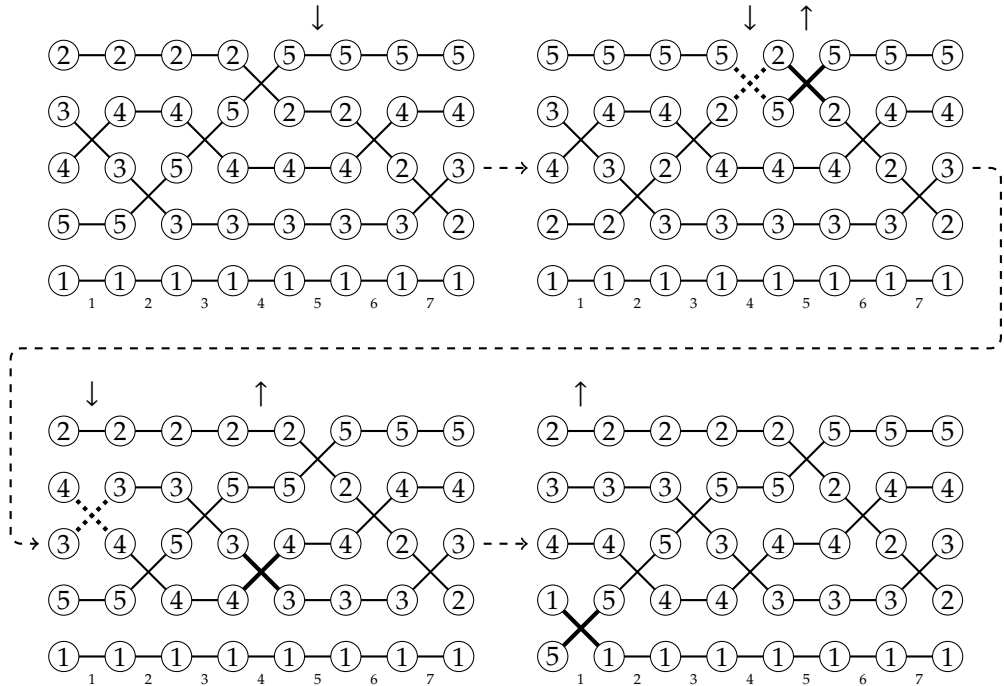

  \centering
  \includestandalone{monk-ex}
  \caption[The shuffling of $3$ into $323432$ at position $5$.]{The shuffling of $3$ into $323432$ at position $5$. At each step, the newly moved cross is bold, with an \upmark{} above it, and the next cross to move (if there is one) is dotted, with a \downmark{} above its position.}
  \label{fig:monk-example-1}
\end{figure}

\begin{example}\label{ex:monks-formula}
  In \autoref{ex:shuffling}, we saw that for $[21]$ and $[321]$, we have the following shuffles of their reduced words: $\leftword{1}121$, $1\leftword{1}21$, $12\leftword{1}1$, $121\leftword{1}$, $\leftword{1}212$, $2\leftword{1}12$, $21\leftword{1}2$ and $212\leftword{1}$. Replacing the $\leftword{1}$ by $\monkdown{\infty}$ and applying \autoref{alg:monk-shuffling}, we get:
\begin{align*}
\leftword{1}121 &\to 3121 & \leftword{1}212 &\to 3212 \\
1\leftword{1}21 &\to 1321 & 2\leftword{1}12 &\to 0212 \\
12\leftword{1}1 &\to 1021 & 21\leftword{1}2 &\to 1012 \\
121\leftword{1} &\to 1201 & 212\leftword{1} &\to 0121 \qedhere
\end{align*}
\end{example}

\subsubsection{Analysis}
To show the correctness of \autoref{alg:monk-shuffling}, we need to show that the algorithm is well defined and terminates with the correct result, and that it is a bijection.
\begin{lemma}\label{thm:monk-non-empty-set}
  The set on \origautoref{alg-line:monk-move-down} in \autoref{alg:monk-shuffling} is nonempty.
\end{lemma}
\begin{proof}
  The labeling $\wl{w}{j}{l}$ lies in $S_\bZ$, so for all but finitely many $l$, $\wl{w}{j}{l} = l$. 
  The first time we reach \origautoref{alg-line:monk-move-down}, $w[j] = \infty$, so $k < w[j]$ imposes no restriction on $k$. 
  Thus, there is an $N > 0$ such that $\wl{w}{j}{-N} = -N \le i$ and $\wl{w}{j}{N} = N > i$. This means that there must be a $k$ with $-N < k < N$ such that $\wl{w}{j}{k} \le i$ and $\wl{w}{j}{k+1} > i$.

  Every subsequent time we reach \origautoref{alg-line:monk-move-down}, we reach it because we inserted a cross to the right of $j$, which made $\wl{w}{j}{w[j] + 1} \le i$ and $\wl{w}{j}{w[j]} > i$. Since $\wl{w}{j}{-N} = -N < i$ for large $N$, there is a $k$ with $-N < k < w[j]$ such that $\wl{w}{j}{k} \le i$ and $\wl{w}{j}{k+1} > i$. 
\end{proof}

\begin{theorem}
\autoref{alg:monk-shuffling} terminates, and the output is a word for $\pi t_{ab}$ for some $a \le i < b$.
\end{theorem}
\begin{proof}
  By \autoref{thm:monk-non-empty-set}, the set on line \origautoref{alg-line:monk-move-down} is nonempty. During each run through the while loop, the leftmost defect moves further to the left, so the algorithm terminates. 
  The algorithm terminates once $w$ is reduced. Then, we have a word for $\pi$, with a single letter added, which swaps an $a \le i$ with $b > i$, and this gives a word for $\pi t_{ab}$.
\end{proof}

\begin{lemma}\label{lemma:monk-shuffle-reversible}
  A single run through the while loop (Lines~\ref{alg-line:monk-while-start}--\ref{alg-line:monk-while-end}) in \autoref{alg:monk-shuffling} can be reversed. That is, we can unbump a bumped letter and find out who bumped it down.
\end{lemma}
\begin{proof}
  In order to reverse \origautoref{alg-line:monk-while-end}, we need to find the value $l_0$ of $w[j]$ before it was decreased to $l$.

Assume that $j$ was selected because the $l_0$ formed a defect with a previously inserted letter (rather than $j$ being the location of the initial insertion). Then, $\wl{w}{j}{l_0} > i$ and $\wl{w}{j}{l_0+1} \le i$. Since $l$ was the largest allowed swap less than $l_0$, we must have $\wl{w}{j}{k} > i$ for $l + 1 \le k \le l_0$, so $l_0 + 1$ is the smallest $k \ge l + 1$ with $\wl{w}{j}{k} \le i$.

If $j$ was the position of the initial insertion, there is no $k \ge l+1$ with $\wl{w}{j}{k} \le i$. If there was such a $k$, then, since for large $N$, $\wl{w}{j}{N} = N > i$, there would be an allowed swap above $l$.

So we replace $w[j]$ by $k_0 = \min \setbuilder{k > w[j]}{\text{$\wl{w}{j}{k} > i$ and $\wl{w}{j}{k+1} \le i$}}$ (where $\min\emptyset = \infty$). If $k_0 \neq \infty$, then the value of $j$ before \origautoref{alg-line:monk-while-start} is the position of the rightmost defect that appears when we set $w[j] = k_0$.
\end{proof}

To unshuffle $i$ from a word for $\pi t_{ab}$, where $a \le i < b$, we start by setting $j$ to be the position of the letter that swaps $a$ and $b$. Then we use \autoref{lemma:monk-shuffle-reversible} repeatedly until we have a word with $\infty$ in it. The details are in \autoref{alg:monk-unshuffling}.

Since the letters of $w$ only affect what happens to letters to their left, we get the following.
\begin{proposition}\label{prop:monk-shuffle-is-back-stable}
Given $i$ and a word $w = a_1 \dotsm a_n$ with $\infty$ at one position. If $w' = a_1' \dotsm a_n'$ is the result of applying \autoref{alg:monk-shuffling} to $i$ and $w$, then $a'_2 \dotsm a'_n$ is the result of applying the algorithm to $i$ and $a_2 \dots a_n$.
\end{proposition}

\begin{remark}
We note that this shuffle rule is not the multiplication rule for slide polynomials in \cite{Slide-paper}, nor the rule in \cite{RC-graphs-and-schubert-polynomials} since those rules both respect monomials. 
For example, $\slide_{232}\slide_{2} = x_1x_2^2(x_1 + x_2) = x_1^2x_2^2 + x_1x_2^3$, but the rectifications of the shuffles of $2$ and $232$ are $4232$, $2432$, $2132$ and $2312$, corresponding to the monomials $x_1^2x_2^2$, $x_1x_2^3$, $x_1^2x_2^2$ and $0$, respectively. 
\end{remark}

\subsection{Sottile's Pieri rule}\label{sec:pieri's-rule}
Sottile discovered a version of Pieri's rule for Schubert polynomials \cite{Sottile-Pieri-rule-paper} (the formula had previously been conjectured by Bergeron and Billey \cite[Section~6]{RC-graphs-and-schubert-polynomials}).

For $i$, $k \in \bZ$ with $k > 0$, let 
\begin{align*}
c[i, k] &= s_{i - k + 1} s_{i - k + 2} \dotsm s_{i-1} s_{i},\\
r[i, k] &= s_{i+k-1} s_{i + k - 2} \dotsm s_{i+1} s_{i}.
\end{align*}
For $\pi$, $\sigma \in S_n$, write $\pi \toc{i}{k} \sigma$ if there are integers $a_1$, $b_1$, \ldots, $a_k$, $b_k$ such that
\begin{enumerate}
\item $\sigma = \pi \cdot t_{a_1b_1} \dotsm t_{a_kb_k}$, \label{item:pieri-1}
\item $a_j \le i < b_j$ and $\len(\pi \cdot t_{a_1 b_1} \dotsm t_{a_j b_j}) = \len(\pi) + j$ for $1 \le j \le k$,  \label{item:pieri-2}
\item the integers $a_1$, \ldots, $a_k$ are distinct.
\end{enumerate}
Similarly, write $\pi \tor{i}{k} \sigma$ if there are integers $a_1$, $b_1$, \ldots, $a_k$, $b_k$ satisfying \ref{item:pieri-1} and \ref{item:pieri-2} as well as 
\begin{enumerate}[label={\arabic*'.}, start=3]
\item the integers $b_1$, \ldots, $b_k$ are distinct.
\end{enumerate}

\begin{theorem}[\protect{\cite[Theorem 1]{Sottile-Pieri-rule-paper}}]\label{thm:pieri-rule}
For $\pi \in S_\bZ$, $i$, $k \in \bZ$ with $k > 0$,
\begin{align}
\bSchub_\pi \bSchub_{c[i,k]} &= \sum_{\pi \toc{i}{k} \sigma} \bSchub_\sigma \label{eq:pieri-c},\\
\bSchub_\pi \bSchub_{r[i,k]} &= \sum_{\pi \tor{i}{k} \sigma} \bSchub_\sigma \label{eq:pieri-r}.
\end{align}
\end{theorem}
We present an algorithm, \autoref{alg:pieri-shuffling}, and its inverse, \autoref{alg:pieri-unshuffling}, for computing the bijection for \autoref{eq:pieri-c}. For the bijection for \autoref{eq:pieri-r}, replace the set $B$ of ``big'' elements by a set $S$ of ``small'' elements in Algorithms \ref{alg:pieri-shuffling} and \ref{alg:pieri-unshuffling}.

\subsubsection{Some definitions}
In this section, we extend the definition of wiring diagrams (and correspondingly words), allowing two kinds of columns that do not contain a cross. 
First, we will allow columns without a cross in the wiring diagrams. In order to retain the one-to-one correspondence between words and wiring diagrams, we place the symbol $\markdown{\infty}$ at the positions without a cross. Second, we will allow columns with a deleted cross (remembering the height of the deleted cross). In the wiring diagram, we place a dot where the deleted cross was, and in the word, we place \downmark{} above the corresponding letter. The reason for the \downmark{} is that we later want to insert a cross below where the deleted cross was (in the first case at a height less than $\infty$).
In the labeling $\wl{w}{j}{k}$, we ignore the letters marked with \downmark{}.
We also make the following helpful definitions.
\begin{definition}
  Given a permutation $\pi$ and a word $Q = \sigma_n \dotsb \sigma_1$ containing $\pi$.
The \emph{rightmost subword} for $\pi$ is the subword $w_n$,
where $w_0$ is the empty word and 
\[
w_k = \begin{cases}
\sigma_k w_{k-1} & \text{if there is a word for $\pi$ ending with $\sigma_k w_{k-1}$,}\\
w_{k-1} & \text{otherwise.}
\end{cases}
\]
\end{definition}

\begin{lemma}
  The rightmost subword for $\pi$ is a word for $\pi$.
\end{lemma}
\begin{proof}
  At the start, $Q$ contains a word for $\pi$. If $\sigma_k$ is added to the word, then by induction, $\sigma_n \dotsb \sigma_k$ contained a word for $\pi  \cdot (\prod w_{k-1})^{-1}$,  so $\sigma_n \dotsb \sigma_{k+1}$ contains a word for $\pi \cdot (\prod w_{k-1})^{-1} \sigma_k^{-1} = \pi \cdot (\prod w_k)^{-1}$. 
\end{proof}

\begin{definition} For every \downmark{} at a position $j$, remove the \upmark{} from the letter $i$ to its right with $\wl{w}{i}{w[i]} = \wl{w}{j}{w[j]+1}$. 
  After this, the \emph{unmarked subword} is the subword consisting of letters that are not marked by \downmark{} or \upmark{}. 
\end{definition}
The fact that each \downmark{} has a unique letter $i$ is shown in \autoref{lemma:pieri-upper-label-unique}.

\subsubsection{The algorithm and its inverse}

\begin{algorithm}
  \linespread{2}\selectfont
  \begin{algorithmic}[1]
    \item[]
    \Input $i$, positions $1 \le j_1 < \dotsb < j_k \le n+k$, a word $w = a_1 \dotsm a_{n+k}$ such that $a_{j_1} = \dotsb = a_{j_k} = \markdown{\infty}$ and $a_1 \dotsm \widehat{a}_{j_1} \dotsm \widehat{a}_{j_k} \dotsm a_{n+k}$ is a reduced word for $\pi$.
    \Output a reduced word for $\sigma$, with $\pi \toc{i}{k} \sigma$.
    \State $B \gets \{i+1, i+2, \dotsc\}$
    \While{$w$ contains \downmark{}}
      \State $j \gets \text{position of the rightmost \downmark{} in $w$}$\label{alg-line:pieri-j-definition}
      \If{$w[j] \neq \markdown{\infty}$}
        \State remove $\wl{w}{j}{w[j] + 1}$ from $B$ \label{alg-line:pieri-B-remove} 
        \State remove \upmark{} from the $w[i]$ such that $\wl{w}{i}{w[i]} = \wl{w}{j}{w[j] + 1}$ \label{alg-line:pieri-remove-up-arrow}
      \EndIf
      \State $k \gets \max \setbuilder{k}{\text{$k < w[j]$, $\wl{w}{j}{k} \not \in B$ and $\wl{w}{j}{k+1} \in B$}}$ \label{alg-line:pieri-decrease}
      \State $w[j] \gets \markup{k}$ \label{alg-line:pieri-k-up}
      \State add $\wl{w}{j}{k}$ to $B$\; \label{alg-line:pieri-B-add}
      \If{there is an $i < j$ with $\wl{w}{i}{w[i]} = \wl{w}{j}{k}$ and $\wl{w}{i}{w[i]+1} = \wl{w}{j}{k+1}$} \label{alg-line:pieri-bump-check}
        \State mark $w[i]$ with \downmark{} \label{alg-line:move-mark-line} 
      \EndIf
    \EndWhile
    \State \Return $w$
  \end{algorithmic}
  \caption{Shuffling $i$ into $w$ at positions $1 \le j_1 \le \dotsb \le j_k \le n+k$.}
  \label{alg:pieri-shuffling}
\end{algorithm}

\begin{algorithm}
  \linespread{2}\selectfont
  \begin{algorithmic}[1]
    \item[]
    \Input $i$, a reduced word $w = a_1\dotsm a_{n+k}$ for $\sigma$.
    \Output a word $w = b_1 \dotsm b_{n+k}$ such that $b_{j_1} = \dotsb = b_{j_k} = \markdown{\infty}$ for a sequence $1 \le j_1 < \dotsb < j_k \le n+k$ and $b_1 \dotsm \widehat{b}_{j_1} \dotsm \widehat{b}_{j_k} \dotsm b_{n+k}$ is a reduced word for $\pi$. 
    \State Mark the letters not in the rightmost subword of $w$ representing $\pi$ with \upmark{}
    \State $B \gets \setbuilder{\wl{w}{j}{w[j]}}{\text{$w[j]$ is marked with \upmark{}}} \cup \{i + 1, i + 2, \dotsc\}$  
    \While{$w$ contains \upmark{}}
      \State $j \gets \text{position of the leftmost \upmark{}}$
      \If{there is an $i < j$ with $\wl{w}{i}{w[i]} = \wl{w}{j}{w[j]+1}$ and $\wl{w}{i}{w[i] + 1} = \wl{w}{j}{w[j]}$}
        \State remove \downmark{} from $w[i]$
      \EndIf
      \State remove $\wl{w}{j}{w[j]}$ from $B$
      \State $k \gets \min \setbuilder{k}{\text{$k > w[j]$, $\wl{w}{j}{k} \in B$ and $\wl{w}{j}{k+1} \notin B$}}$ \Comment{\makebox[5em][l]{$\min \emptyset = \infty$}} \label{alg-line:pieri-unshuffle-increase}
      \State $w[j] \gets \markdown{k}$
      \If{$k \neq \infty$}
        \State add $\wl{w}{j}{k+1}$ to $B$
        \State add \upmark{} to the $w[i]$ forming a defect with $w[j]$ in the unmarked word \label{alg-line:pieri-unshuffling-defect}
      \EndIf
    \EndWhile
    \State \Return $w'$
  \end{algorithmic}
  \caption{Unshuffling $i-k+1$, \ldots, $i$ from a word $w$ for $\sigma$ with $\pi \toc{i}{k} \sigma$.}
  \label{alg:pieri-unshuffling}
\end{algorithm}
\autoref{alg:pieri-shuffling} is more complicated than \autoref{alg:monk-shuffling} because now we are inserting things at $k$ places and have to keep track of all the insertions and their interactions. 

We start with the word $w'$ of length $n+k$ formed by setting the letters at positions $j_1$, \ldots $j_k$ to $\markdown{\infty}$ and filling out the rest of the word with the letters of $w$ in order. We again rectify the word from right to left. Now, this can lead to multiple positions being in a bumped state at the same time, so we mark the letters in this state with \downmark{} and ignore them when we check if the other insertions bump something on \origautoref{alg-line:pieri-bump-check}. 
Because \autoref{eq:pieri-c} in \autoref{thm:pieri-rule} requires the $a_i$ to be distinct, we keep track of a set $B$ of ``big'' elements. We only add $(a,b)$-crosses where $a \notin B$, $b \in B$ and $a < b$ (the last condition is automatic and ensures that we do not create a defect to the right of the inserted cross). When we add an $(a,b)$-cross, we add $a$ to $B$ and when we bump an $(a,b)$-cross down, we remove $b$ from $B$ (by \autoref{lemma:pieri-upper-label-fixed}, it has previously been added to $B$).

So we go through the word from right to left until we come upon \downmark{} at some position $j$. If the $j$th letter is not $\markdown{\infty}$, we remove the label on the wire it moves down from $B$. We decrease the $j$th letter until we reach an allowed swap of $a$ and $b$, and add $a$ to $B$ and mark the $j$th letter with \upmark{}. If $a$ and $b$ cross again at position $i$, (ignoring crosses previously marked with \downmark{}), we mark the $i$th letter with \downmark{}. 
The \upmark{}s are not needed for the shuffling, but they helps us show that the unshuffling works. 

In the case of $k = 1$, it is relatively easy to see that the algorithms do the same thing: we insert $\markdown{\infty}$ and decrease it on \origautoref{alg-line:pieri-decrease} and unmark it and mark the place where we created a defect and keep going, in this case the set $B$ does nothing, we add the label on the wire we moved up and then remove it again before we do anything.

\begin{example}\label{ex:pieri-53-345-at-345}
  In \autoref{fig:pieri-53-345-at-345}, we shuffle $345$ into $53$ at positions $3$, $4$, and $5$, that is, we rectify $53\markdown{\infty}\markdown{\infty}\markdown{\infty}$. The crosses corresponding to letters marked with \downmark{} (from \origautoref{alg-line:move-mark-line} of \autoref{alg:pieri-shuffling}) are shown as dots and the crosses marked with \upmark{} are bold if nothing was marked with \downmark{} when they received their \upmark{}.
Thus, the unbolded crosses make up the unmarked word, which at every step is a word for $[124365]$.
\end{example}

\begin{figure}
  \centering
  \includestandalone[width=0.99\textwidth]{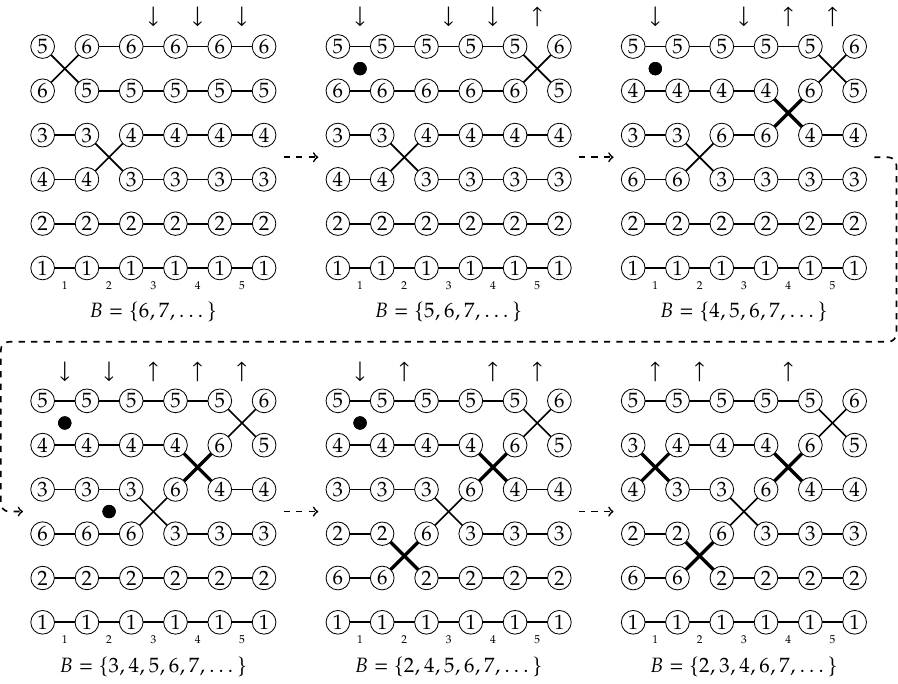}
  \caption{The shuffling of $345$ into $53$ at positions $3$, $4$, and $5$.
}
  \label{fig:pieri-53-345-at-345}
\end{figure}

To show that \autoref{alg:pieri-shuffling} works, we again need to show that it and \autoref{alg:pieri-unshuffling} are well defined, that it terminates with the correct answer, and that it is a bijection.

\subsubsection{Well-definedness}
We begin with well-definedness. In particular, we show that Lines \ref{alg-line:pieri-B-remove}--\ref{alg-line:pieri-decrease} are well defined. 
We show that if an insertion creates a defect, then until the defect is removed, its upper label is fixed and while its lower label can change, it stays in $B$.

\begin{lemma}\label{lemma:pieri-upper-label-fixed}
  The value that is removed from $B$ on \origautoref{alg-line:pieri-B-remove} in \autoref{alg:pieri-shuffling} is the same as was added on \origautoref{alg-line:pieri-B-add} when marking $w[j]$ with \downmark{}.
\end{lemma}
\begin{proof}
  If the insertion of $(a,b)$ added $a$ to $B$, then the algorithm cannot add a cross $(a,c)$, since $a \in B$. Further, if it was to add $(c, a)$, it would first need to add $(c,b)$, since the $b$-wire is just below the $a$-wire. But after adding $(c, b)$, $c \in B$ and hence $(c, a)$ cannot be added.  
\end{proof}
\begin{lemma}\label{lemma:pieri-upper-label-unique}
  On \origautoref{alg-line:pieri-remove-up-arrow} in \autoref{alg:pieri-shuffling}, 
  there is a unique $i > j$ with $w[i]$ marked with \upmark{} such that $\wl{w}{i}{w[i]} = \wl{w}{j}{w[j]+1}$.
\end{lemma}
\begin{proof}
  By \autoref{lemma:pieri-upper-label-fixed}, the label $\wl{w}{j}{w[j]+1}$ remains fixed, so there is such an $i$. Since the label was previously added to $B$ on \origautoref{alg-line:pieri-B-add}, there cannot be more than one such $i$.
\end{proof}

\begin{lemma}\label{lemma:B-status-preserved}
  If $w[i]$ is marked with \downmark{} on \origautoref{alg-line:move-mark-line} in \autoref{alg:pieri-shuffling}, and $w'$ is the word (at the end of a run through the while loop) at some point before $j$ is set to $i$ on \origautoref{alg-line:pieri-j-definition}, then $\wl{w'}{i}{w'[i]} \in B$. 
\end{lemma}
\begin{proof}
  Just after the run through the while loop that marks $w[i]$ with \downmark{}, it is true, since we inserted a cross $(a,b)$ with $b \in B$, and $\wl{w}{i}{w[i]} = b$.
  Suppose $b' = \wl{w'}{i}{w'[i]} \in B$ at the beginning of a run through the while loop. At the end of the while loop, there are two possible issues: the label $b'$ could have changed to $b''$  with $b''\notin B$, or $b'$ could be removed from $B$ (both cannot happen at the same time, since the $b'$-wire would be strictly above the inserted cross at position $j$). If the insertion changes the label to $b''$, it is by inserting the cross $(b'', b')$ and after the insertion, $b''$ is added to $B$, so the first issue does not arise. 

In order for \origautoref{alg-line:pieri-B-remove} to remove $b'$ from $B$, by \autoref{lemma:pieri-upper-label-fixed}, $(b', c)$ for some $c$ was previously inserted creating a defect with the cross at $j$. Before this insertion, $b' \notin B$ and afterwards, $b' \in B$. So before and after it, we did not insert any $(b', d)$ and hence did not change the label $\wl{w'}{i}{w'[i]}$, so there are two further cases. If $b' = b$, then the insertion $(b, c)$ must have happened before the insertion of $(a,b)$ (since by that point, $b \in B$) and then the insertion of $(a,b)$ would swap the $a$ and $b$ wires at $j$, since the defect it creates is at $i < j$. Otherwise, 
 the insertion $(b', c)$ also changed the label $\wl{w'}{i}{w'[i]}$ to $b'$, but this cannot happen since the insertion created a defect. 
So the second issue also does not arise. 
\end{proof}

\begin{lemma}
The set on \origautoref{alg-line:pieri-decrease} in \autoref{alg:pieri-shuffling} is non-empty.
\end{lemma}
\begin{proof}
  If $w[j] = \markdown{\infty}$, then $k < w[j]$ imposes no restriction on $k$. Thus, since there is an $N > 0$ such that $N \in B$ and $-N \notin B$, there is a $k$ with $-N < k < N$ such that $\wl{w}{j}{k} \notin B$ and $\wl{w}{j}{k+1} \in B$.

  Otherwise, by \autoref{lemma:B-status-preserved}, $\wl{w}{j}{w[j]} \in B$, and since for large $N > 0$, $-N \notin B$, so there is a $k$ with $-N < k < w[j]$ such that $\wl{w}{j}{k} \notin B$ and $\wl{w}{j}{k+1} \in B$.
\end{proof}

Putting this together, we get that \autoref{alg:pieri-shuffling} is well-defined.
Further, since the unmarked word is not changing, we get the following, which shows that the unshuffling (\autoref{alg:pieri-unshuffling}) is also well-defined.
\begin{lemma}
  On \origautoref{alg-line:pieri-unshuffling-defect} in \autoref{alg:pieri-unshuffling} there is a unique $i$ that forms a defect with $w[j]$ in the unmarked word (ignoring the \downmark just added to $w[j]$).
\end{lemma}
\subsubsection{Correctness}
Moving on to correctness, we show that the algorithm terminates and that the output is a word for $\sigma$ with $\pi \toc{i}{k} \sigma$. 

\begin{lemma}\label{lemma:pieri-terminates}
  \autoref{alg:pieri-shuffling} terminates.
\end{lemma}
\begin{proof}
  After each run through the while loop, the rightmost \downmark{} has moved further to the left or disappeared.
\end{proof}

\begin{lemma}\label{lemma:unmarked-subword}
  The unmarked subword is the rightmost subword for $\pi$.
\end{lemma}
\begin{proof}
  At the start of the algorithm, it is true. 
During a run through the while loop, there are two cases. The algorithm decreases the $j$th letter, marked by \downmark{}. Either it does not form a defect with a letter in the unmarked subword, in which case it gets marked with \upmark{}, and unmarked subword remains the same. 
Or it forms a defect with a letter in the unmarked subword, in which case it is marked with \upmark{} and the other letter is marked with \downmark{}. By definition and \autoref{lemma:pieri-upper-label-unique}, this \upmark{} is not included in the unmarked subword, so it remains true.
\end{proof}

\begin{lemma}\label{lemma:pieri-a-le-b}
  Let $(a,b)$ with $a \notin B$ and $b \in B$ be the next insertion on \origautoref{alg-line:pieri-decrease}. Then $a < b$.
\end{lemma}
\begin{proof}
Since $a \notin B$, $a \le i$. Either $b > i$, in which case $a < b$, or there is a sequence of previous insertions $(b, c_k)$, $(c_k, c_{k-1})$, \ldots, $(c_1, c_0)$, at positions $p_k < \dotsb < p_0$, where, by induction $b < c_k < \dotsb < c_0$ and $c_0 > i$.
But by adding the cross at position $p_k$ first, it is an insertion $(b, c_0)$, and therefore everything that ends up between $b$ and $c_0$ (and in particular $a$) has to be either less than $b$ or greater than $c_0 > i$.
\end{proof}

\begin{theorem}\label{thm:pieri-correct-output}
  The algorithm ends with a word for a permutation $ \sigma$ with $\pi \toc{i}{k} \sigma$.
\end{theorem}
\begin{proof}
  By \autoref{lemma:unmarked-subword}, $w$ contains a subword for $\pi$. Let $(c_1, d_1)$, \ldots, $(c_k, d_k)$ be the crosses marked by \upmark{}, read from right to left. That is, $w$ is a word for $\pi t_{c_1d_1} \dotsb t_{c_kd_k}$. Let $B_0 = \{i+1, i+2, \dotsc\}$ and $B_j = B_{j-1} \cup \{c_j\}$. Since by \autoref{lemma:pieri-a-le-b} each $t_{c_jd_j}$ corresponds to the insertion of a cross with $c_j < d_j$, the length increases by one at each step. By \origautoref{alg-line:pieri-decrease}, $c_j \notin B_{j-1}$ and $d_j \in B_{j-1}$. In particular, the $c_j$ are distinct and $c_j \le i$ for all $j$.

If all $d_j > i$, then we are done. Otherwise, let $j_0$ be the smallest $j$ such that $d_{j_0} \le i$. Then, there is a $j_1 < j_0$ such that $c_{j_1} = d_{j_0}$. Since the $c_j$ are distinct, and since $d_j > i$ for $j_1 < j < j_0$, $t_{c_jd_j}$ and $t_{c_{j_0}d_{j_0}}$ commute, and $t_{c_{j_1}d_{j_1}}t_{c_{j_0}d_{j_0}} = t_{c_{j_0}d_{j_1}}t_{c_{j_1}d_{j_1}}$, since $c_{j_1} = d_{j_0}$.
This is the sequence of labels on crosses we get by adding the $j_0$th cross to the word before we add the $j_1$th, so the length still increases by one at each step. The $c_j$ have not changed, and $d_{j_0}$ was replaced by $d_{j_1} > i$. So by repeating this, we see that $w$ is a word for $\sigma$ with $\pi \toc{i}{k} \sigma$. 
\end{proof}

\subsubsection{Bijection}
Finally, we show that \autoref{alg:pieri-shuffling} computes a bijection.
\begin{theorem}
  \autoref{alg:pieri-shuffling} computes a bijection.
\end{theorem}
\begin{proof}
  The only line in \autoref{alg:pieri-shuffling} that is not clearly reversible and reversed by \autoref{alg:pieri-unshuffling} is \origautoref{alg-line:pieri-decrease}. But that line is just saying that all the integers between $k$ and $w[j]$ lie in $B$, so it is reversed by \origautoref{alg-line:pieri-unshuffle-increase}. Since, the algorithm ends with the letters not in the rightmost subword marked with \upmark{}, we get that it computes a bijection. 
\end{proof}

Just like in \autoref{alg:monk-shuffling}, letters only affect what happens to their left, 
so we get the following (cf. \autoref{prop:monk-shuffle-is-back-stable}).
\begin{proposition}\label{prop:pieri-shuffle-is-back-stable}
Given $i$ and a word $w = a_1 \dotsm a_n$ with $\markdown{\infty}$ at some of its positions. If $w' = a_1' \dotsm a_n'$ is the result of applying \autoref{alg:monk-shuffling} to $i$ and $w$, then $a'_2 \dotsm a'_n$ is the result of applying the algorithm to $i$ and $a_2 \dots a_n$.
\end{proposition}

\section{Slides are not a Gr\"obner degeneration of Schuberts}
\label{sec:slides-not-degen-of-schuberts}
As we saw in \autoref{sec:geometry}, when $\succ$ is an antidiagonal term order, it degenerates $\mSchub{\pi}$ to the pipe dream complex: $\init_\succ I_\pi$ has primary decomposition 
\[
\init_\succ I_\pi = \bigcap_{P \in \PD(\pi)}\idealbuilder{z_{ij}}{\text{$(i,j)$ is a cross in $P$}},
\]
so $\mSchub{\pi}$ splits into a union of coordinate subspaces, one for each pipe dream for $\pi$.
An antidiagonal weight order $\succ_w$ defines a partial degeneration of $\mSchub{\pi}$: $\init_{\succ_w} I_\pi$ has primary decomposition  
\[
\init_{\succ_w} I_\pi = J_1 \cap \dotsb \cap J_k,
\]
so $\mSchub{\pi}$ splits into components corresponding to these ideals. By picking a finer weight order, each $J_i$ can be further degenerated to a a monomial ideal. This way, $w$ gives a partition of the pipe dream complex into $k$ pieces. 
In general, it is unknown what these resulting varieties and splittings are \cite[Remark 1.8.6]{Grobner-geometry-paper}. We want to see if they include the decomposition of the the subword complex into slide complexes.

\subsection{Failure for $[1432]$}
The simplest example where this fails is $[1432]$. There are two words for $[1432]$, $232$ and $323$, so there are two slides.
\begin{figure}\centering
\includestandalone{rothe-diagram-1432}
\caption[The Rothe diagram for the permutation $\sqbl1432\sqbr$.]{The Rothe diagram for the permutation $[1432]$. The Fulton essential set is marked in gray.}
\label{fig:rothe-diagram-1432}
\end{figure}
The Rothe diagram for $[1432]$ in \autoref{fig:rothe-diagram-1432} shows that there are two rank conditions: the north west $2\times 3$ rectangle and the north west $3 \times 2$ rectangle both have rank at most $1$. So the ideal $I$ is 
\[  
\langle 
z_{12}z_{21} - z_{11}z_{22}, 
z_{13}z_{21} - z_{11}z_{23}, 
z_{13}z_{22} - z_{12}z_{23}, 
z_{12}z_{31} - z_{11}z_{32}, 
z_{22}z_{31} - z_{21}z_{32} 
\rangle.
\]
Fully Gr\"obner degenerating $I$ with an antidiagonal gives us 
\[
\langle 
z_{12}z_{21}, 
z_{13}z_{21},
z_{13}z_{22},
z_{12}z_{31},
z_{22}z_{31}
\rangle, 
\]
which has the primary decomposition 
\[
\overbracket{\langle z_{12}, z_{13}, z_{22} \rangle
\cap 
\langle z_{12}, z_{13}, z_{31} \rangle
\cap
\langle z_{13}, z_{21}, z_{31} \rangle
\cap
\langle z_{21}, z_{22}, z_{31} \rangle}^{323}
\cap
\overbracket{\langle z_{12}, z_{21}, z_{22} \rangle}^{232}
.
\]
 The pipe dream complex was shown in \autoref{fig:subword-complex-1432} in \autoref{sec:slide}.

A weight $w = (w_{11}, w_{12}, w_{13}, w_{21}, w_{22}, w_{23}, w_{31}, w_{32})$ defines an antidiagonal Gr\"obner degeneration of $I$ if and only if
\begin{align*}
w_{12} + w_{21} &\ge w_{11} + w_{22}, \\ 
w_{13} + w_{21} &\ge w_{11} + w_{23},\\
w_{13} + w_{22} &\ge w_{12} + w_{23}, \\ 
w_{12} + w_{31} &\ge w_{11} + w_{32},\\
w_{22} + w_{31} &\ge w_{21} + w_{32}. 
\end{align*}
These equations define a three dimensional cone crossed with a seven dimensional lineality space: 
\[
P = a_1 R_1 + a_2 R_2 + a_3 R_3 + L,
\]
where $a_1$, $a_2$, $a_3 \ge 0$,  
\[
\RPI = \parenMatrixstack{ 
0 & 0 & 0 \\
0 & 0 & -1\\
 0 & 0 & 0
},\quad
\RPII = \parenMatrixstack{
-1 & 0 & 0\\
 0 & 0 & 0\\
 0 & 0 & 0
},\quad
\RPIII = \parenMatrixstack{
 0 & 0 & 0\\
 0 & 0 & 0\\
 0 & -1 & 0
},
\]
and $L$ is a matrix of the form
\[
\parenMatrixstack{
a+e & a+f & a+g\\
b+e & b+f & b+g\\
c+e & c+f & d
}.
\]
We will go through the process of degenerating using the ray $\RPIII$, and report all the splittings in \autoref{tab:grobner-degens-1432}.

Gr\"obner degenerating $I$ with $w$ on $\RPIII$ does nothing to the first three generators of $I$, and takes the first term of the remaining two generators. So 
\[
\init_w I = \langle z_{12}z_{21} - z_{11}z_{22},
z_{13}z_{21} - z_{11}z_{23},
z_{13}z_{22} - z_{12}z_{23},
z_{12}z_{31},
z_{22}z_{31} \rangle,
\]
which has the primary decomposition
\[
\overbracket{
\langle 
z_{22}, 
z_{12}, 
z_{13}z_{21} - z_{11}z_{23} 
\rangle 
}^{I_1}
\cap 
\overbracket{
\langle 
z_{31}, 
z_{13}z_{22} - z_{12}z_{23},
z_{13}z_{21} - z_{11}z_{23},
z_{12}z_{21} - z_{11}z_{22}
\rangle}^{I_2}
.
\]
We fully degenerate both ideals
\begin{align*}
\init I_1 &= \langle z_{22}, z_{12}, z_{13}z_{21}\rangle, \\
\init I_2 &= \langle z_{31}, z_{13}z_{22}, z_{13}z_{21}, z_{12}z_{21}\rangle,
\end{align*}
and find their primary decompositions
\begin{align*}
\init I_1 &= 
\langle z_{12}, z_{13}, z_{22} \rangle \cap 
\langle z_{12}, z_{21}, z_{22} \rangle,\\
\init I_2 &= \langle z_{12}, z_{13}, z_{31} \rangle \cap 
\langle z_{13}, z_{21}, z_{31} \rangle \cap 
\langle z_{21}, z_{22}, z_{31} \rangle, 
\end{align*}
to find which pipe dreams they correspond to.
The resulting splitting of the pipe dream complex is shown in \autoref{fig:subword-complex-1432-split-right}. There, $I_1$ corresponds to the lower right piece and $I_2$
corresponds to the upper left part.

This Gr\"obner degeneration and the other five intermediate ones are shown in \autoref{tab:grobner-degens-1432} and \autoref{fig:subword-complex-1432-splittings}.
As we can see, none of the degenerations split the subword complex into the two pieces corresponding to the slides shown in \autoref{fig:subword-complex-1432}.

\begin{table}\centering
\caption{Intermediate antidiagonal Gr\"obner degenerations of $[1432]$.}
\label{tab:grobner-degens-1432}
\begin{tabular}{l|l|l}
Face & Primary Decomposition & Complex\\
\hline
$\RPIII$ 
&
  $\begin{aligned} 
    &\langle z_{22}, z_{12}, z_{13}z_{21} - z_{11}z_{23}\rangle\\
    &\langle z_{31}, z_{13}z_{22} - z_{12}z_{23}, z_{13}z_{21} - z_{11}z_{23}, z_{12}z_{21} - z_{11}z_{22}\rangle
  \end{aligned}$ & \autoref{fig:subword-complex-1432-split-right}
\\\hline
$\RPII$ 
&
  $\begin{aligned}
    &\langle z_{22}, z_{21}, z_{12}\rangle\\
    &\langle z_{31}, z_{21}, z_{13}z_{22} - z_{12}z_{23}\rangle\\
    &\langle z_{13}, z_{12}, z_{22}z_{31} - z_{21}z_{32}\rangle
  \end{aligned}$ & \autoref{fig:subword-complex-1432-split-in-three}
\\\hline
$\RPI$ 
&
  $\begin{aligned}
    &\langle z_{22}, z_{21}, z_{12}z_{31} - z_{11}z_{32}\rangle\\
    &\langle z_{13}, z_{22}z_{31} - z_{21}z_{32}, z_{12}z_{31} - z_{11}z_{32}, z_{12}z_{21} - z_{11}z_{22}\rangle
  \end{aligned}$& \autoref{fig:subword-complex-1432-split-left}
\\\hline
$\RPII \cup \RPIII$ 
&
  $\begin{aligned}
    &\langle z_{22}, z_{21}, z_{12}\rangle\\
    &\langle z_{31}, z_{13}, z_{12}\rangle\\
    &\langle z_{22}, z_{13}, z_{12}\rangle\\
    &\langle z_{31}, z_{21}, z_{13}z_{22} - z_{12}z_{23}\rangle
  \end{aligned}$ & \autoref{fig:subword-complex-1432-split-right-three}
\\\hline
$\RPI \cup \RPIII$ 
&
  $\begin{aligned}
    &\langle z_{22}, z_{21}, z_{12}\rangle\\
    &\langle z_{22}, z_{13}, z_{12}\rangle\\
    &\langle z_{31}, z_{22}, z_{21}\rangle\\
    &\langle z_{31}, z_{13}, z_{12}z_{21} - z_{11}z_{22}\rangle
  \end{aligned}$ & \autoref{fig:subword-complex-1432-split-left-right}
\\\hline
$\RPI \cup \RPII$ 
&
  $\begin{aligned}
    &\langle z_{31}, z_{21}, z_{13}\rangle\\
    &\langle z_{22}, z_{21}, z_{12}\rangle\\
    &\langle z_{31}, z_{22}, z_{21}\rangle\\
    &\langle z_{13}, z_{12}, z_{22}z_{31} - z_{21}z_{32}\rangle
    \end{aligned}$ & \autoref{fig:subword-complex-1432-split-left-three}
\\\hline
\end{tabular}
\end{table}

\begin{figure}\centering
\begin{subfigure}{0.4\textwidth}
\includestandalone[width=\textwidth]{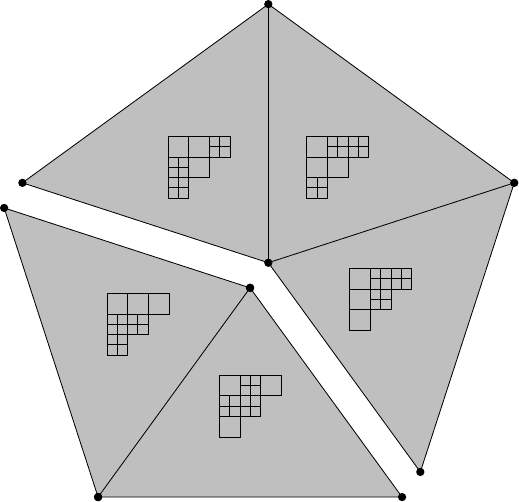}
\caption{$\RPI$}
\label{fig:subword-complex-1432-split-left}
\end{subfigure}
\hfill
\begin{subfigure}{0.4\textwidth}
\includestandalone[width=\textwidth]{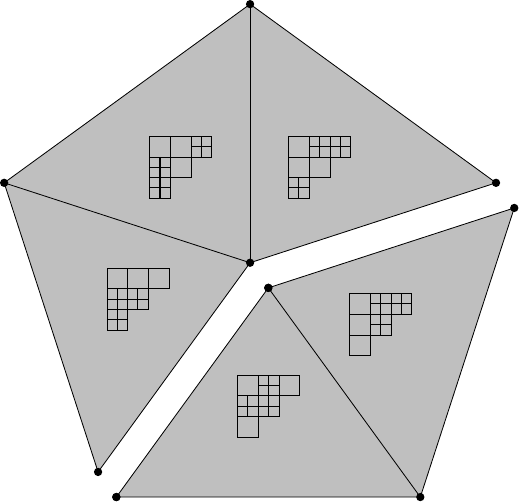}
\caption{$\RPIII$}
\label{fig:subword-complex-1432-split-right}
\end{subfigure}
\hfill
\begin{subfigure}{0.4\textwidth}
\includestandalone[width=\textwidth]{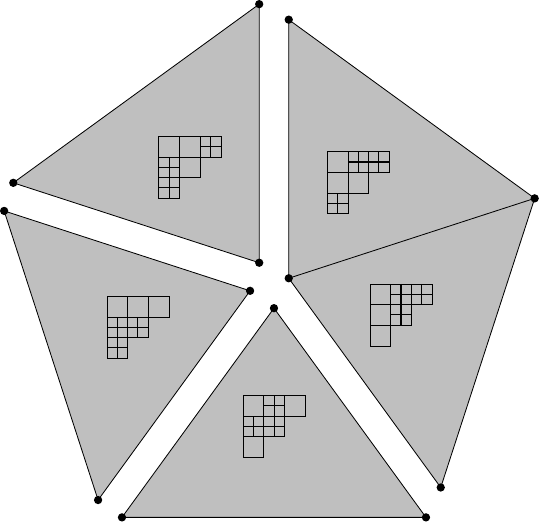}
\caption{$\RPI \cup \RPII$}
\label{fig:subword-complex-1432-split-left-three}
\end{subfigure}
\hfill
\begin{subfigure}{0.4\textwidth}
\includestandalone[width=\textwidth]{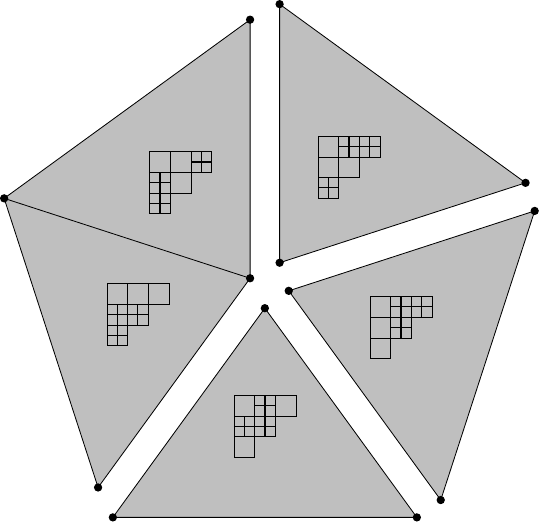}
\caption{$\RPII \cup \RPIII$}
\label{fig:subword-complex-1432-split-right-three}
\end{subfigure}
\hfill
\begin{subfigure}{0.4\textwidth}
\includestandalone[width=\textwidth]{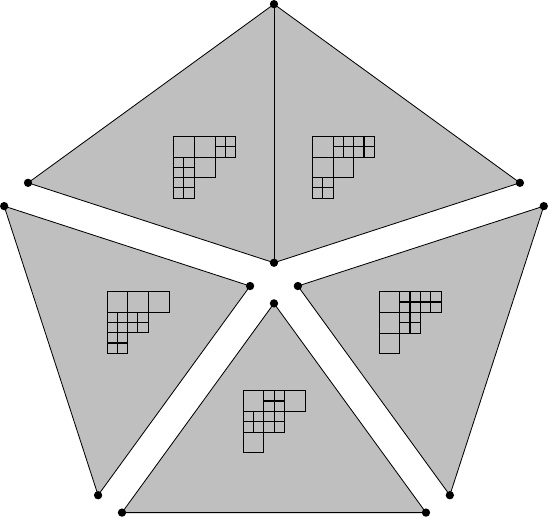}
\caption{$\RPI \cup \RPIII$}
\label{fig:subword-complex-1432-split-left-right}
\end{subfigure}
\hfill
\begin{subfigure}{0.4\textwidth}
\includestandalone[width=\textwidth]{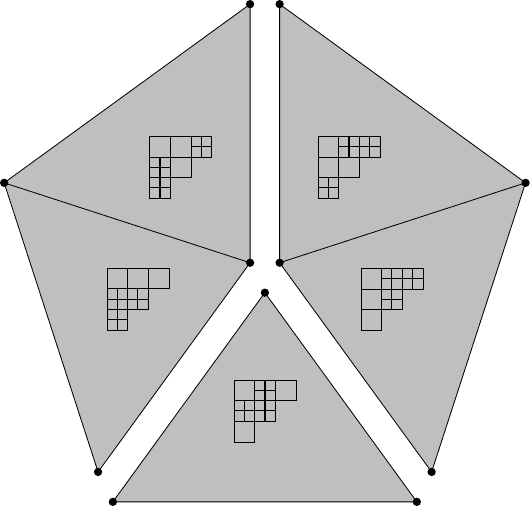}
\caption{$\RPII$}
\label{fig:subword-complex-1432-split-in-three}
\end{subfigure}
\caption[Splittings induced by antidiagonal weight orders, $\pi = \sqbl1432\sqbr$.]{Splittings induced by antidiagonal weight orders, of the pipe dream complex for $\pi = [1432]$.}\label{fig:subword-complex-1432-splittings}
\end{figure}

\subsection{Failure for Grassmannian permutations}
\begin{figure}
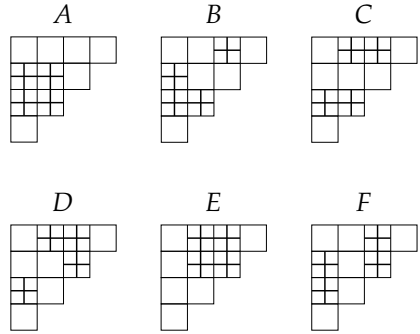
\centering
\includestandalone{pipe-dreams-14523}
\caption[The pipe dreams for $\sqbl14523\sqbr$.]{The pipe dreams for $[14523]$. The slide $3243$ corresponds to $\{A, B, C, D, E\}$ and $3423$ corresponds to $\{F\}$.}
\label{fig:pipe-dreams-14523}
\end{figure}
Here we show the results of all minimal Gr\"obner degenerations of the Grassmannian permutation $[14523]$ to show that even in this case, it is not the case that the slide complexes are an intermediate Gr\"obner degeneration of the matrix Schubert variety. After removing cone points, the subword complex for $[14523]$ is three dimensional, so we do not attempt to draw it. Instead, we are showing the pipe dreams in \autoref{fig:pipe-dreams-14523}. There are two words for $[14523]$, $3243$, and $3423$. All pipe dreams except $F$ correspond to $3243$.
\begin{figure}
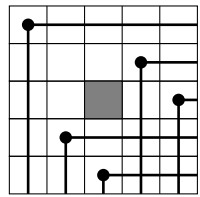
\centering
\includestandalone{rothe-diagram-14523}
\caption[The Rothe diagram for the permutation $\sqbl14523\sqbr$.]{The Rothe diagram for the permutation $[14523]$. The Fulton essential set $\{(3,3)\}$ is marked in gray.}
\label{fig:rothe-diagram-14523}
\end{figure}
The Rothe diagram for $[14523]$ in \autoref{fig:rothe-diagram-14523} shows that the rank of the north-west $3\times 3$ matrix is $1$, so the ideal is
\[
\begin{aligned}
I = \langle
&z_{12}z_{21} - z_{11}z_{22},
z_{13}z_{21} - z_{11}z_{23},
z_{13}z_{22} - z_{12}z_{23},\\
&z_{12}z_{31} - z_{11}z_{32},
z_{13}z_{31} - z_{11}z_{33},
z_{13}z_{32} - z_{12}z_{33},\\
&z_{22}z_{31} - z_{21}z_{32},
z_{23}z_{31} - z_{21}z_{33},
z_{23}z_{32} - z_{22}z_{33}
\rangle.
\end{aligned}
\]
A weight $w = (w_{11}, \dotsc, w_{33})$ defines an antidiagonal term order if
\begin{align*}
w_{12} + w_{21} &\ge w_{11} + w_{22},\\
w_{13} + w_{22} &\ge w_{12} + w_{23},\\
w_{22} + w_{31} &\ge w_{21} + w_{32},\\
w_{23} + w_{32} &\ge w_{22} + w_{33}.
\end{align*}
The antidiagonal cone is a four dimensional cone crossed with a five dimensional lineality space. The minimal Gr\"obner degenerations, which correspond to the rays of the cone, together with the rays and the are shown in \autoref{tab:minimal-grobner-degens-14523}. Again, we see that none of the degenerations correspond to the decomposition of the subword complex into slide complexes.

\begin{table}[hbt!]\centering
\caption{Minimal antidiagonal Gr\"obner degenerations of $[14523]$.}
\label{tab:minimal-grobner-degens-14523}
\begin{tabular}{l|l}
Ray & Splitting\\
\hline
$
\delimitershortfall=0pt
\vphantom{
\matrixspacingrule
}\kern-\nulldelimiterspace
\parenMatrixstack{
-1 & -1 & 0 \\
0 & 0 & 0 \\
0 & 0 & 0
}
\vphantom{
\matrixspacingrule
}\kern-\nulldelimiterspace
$
& $\{A\}$ , $\{B, C, D, E, F\}$
\\\hline
$
\delimitershortfall=0pt
\vphantom{
\matrixspacingrule
}\kern-\nulldelimiterspace
\parenMatrixstack{
-1 & 0 & 0 \\
0 & 0 & 0 \\
0 & 0 & 0
}
\vphantom{
\matrixspacingrule
}\kern-\nulldelimiterspace
$ 
& 
$\{A, B, F\}$ , $\{C, D, E\}$
\\\hline
$
\delimitershortfall=0pt
\vphantom{
\matrixspacingrule
}\kern-\nulldelimiterspace
\parenMatrixstack{
-1 & 0 & 0 \\
-1 & 0 & 0 \\
 0 & 0 & 0
}
\vphantom{
\matrixspacingrule
}\kern-\nulldelimiterspace
$
& 
$\{A, B, C, D, F\}$ , $\{E\}$
\\\hline
$
\delimitershortfall=0pt
\vphantom{
\matrixspacingrule
}\kern-\nulldelimiterspace
\parenMatrixstack{
-1 & -1 & 0 \\
-1 & -1 & 0 \\
 0 &  0 & 0
}
\vphantom{
\matrixspacingrule
}\kern-\nulldelimiterspace
$ 
& 
$\{A, B, C\}$, $\{D, E, F\}$
\\\hline
\end{tabular}
\end{table}

\section{Backwards saturated sets are balls or spheres}
\label{sec:backwards-saturated-balls}

Since, as we saw in \autoref{sec:slides-not-degen-of-schuberts}, the slide complexes are not Gr\"obner degenerations of pipe dream complexes, we keep looking for an explanation of where they come from. 
By analogy with how the shellability of the pipe dream complex explains why we can use monomials corresponding to faces to compute $\Groth_\pi$, the fact that we can also use glides to compute $\Groth_\pi$ suggests that a ``shelling by slide balls'' might exist.

Our goal, therefore, is to find an ordering of the reduced words $w_1$, \ldots, $w_k$ for $\pi$ such that, if $\Delta(Q, \pi)$ is $d$-dimensional, then $\Delta(Q, w_1) \cup \dotsb \cup \Delta(Q, w_{i-1})$ is a $d$-ball and $\Delta(Q, w_i) \cap (\Delta(Q, w_1) \cup \dotsb \cup \Delta(Q, w_{i-1}))$, is nonempty, pure and $(d-1)$-dimensional, for $1 \le i \le k$. 
We have not succeeded, but along the way, we managed to find a useful partial results, see \autoref{sec:forest-polynomials}.

\label{sec:which-sets-give-balls}
\begin{definition}
  Given a set $W = \{w_1, \dotsc, w_n\}$ of reduced words for $\pi$, let $\Delta_W(Q, \pi)$ be the subcomplex of the subword complex whose faces $P$ satisfy that $Q \setminus P$ contains a word from $W$.
\end{definition}
If $W = \{w\}$, then $\Delta_W(Q, \pi)$ is the slide complex $\Delta(Q, w)$, and if $W = \RW(\pi)$, then $\Delta_W(Q, \pi) = \Delta(Q, \pi)$. 
In both these cases, $\Delta_W(Q, \pi)$ is a ball or a sphere. We give will give a condition on $W$ that guarantees that $\Delta_W(Q, \pi)$ is a ball or sphere. This generalization includes the two previous cases as well as words corresponding to the forest polynomials of Nadeau and Tewari \cite{Forest-polynomial-paper}, see \autoref{sec:forest-polynomials}.
Of course, the ``$\pi$'' in $\Delta_W(Q, \pi)$ is superfluous. We are mainly including it to avoid mistakes in the proof of \autoref{thm:delta-w-is-vertex-decomposable}.

\begin{definition}
  Given a set $W$ of reduced words for $\pi$. Let $W_\sigma$ be the subset of $W$ consisting of words that begin with $\sigma$, and let $W_\sigma' = \{\,w \st \sigma w \in W_\sigma\,\}$ be those words, with their first letter removed.
  A set $W$ of words for $\pi$ is \emph{\bs{}} if:
  \begin{itemize}
  \item $W$ is empty, or
  \item for each $\sigma$ such that $W_\sigma \neq \emptyset$, $W_\sigma'$ is \bs{} and every word in $W$ contains a word in $W_\sigma'$ as a subword.
  \end{itemize}
\end{definition}

\begin{example}
  If $W = \{w\}$, then $W$ is \bs{}.
If $W = \RW(\pi)$, then $W$ is \bs{}.
If $W$ is the set of all words for $\pi$ starting with $\sigma$, then, $W_\sigma' = \RW(\sigma\pi)$ is \bs{} and every word in $W$ contains a word in $W_\sigma'$, so $W$ is \bs{}.
\end{example}

\begin{example}
  The set $W = \{1434, 4134, 4314, 4341\}$ is \bs{}. Here $W_1 = \{1434\}$, $W_1' = \{434\}$, $W_4 = \{4134, 4314, 4341\}$ and $W_4' = \{134, 314, 341\}$. All the words in $W_4$ contain $434$, and $1434$ contains $134$. $W_1'$ is \bs{} since it only contains one word and $W_4'$ is \bs{} since it equals $\RW([21453])$. 
\end{example}

The definition was made so that the proof of \autoref{thm:subword-complex-vertex-decomposable} only needs a minor modification to show that $\Delta_W(Q, \pi)$ is vertex-decomposable.

\begin{theorem}\label{thm:delta-w-is-vertex-decomposable}
  If $W$ is a \bs{} set of reduced words for $\pi$, then $\Delta_W(Q, \pi)$ is vertex-decomposable.
\end{theorem}
\begin{proof}
  Since $W$ is a set of reduced words for $\pi$, the facets of $\Delta_W(Q, \pi)$ are a subset of the facets of $\Delta(Q, \pi)$, so $\Delta_W(Q, \pi)$ is pure.

  Let $Q = (\sigma, \sigma_1, \dotsc, \sigma_n)$, $Q' = (\sigma_1, \dotsc, \sigma_n)$. The link of $\sigma$ consists of those faces $P$ whose complements $Q\setminus P$ are of the form 
$\sigma R$, where $R$ contains a word from $W$. So the link can be identified with $\Delta_W(Q', \pi)$, which is vertex-decomposable by induction.

The deletion of $\sigma$ consists of those faces $P$ whose complements $Q \setminus P$ are of the form $\sigma R$, where $\sigma R$, contains a word from $W$. 
There are two cases. 

If $W_\sigma$ is empty, the deletion of $\sigma$ is equal to the link, since none of the words contained in the complements of the faces actually use the $\sigma$. 

If $W_\sigma$ is nonempty, there are two kinds of words. 
A complement $\sigma R$ contains $\sigma w'$ if and only if $R$ contains $w'$. A complement $\sigma R$ contains  $\tau w'$ for $\tau \neq \sigma$ if and only if $R$ contains $\tau w'$, but since $W$ is \bs{}, $R$ then contains some $w'' \in W_\sigma'$. 
So the deletion of $\sigma$ equals $\Delta_{W_\sigma'}(Q', \sigma \pi)$.
Since the words in $W_\sigma'$ are words for $\sigma\pi$, $\Delta_{W_\sigma'}(Q', \sigma \pi)$ is vertex-decomposable by induction. Hence, $\Delta_W(Q, \pi)$ is vertex-decomposable.
\end{proof}

The version where $W$ is a singleton was proven by Smirnov and Tutubalina \cite[Theorem 5]{Smirnov-Tutubalina} (also by modifying the proof of \autoref{thm:subword-complex-vertex-decomposable}). 
Continuing like in \autoref{thm:subword-complexes-are-balls-or-spheres}, we get that these complexes are balls or spheres. 

\begin{theorem}\label{thm:delta-w-ball-or-sphere}
  If $W$ is a \bs{} set of reduced words for $\pi$, then $\Delta_W(Q, \pi)$ is a ball or a sphere.
\end{theorem}
\begin{proof}
  Since $W$ consists of reduced words for $\pi$, the facets of $\Delta_W(Q, \pi)$ are a subset of the facets of $\Delta(Q, \pi)$. Therefore, by \autoref{lemma:subword-complex-thin}, each codimension $1$ face is contained in at most two facets, so by \autoref{thm:ball-sphere} $\Delta_W(Q, \pi)$ is a ball or a sphere.
\end{proof}

This theorem shows that many unions of slide complexes are balls or spheres. 
However, there are $\pi$ where there is no ordering $w_1$, \ldots, $w_n$ of the reduced words for $\pi$, such that vertex-decomposing, starting with the first letter, shows that $\Delta_{\{w_1, \dotsc, w_k\}}(Q, \pi)$ is a ball for every $k$, $1 \le k \le n$.

\begin{example}\label{ex:bs-not-strong-enough}
  Let $\pi$ be the product of $abc$, where $a$, $b$ and $c$ all commute with each other, and let $Q$ be a long enough triangular word. 
The words for $\pi$ are $abc$, $acb$, $bac$, $bca$, $cab$ and $cba$. Then, for any set $W$ containing all but one of the words, repeatedly taking the link or deletion of the first letter of $Q$ fails to show that $\Delta_W(Q, \pi)$ is vertex-decomposable. Assume without loss of generality that $W$ contains all the words except $cba$. As we saw from the proof of \autoref{thm:delta-w-is-vertex-decomposable}, when we take the link, $Q$ gets shorter but $W$ does not change. So we take the link until the first letter is $c$, and then take the deletion of the $c$. The result is the subcomplex of $\Delta(Q, c\pi)$ whose facets $P$ satisfy that $Q \setminus P$ contains a word in $\{ab, bac, bca\}$. Due to the differing lengths of the words, the result might not be pure and therefore not vertex-decomposable.

Of course, a more careful analysis would probably show that (some of) these $\Delta_W(Q, \pi)$ are balls or spheres. Indeed, a computer search suggests that some of them are homology-spheres.
\end{example}

\subsection{Forest complexes are balls or spheres}\label{sec:forest-polynomials}
Nadeau and Tewari \cite{Forest-polynomial-paper} introduced a new class of polynomials, forest polynomials, that give a coarsening of the slide expansion of Schubert polynomials. They expand positively in slide polynomials and Schubert polynomials expand positively in forest polynomials. Since the definition is long, we do not include it, but instead only the details needed to show that the simplicial complexes corresponding to forest polynomials are balls or spheres.
\begin{remark}
Note that in \cite{Forest-polynomial-paper}, they index slide polynomials by reversed words. This switches $\pi$ and $\inv \pi$ and left and right in their definitions.
\end{remark}
Given a reduced word $w$ for $\pi$, a ``local binary search forest'' $P(w)$ is built inductively by inserting letters from right to left according to some rules \cite[Section~4.1]{Forest-polynomial-paper}. 
If $w = \sigma w'$, then $P(w)$ is $P(w')$ with one added node. This new node is an orphan: it is not the child of any node in $P(w')$.
Two words $w_1$ and $w_2$ are equivalent, $w_1 \forestequiv w_2$ if $P(w_1) = P(w_2)$.
Further, they show that $\forestequiv$ is generated by the relation $\forestsim$, where $w_1 \forestsim w_2$ if $w_1 = uabv$ and $w_2 = ubav$ for words $u$ and $v$ and letters $a$ and $b$ such that after inserting $b$ and then $a$ in $P(v)$, the node added by inserting $b$ is not a child of the node added by inserting $a$.

Let $\sG_\pi$ be the set of equivalence classes of reduced words for $\pi$. Then
the Schubert polynomial $\Schub_\pi$ expands as \cite[Theorem 1.4]{Forest-polynomial-paper}
\[
\Schub_\pi = \sum_{\sC \in \sG_\pi} \forest_\sC,
\]
where $\forest_\sC$ is the forest polynomial \cite[Equation 3.4]{Forest-polynomial-paper}
\[
\forest_\sC = \sum_{w \in \sC} \slide_w.
\]
\begin{theorem}
  The equivalence classes $\sC$ are \bs{}, so the forest complexes $\Delta_\sC(Q, \pi)$ are balls or spheres.
\end{theorem}
\begin{proof}
  Let $P$ be the local binary search forest corresponding to $\sC$.
  If $\sigma w \in \sC$, let $p$ be the node created by the insertion of $\sigma$. Then $p$ is an orphan in $P$, and removing it from $P$ gives the forest for the set $\sC_\sigma' = \setbuilder{v}{\sigma v \in \sC}$, so by induction, $\sC_\sigma'$ is \bs{}. 
Let $w' \in \sC$. During the construction of $P(w') = P$, the insertion of some letter $\sigma$ created the node $p$. Let $w' = u\sigma v$, where $\sigma$ is this letter.
Since $p$ is an orphan in $P$, commuting $\sigma$ to the left gives an equivalent word, $w' \forestequiv \sigma u v$. But $\sigma uv \in \sC_\sigma$, so $w'$ contains $uv \in \sC_\sigma'$. 

Thus, $\sC$ is \bs{}, and by \autoref{thm:delta-w-ball-or-sphere}, $\Delta_\sC(Q, \pi)$ is a ball or a sphere.
\end{proof}

\section{Decomposing tableau complexes} 
\label{sec:decomposing-tableau-complexes}
Just as the subword complex for $\pi$ splits into balls corresponding to the terms in the slide expansion  of the Schubert polynomial $\Schub_\pi$ (\autoref{thm:schubert-slide-ball-decomposition}), we will show that the tableau complex of semistandard Young tableaux splits into balls corresponding to the terms in the expansion of the Schur polynomial $\schur_\lambda$ into a sum of fundamental quasisymmetric polynomials:
\begin{equation}\label{eq:schur-fundamental-quasisymmetric}
\schur_\lambda = \sum_{T \in \SYT(\lambda)}\quasi_{\comp(\Des(T))}.
\end{equation}

As we saw in \autoref{thm:composition-tableau-complex} each fundamental quasisymmetric polynomial has a corresponding composition tableau complex, which is a ball or sphere. So what remains is to identify them inside the tableaux complex of semistandard Young tableaux.

\begin{theorem}
  Given a standard Young tableaux $T$ of shape $\lambda$ and let $\sS_T$ be the set of semistandard Young tableau with standardization $T$. The tableau complex $\Delta(\sS_T)$ is a ball or a sphere.
\end{theorem}
\begin{proof}
  Let $S$ be a semistandard Young tableau with shape $\lambda$ and let $s_k$ be the entry of $S$ in the box where $T$ contains $k$. The standardization of $S$ is $T$ if and only if $s_k \le s_{k+1}$ for all $k$, and $s_k < s_{k+1}$ when $k$ is a descent of $T$.

Thus, the set $S_T$ of semistandard Young tableaux with standardization $T$ is given by taking the parts defining semistandard Young tableaux in \autoref{thm:semistandard-tableau-complex} and adding $\inv T(k) < \inv T(k+1)$ to $B$ and 
\(
\setbuilder{(\inv T(k), \inv T(k+1))}{\text{for all descents $k$ of $T$}}
\)
to $\Psi$. 
Thus, by \autoref{thm:tableau-complex-ball-or-sphere}, $\Delta(\sS_T)$ is a ball or a sphere.
\end{proof}
The condition that $s_k \le s_{k+1}$ and $s_k < s_{k+1}$ if $k$ is a decent of $T$ are exactly the conditions that a composition tableau of shape $\comp(\Des(T))$ satisfy, so
\(
\quasi_{\comp(\Des(T))} = \sum x^S,
\)
where the sum is over all semistandard Young tableaux $S$ with standardization $T$. Thus, we get the following (cf. \autoref{thm:schubert-slide-ball-decomposition}).
\begin{theorem}\label{thm:schur-splits-into-balls}
  The tableau complex $\Delta(\SSYT_n(\lambda))$ can be decomposed into balls or spheres, with the balls or spheres corresponding to the fundamental quasisymmetric polynomials appearing in the expansion 
\[
\schur_\lambda = \sum_{T \in \SYT(\lambda)}\quasi_{\comp(\Des(T))}.
\]
\end{theorem}

There is a $K$-theoretic version of the above. With 
\[
\tilde{F}_\mu = \sum_{T}(-1)^{\size{T} - \size{\lambda}}x^T,
\] 
where the sum is over set-valued composition tableau, and with $\tilde{G}_\lambda$ the stable Grothendieck polynomials, the following generalization of \autoref{eq:schur-fundamental-quasisymmetric} appears in \cite{Multifundamental-paper}.
\[
  \tilde{G}_\lambda = \sum_\mu (-1)^{\size{\mu} - \size{\lambda}}\tilde{F}_\mu.
\]
As above, $\mu$ corresponds to a ball (or sphere) of (set-valued) semistandard Young tableau of shape $\lambda$ with a given (set-valued) standardization. 
But we have not found a non-tautological description in terms of standard Young tableau of which compositions $\mu$ appear (in \cite{Multifundamental-paper}, the description is given in terms of ``multi-Jordan-Holder sets'').

\bibliography{mybib}

\newcommand{\etalchar}[1]{$^{#1}$}
\begin{thebibliography}{HLMvW11}

\bibitem[AS17]{Slide-paper}
Sami Assaf and Dominic Searles.
\newblock Schubert polynomials, slide polynomials, {S}tanley symmetric
  functions and quasi-{Y}amanouchi pipe dreams.
\newblock {\em Adv. Math.}, 306:89--122, 2017.

\bibitem[BB93]{RC-graphs-and-schubert-polynomials}
Nantel Bergeron and Sara Billey.
\newblock {RC}-graphs and {S}chubert polynomials.
\newblock {\em Exp. Math.}, 2(4):257--269, 1993.

\bibitem[BHY19]{Billey-Holroyd-Young-Little-bump-paper}
Sara~C. Billey, Alexander~E. Holroyd, and Benjamin~J. Young.
\newblock A bijective proof of {Macdonald{\textquoteright}s} reduced word
  formula.
\newblock {\em Algebraic Combinatorics}, 2(2):217--248, 2019.

\bibitem[BJS93]{Billey-Jockusch-Stanley}
Sara~C. Billey, William Jockusch, and Richard~P. Stanley.
\newblock Some combinatorial properties of {S}chubert polynomials.
\newblock {\em J. Algebraic Comb.}, 2(4):345--374, 11 1993.

\bibitem[BLVS{\etalchar{+}}99]{Oriented-matroids}
Anders Björner, Michel Las~Vergnas, Bernd Sturmfels, Neil White, and
  Günter~M. Ziegler.
\newblock {\em Oriented Matroids}.
\newblock Encyclopedia of Mathematics and its Applications. Cambridge
  University Press, 2 edition, 1999.

\bibitem[BP79]{Vertex-decomposability-paper}
Louis~J. Billera and J.~Scott Provan.
\newblock A decomposition property for simplicial complexes and its relation to
  diameters and shellings.
\newblock {\em Ann. N. Y. Acad. Sci.}, 319(1):82--85, 05 1979.

\bibitem[Eis95]{Eisenbud}
David Eisenbud.
\newblock {\em Commutative Algebra with a View Toward Algebraic Geometry},
  volume 150 of {\em Graduate Texts in Mathematics}.
\newblock Springer-Verlag, 1995.

\bibitem[FS94]{Fomin-Stanley}
S.~Fomin and R.~P. Stanley.
\newblock Schubert polynomials and the nil{C}oxeter algebra.
\newblock {\em Adv. Math.}, 103(2):196--207, 1994.

\bibitem[Ful92]{Fulton-92}
William Fulton.
\newblock Flags, {S}chubert polynomials, degeneracy loci, and determinantal
  formulas.
\newblock {\em Duke Math. J.}, 65(3):381--420, 1992.

\bibitem[Ful96]{Fulton-young-tableaux-book}
William Fulton.
\newblock {\em Young Tableaux: With Applications to Representation Theory and
  Geometry}.
\newblock London Mathematical Society Student Texts. Cambridge University
  Press, 1996.

\bibitem[HLMvW11]{Other-composition-tableaux-paper}
J.~Haglund, K.~Luoto, S.~Mason, and S.~{v}an Willigenburg.
\newblock Quasisymmetric {S}chur functions.
\newblock {\em J. Comb. Theory Ser. A}, 118(2):463--490, 2011.

\bibitem[KK02]{Kogan-Kumar}
Mikhail Kogan and Abhinav Kumar.
\newblock A proof of {P}ieri's formula using the generalized {S}chensted
  insertion algorithm for rc-graphs.
\newblock {\em Proc. Amer. Math. Soc.}, 130(9):2525--2534, 2002.

\bibitem[KM04]{Subword-complexes-paper}
Allen Knutson and Ezra Miller.
\newblock Subword complexes in {C}oxeter groups.
\newblock {\em Adv. Math.}, 184(1):161--176, 2004.

\bibitem[KM05]{Grobner-geometry-paper}
Allen Knutson and Ezra Miller.
\newblock Gr\"{o}bner geometry of {S}chubert polynomials.
\newblock {\em Ann. of Math.}, 161(3):1245--1318, 2005.

\bibitem[KMY08]{Tableau-complexes-paper}
Allen Knutson, Ezra Miller, and Alexander Yong.
\newblock Tableau complexes.
\newblock {\em Isr. J. Math.}, 163:317--343, 2008.

\bibitem[Lit03]{Little-bump-paper}
David~P. Little.
\newblock Combinatorial aspects of the {L}ascoux–{S}chützenberger tree.
\newblock {\em Adv. Math.}, 174(2):236--253, 2003.

\bibitem[LLS21]{Back-stable-paper}
Thomas Lam, Seung~Jin Lee, and Mark Shimozono.
\newblock Back stable {S}chubert calculus.
\newblock {\em Compos. Math.}, 157(5):883–962, 2021.

\bibitem[LP07]{Multifundamental-paper}
Thomas Lam and Pavlo Pylyavskyy.
\newblock Combinatorial {H}opf algebras and {K}-homology of {G}rassmanians.
\newblock {\em Int. Math. Res. Not.}, 2007, 2007.

\bibitem[LS82]{Lascoux-Schutzenberger}
Alain Lascoux and Marcel-Paul Schützenberger.
\newblock Polynômes de {S}chubert.
\newblock {\em C. R. Acad. Sci. Paris Sér. I Math.}, 294(13):447–450, 1982.

\bibitem[Mon59]{Monk}
D.~Monk.
\newblock The geometry of flag manifolds.
\newblock {\em Proc. London Math. Soc.}, s3-9(2):253--286, 1959.

\bibitem[MS05]{Miller-Sturmfels}
Ezra Miller and Bernd Sturmfels.
\newblock {\em Combinatorial Commutative Algebra}, volume 227 of {\em Graduate
  Texts in Mathematics}.
\newblock Springer, 2005.

\bibitem[Nen20]{Nenashev}
Gleb Nenashev.
\newblock Differential operators on {S}chur and {S}chubert polynomials.
\newblock 2020.
\newblock arXiv:2005.08329v2.

\bibitem[NT24]{Forest-polynomial-paper}
Philippe Nadeau and Vasu Tewari.
\newblock Forest polynomials and the class of the permutahedral variety.
\newblock {\em Adv. Math.}, 453:109834, 2024.

\bibitem[PS17]{Glide-paper}
Oliver Pechenik and Dominic Searles.
\newblock Decompositions of {G}rothendieck polynomials.
\newblock {\em Int. Math. Res. Not.}, 2019(10):3214--3241, 2017.

\bibitem[Sea20]{Skyline-slides-paper}
Dominic Searles.
\newblock Polynomial bases: {P}ositivity and {S}chur multiplication.
\newblock {\em Trans. Am. Math. Soc.}, 373(2):819--847, 2020.

\bibitem[Sot96]{Sottile-Pieri-rule-paper}
Frank Sottile.
\newblock Pieri's formula for flag manifolds and {S}chubert polynomials.
\newblock {\em Ann. Inst. Fourier}, 46(1):89--110, 1996.

\bibitem[ST21]{Smirnov-Tutubalina}
E.~Yu. Smirnov and A.~A. Tutubalina.
\newblock Slide polynomials and subword complexes.
\newblock {\em Sbornik Math.}, 212(10):1471--1490, 2021.

\bibitem[Sta99]{Stanley-2}
Richard~P. Stanley.
\newblock {\em Enumerative combinatorics. {V}ol. 2}, volume~62 of {\em
  Cambridge Studies in Advanced Mathematics}.
\newblock Cambridge University Press, 1999.

\bibitem[Stu96]{Sturmfels-Grobner-bases-book}
Bernd Sturmfels.
\newblock {\em Gr\"{o}bner bases and convex polytopes}.
\newblock Number~8 in University Lectures Series. American Mathematical
  Society, 1996.

\bibitem[WY23]{Woo-Yong-Survey-paper}
Alexander Woo and Alexander Yong.
\newblock Schubert geometry and combinatorics.
\newblock 2023.
\newblock arXiv:2303.01436v1.

\end{thebibliography}
\end{document}